\documentclass[11pt]{amsart}
\usepackage{amsmath, amsfonts, amssymb, color, enumitem}

\usepackage{graphicx}
\usepackage{geometry}
\usepackage[utf8]{inputenc}
\usepackage[T1]{fontenc}
\usepackage{xcolor}
\usepackage{hyperref}

\usepackage[normalem]{ulem}

\usepackage{tikz}
\usetikzlibrary{arrows}
\usetikzlibrary{patterns}

\usepackage{dsfont}

\newtheorem{theorem}{Theorem}[section]
\newtheorem{lemma}[theorem]{Lemma}
\newtheorem{proposition}[theorem]{Proposition}
\newtheorem{corollary}[theorem]{Corollary}
\newtheorem{definition}[theorem]{Definition}
\newtheorem{remark}[theorem]{Remark}
\newtheorem{example}[theorem]{Example}

\numberwithin{equation}{section}

\newcommand{\de}{\, \mathrm{d}}

\newcommand{\shadow}[2]{\mathcal{S}^{#1}(#2)}
\newcommand{\leqcp}{\leq_{c,+}}
\newcommand{\leqcs}{\leq_{c,s}}
\newcommand{\leqc}{\leq_{c}}
\newcommand{\geqc}{\geq_{c}}
\newcommand{\leqp}{\leq_{+}}

\newcommand{\lc}{\text{lc}}
\newcommand{\mc}{\text{mc}}
\newcommand{\sun}{\text{sun}}
\newcommand{\sowh}{\in \cdot \,}
\newcommand{\MM}{\mathsf{M}}
\newcommand{\MMC}{\mathsf{M}^{\text{rc}}}
\newcommand{\PP}{\mathsf{P}}
\newcommand{\PPC}{\mathsf{P}^{\text{rc}}}
\newcommand{\MZ}{\mathcal{M}_0}
\newcommand{\MO}{\mathcal{M}_1}
\newcommand{\PZ}{\mathcal{P}_0}
\newcommand{\PO}{\mathcal{P}_1}
\newcommand{\TZ}{\mathcal{T}_0}
\newcommand{\TO}{\mathcal{T}_1}
\newcommand{\1}{\mathds{1}}
\newcommand{\tmax}{t^{*}}

\newcommand{\supp}{\mathrm{supp}}
\newcommand{\R}{\mathbb{R}}
\renewcommand{\P}{\mathbb{P}}

\begin{document}
	\title[Shadow Martingales]{Shadow martingales -- a stochastic mass transport approach to the peacock problem}
	\author{Martin Br\"uckerhoff \address[Martin Br\"uckerhoff]{Universit\"at M\"unster, Germany} \email{martin.brueckerhoff@devk.de} \hspace*{0.5cm} Martin Huesmann \address[Martin Huesmann]{Universit\"at M\"unster, Germany} \email{martin.huesmann@uni-muenster.de} \\ \\ Nicolas Juillet \address[Nicolas Juillet]{Universit\'{e} de Strasbourg, France} \email{nicolas.juillet@uha.fr}} 
	\thanks{MB and MH have been partially supported by the Vienna Science and Technology Fund (WWTF) through project VRG17-005; they are funded by the Deutsche Forschungsgemeinschaft (DFG, German Research Foundation) under Germany's Excellence Strategy EXC 2044 –390685587, Mathematics Münster: Dynamics–Geometry–Structure. NJ thanks the Erwin-Schr\"odinger Institue (ESI) for supporting his stay in Vienna in June 2019.}
	
	\begin{abstract}
		Given a family of real probability measures $(\mu_t)_{t\geq 0}$ increasing in convex order (a  peacock) we describe a systematic method to create a martingale exactly fitting the marginals at any time. The key object for our approach is the obstructed shadow of a measure in a peacock, a generalization of the (obstructed) shadow introduced in \cite{BeJu16,NuStTa17}. As input data we take an increasing family of measures $(\nu^\alpha)_{\alpha \in [0,1]}$ with $\nu^\alpha(\R)=\alpha$ that are submeasures of $\mu _0$, called a  parametrization of $\mu_0$. Then, for any $\alpha$ we define an evolution $(\eta^\alpha_t)_{t\geq 0}$ of the measure $\nu^\alpha=\eta^\alpha_0$ across our peacock by setting $\eta^\alpha_t$ equal to the obstructed shadow of $\nu^\alpha$ in $(\mu _s)_{s \in [0,t]}$. We identify conditions on the parametrization $(\nu^\alpha)_{\alpha\in [0,1]}$ such that this construction leads to a unique martingale measure $\pi$, the shadow martingale, without any assumptions on the peacock.  In the case of the left-curtain parametrization $(\nu_{\lc}^\alpha)_{\alpha \in [0,1]}$ we identify the shadow martingale as the unique solution to a continuous-time version of the martingale optimal transport problem.
		
		Furthermore, our method enriches the knowledge on the Predictable Representation Property (PRP) since any shadow martingale comes with a canonical Choquet representation in extremal Markov martingales.
		\smallskip
		
		

\end{abstract}
	
\date{\today}
\maketitle

\section{Introduction}

Two finite measures $\mu$ and $\mu'$ on $\mathbb{R}$ with finite first moments are said to be in convex order, denoted by $\mu \leqc \mu'$, if $\int \varphi \de \mu \leq \int \varphi \de \mu'$ for all convex $\varphi: \mathbb{R} \rightarrow \R$.
Peacocks are families $(\mu_t)_{t \geq 0}$ of probability measures  on $\mathbb{R}$ with finite first moments that increase in convex order. 
Given a peacock $(\mu _t)_{t \geq 0}$, the peacock problem is to construct a probability measure $\pi$ such that the canonical process $X=(X_t)_{t \geq 0}$ is a martingale w.r.t.\ its natural filtration and the marginal distributions coincide with $(\mu _t)_{t \geq 0}$, i.e.\ $\mathsf{Law}_\pi(X_t)=\mu_t$ for each $t \geq 0$.

There is a wide range of beautiful solutions to this problem employing different ideas and techniques, e.g.\ \cite{Ke72, Lo08, BoPrRo12, KaTaTo17, Ho17, MaYo02, Ju16, HaKl07, Al08, Ju18, HeTaTo16}. 
On the one extreme, there is the fundamental non-constructive result of Kellerer \cite{Ke72} proving the existence of Markov solutions for any given peacock.
On the other end of the spectrum, there are very explicit constructions for specific sub classes of peacocks, many of which can be found in the monograph \cite{HiPrRoYo11} by Hirsch, Profeta, Roynette, and Yor.
However, it is difficult to manage both aspects by constructing an explicit solution for a generic peacock.
Only recently there have been  contributions in this direction by Lowther \cite{Lo08}, Hobson \cite{Ho17}, Juillet \cite{Ju18} and Henry-Labord\`ere and Touzi \cite{HeTo16}.

We propose a new method to systematically construct a martingale associated with
a peacock. Thereby, we rely on the rich theory of optimal transport. In optimal
transport a coupling of two probability measures is interpreted as a plan to transport
one marginal to the other one. More precisely, given a coupling $\pi$ of two probability
measures $\mu_0$ and $\mu_1$ on $\R$, i.e. a probability measure $\pi$ on $\R^2$ with $\pi(A\times \R)=\mu_0(A)$
and $\pi(\R\times B)=\mu_1(B)$ for all Borel sets $A$ and $B$. The quantity $\pi(A \times B)$ can be
interpreted as the amount of mass (of the measure $\mu_0$) that is transported from the set
$A$ to the set $B$ under $\pi$. Conversely, a coupling $\pi$ is fully characterized by the family
of values $(\pi(A \times B))_{A,B}$ and this characterization still holds if we only consider certain
families of sets, e.g.\ only sets $A \times B$ with $A$ of the type $(-\infty, q]$. Note, that given $q$
the values of $B \mapsto \pi((-\infty, q] \times B)$ are encoded in the second marginal of $\pi_{|(-\infty,q]\times\R}$ ,
i.e.\ in the measure $\eta_q := \pi((-\infty, q] \times\cdot )$. Therefore, the family $(\eta_q)_{q\in\R}$ associated
with $((\mu_0)_{|(-\infty,q]} )_{q\in\R}$, i.e.\ the one-step evolution $(\mu_0((-\infty,q]),\eta_q)$,
completely determines the transport plan $\pi$.

In recent years, this mindset of optimal transportation found several new applications within stochastic analysis sometimes subsumed under the name stochastic mass transport, see e.g.\ \cite{BeJu21, BeCoHu17, GhKiPa21, KlKu15, BeEdElSc18}. In various striking applications it turned out to be useful to interpret a stochastic process $X\equiv (X_t)_{t\geq 0}$ as a device to transport mass from time $0$ to the distribution of  $X$ at a (potentially random) time $\tau$. To identify the induced coupling of the distribution of $X$ at time $0$ and time $\tau$ it is then necessary to trace the evolution of fixed parts of the initial distribution, e.g.\ to consider for $A\subset \R$ the evolution in $t$ of
$$ \eta_t^A:=\P[X_t\in \cdot |X_0 \in A].$$

However, observe, that already for two step processes $(X_0,X_1,X_2)$ (corresponding to three marginal transport problems) the knowledge of only $(\eta_t^A)_{t=1,2}$ is in general not sufficient to pin down the law of the full process since it neglects correlations between $X_1$ and $X_2$.

We call a measure $\nu$ a submeasure of $\mu$, if $\nu \leqp \mu$, i.e.\ if \label{eq:Submeasures}
$\nu(A) \leq \mu(A)$  for all measurable sets $A$.
For instance, the restrictions $(\mu_0)_{|A}$ of $\mu_0=\mathsf{Law}(X_0)$ to the measurable sets $A=(-\infty,q]$ are submeasures of $\mu_0$. Using this terminology, the goal of this article is to uniquely define a martingale associated with a peacock $(\mu _t)_{t \geq 0}$ from the following input data only:

\begin{itemize}\label{p:input}
	\item A \textit{parametrization} of $\mu_0$, i.e.\ a family of submeasures $(\nu^{\alpha})_{\alpha \in [0,1]}$ of $\mu_0$ s.t.\ $\nu^{\alpha}(\mathbb{R})=\alpha$, $\nu^{\alpha} \leq_+ \nu^{\beta}$  for $\alpha \leq \beta$, and $\nu^1=\mu_0$.
	\item For each $\alpha$, the \textit{evolution} of $\nu ^{\alpha}$ through the marginals $(\mu _t)_{t \geq 0}$, i.e.\ a family $(\eta _t ^{\alpha})_{t \geq 0}$ of submeasures of $(\mu _t)_{t \geq 0}$, $\eta  ^{\alpha} _t \leqp \mu _t$ for all $t \geq 0$, satisfying $\nu ^{\alpha} = \eta _0 ^{\alpha} \leqc \eta _s ^{\alpha} \leqc \eta _t ^{\alpha}$ for all $0 \leq s \leq t$. These evolutions also need to be consistent in the sense that $\eta^\alpha_t\leqp \eta^\beta_t$ for all $\alpha \leq \beta$ in $[0,1]$ and $t \geq 0$.
\end{itemize}

It is easy to see that, without further assumptions, this data is not sufficient to uniquely determine the law of a martingale. It turns out that a certain convexity of $(\nu^{\alpha})_{\alpha \in [0,1]}$  together with some kind of minimality in the choice of $(\eta^{\alpha})_{\alpha \in [0,1]}$ is the key to 
uniquely define a martingale measure via this procedure.

Before introducing the appropriate notions we would like to present our solution in a special setting 
which already gives a good idea of the general case (namely Theorem \ref{thm:GenExist} in Subsection \ref{ss:main_results}):

\begin{corollary}\label{thm:intro1}
	Let $(\mu _t)_{t \geq 0}$ be a peacock with $\mu(\{x\}) = 0$ for all $x \in \mathbb{R}$. For any nested family of intervals $(I_{\alpha})_{\alpha \in [0,1]}$ for which
	\begin{enumerate}
		\item [(i)] $\mu _0(I_{\alpha}) = \alpha$ for any $\alpha \in [0,1]$,
		\item [(ii)] $\alpha \mapsto  \int _{I_{\alpha}} y \de \mu _0(y)$ is a convex function and
		\item [(iii)] $\sup I_{\alpha} < + \infty$ and $\partial I_{\alpha} \cap \partial I_{\beta} = \emptyset$ for all $\alpha \neq \beta$ in $[0,1]$,
	\end{enumerate}
	there exists a unique solution $\pi$ to the peacock problem w.r.t.\ $(\mu _t)_{t \geq 0}$ such that for any other solution $\rho$ to the peacock problem w.r.t.\ $(\mu _t)_{t \geq 0}$ it holds
	\begin{equation}\label{eq:ThmIntro1}
	\quad\mathsf{Law}_{\pi} (X_t | X_0 \in I_{\alpha}) \leqc \mathsf{Law}_{\rho} (X_t | X_0 \in I_{\alpha})
	\end{equation}
	for all $\alpha \in [0,1]$ and $t \geq 0$.
	Moreover, $(X_0,X_t)_{t \geq 0}$ is a Markov process under $\pi$.
\end{corollary}

\begin{remark}
	As the reader should have noticed, a completely rigorous statement of Corollary \ref{thm:intro1} requires to specify on which measurable space the martingale measure $\pi$ in Corollary \ref{thm:intro1} is defined. Here, as well as in all of the paper, we use $\mathbb{R}^{[0,\infty)}$ with the Borel $\sigma$-algebra induced by the product topology on $\mathbb{R}^{[0,\infty)}$. 
	
	Alternatively, if the map $t \mapsto \mu _t$ is right-continuous w.r.t.\ the weak topology (cf.\ Section \ref{sec:Notation}), by standard martingale regularization (see e.g.\ \cite[II \S 2]{ReYo99}) there exists a c\`adl\`ag modification of the canonical process on $\mathbb{R}^{[0,\infty)}$ under $\pi$ that one can use to define the solution $\pi$ directly on the Skorokhod space of c\`adl\`ag functions from $[0,\infty)$ to $\mathbb{R}$, see Section \ref{sec:ContTimeRC} for an implementation. 
\end{remark}

Consider  the peacock $(\mu _t)_{t \geq 0}$ consisting of uniform distributions $\mu _t = \mathrm{Unif}_{[-1-t,1+t]}$ on the intervals $[-1-t,1+t]$ and the interval family $I_\alpha=[- \alpha, \alpha]$, $\alpha\in [0,1]$. It is not difficult to check that this pair satisfies the conditions (i)-(iii) in Corollary \ref{thm:intro1} and for two choices of $\alpha$, Figure \ref{fig:ExplMC} illustrates the evolution $(\eta^\alpha_t)_{t \geq 0}$ of $\nu ^{\alpha} = (\mu _0)_{|I_{\alpha}}$ over time under the solution to the peacock problem constructed in Corollary \ref{thm:intro1}.

We would like to highlight a few features of Corollary \ref{thm:intro1} which all appear in the general case, Theorem \ref{thm:GenExist}, again:

\begin{itemize}
	\item The parametrization of $\mu _0$ induced by the intervals $I_{\alpha}$, namely $({\mu_0}_{|I_{\alpha}})_{\alpha \in [0,1]}$, is in a certain sense convex, cf.\ item (ii) in Corollary \ref{thm:intro1}.
	\item The minimality condition \eqref{eq:ThmIntro1} affects only the conditional one-dimensional marginal distributions under $\pi$. Thus, only the evolution $\eta_t ^{\alpha} = \alpha \mathrm{Law}_{\pi}(X_t|X_0 \in I_{\alpha})$ of $\nu ^{\alpha} = (\mu _0)_{|I_{\alpha}}$ is prescribed by the requirement to be minimal in convex order but (a priori) no joint distributions are fixed.
	\item In particular, \eqref{eq:ThmIntro1} says that $\mathsf{Law}_{\pi} (X_t | X_0 \in I_{\alpha}) $ is minimal in convex order among all solutions to the peacock problem, for every $\alpha$. Explicitly, for every $t$, every $\alpha$, every convex function $\varphi$, and any other solution $\rho$ to the peacock problem w.r.t.\ $(\mu_t)_{t \geq 0}$ we have
	$$ \int \varphi\ \de \mathsf{Law}_{\pi} (X_t | X_0 \in I_{\alpha})\leq  \int \varphi\ \de \mathsf{Law}_{\rho} (X_t | X_0 \in I_{\alpha}),$$
	so that, in this precise sense, $\mathsf{Law}_{\pi} (X_t | X_0 \in I_{\alpha})$ is as concentrated as possible. Hence, we can think of $\pi$ as a plan to transport ${\mu_0}_{|I_{\alpha}}$  through $(\mu_t)_{t \in [0,1]}$ as concentrated as possible subject to the martingale constraint.
	\item  Finally, it will become apparent during the proof of Theorem \ref{thm:GenExist} that the Markov property turns out to be a consequence of the fact that $\mathrm{Law}(X|X_0)$ is uniquely determined by its marginal distributions, see Lemma \ref{lemma:NSItoMarkov}, Proposition \ref{prop:NSItoUniq}.
\end{itemize}

\begin{center} \label{fig:ExplMC}
	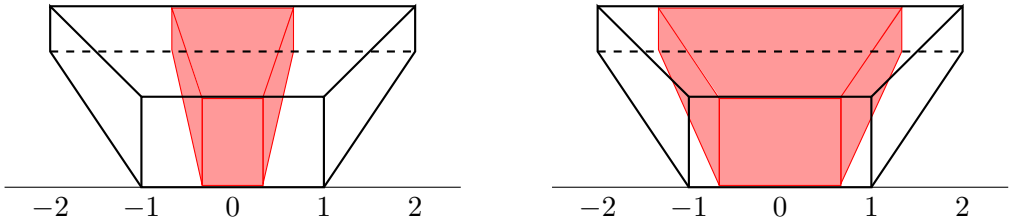
\begin{figure}
		\begin{tikzpicture}[scale=1.2]
		\draw[-] (-2.5,0) -- (2.5,0);
		\draw[-] (3.5,0) -- (8.5,0);
		
		\filldraw[red,fill opacity=0.4] (-1/3,0.02) rectangle (1/3,0.98);
		\filldraw[red,fill opacity=0.4] (-1/3,0.02) -- (-2/3,1.52) -- (-2/3,1.98) -- (-1/3,0.98) -- (-1/3,0.02);
		\filldraw[red,fill opacity=0.4]  (-2/3,1.98) -- (2/3,1.98) -- (1/3,1) -- (-1/3,1);
		\filldraw[red,fill opacity=0.4]  (2/3,2) -- (2/3,1.5) -- (1/3,0.04) -- (1/3,1);
		\draw[thick] (-1,0) rectangle (1,1);
		\draw[thick] (-1,0) -- (-2,1.5) -- (-2,2) --(-1,1);
		\draw[thick] (1,0) -- (2,1.5) -- (2,2) --(1,1);
		\draw[thick] (-2,2) -- (2,2);
		\draw[thick,dashed] (-2,1.5) -- (2,1.5);
		\node[below] at (-2,0) {$-2$};
		\node[below] at (-1,0) {$-1$};
		\node[below] at (0,0) {$0$};
		\node[below] at (1,0) {$1$};
		\node[below] at (2,0) {$2$};
		
		\filldraw[red,fill opacity=0.4] (-1/3*2+6,0.02) rectangle (1/3*2+6,0.98);
		\filldraw[red,fill opacity=0.4] (-1/3*2+6,0.02) -- 
		(-2/3*2+6,1.52) -- (-2/3*2+6,1.98) -- (-1/3*2+6,0.98) -- (-1/3*2+6,0.02);
		\filldraw[red,fill opacity=0.4]  (-2/3*2+6,1.98) -- (2/3*2+6,1.98) -- (1/3*2+6,1) -- (-1/3*2+6,1);
		\filldraw[red,fill opacity=0.4]  (2/3*2+6,2) -- (2/3*2+6,1.5) -- (1/3*2+6,0.04) -- (1/3*2+6,1);
		\draw[thick] (-1+6,0) rectangle (1+6,1);
		\draw[thick] (-1+6,0) -- (-2+6,1.5) -- (-2+6,2) --(-1+6,1);
		\draw[thick] (1+6,0) -- (2+6,1.5) -- (2+6,2) --(1+6,1);
		\draw[thick] (-2+6,2) -- (2+6,2);
		\draw[thick,dashed] (-2+6,1.5) -- (2+6,1.5);
		\node[below] at (-2+6,0) {$-2$};
		\node[below] at (-1+6,0) {$-1$};
		\node[below] at (0+6,0) {$0$};
		\node[below] at (1+6,0) {$1$};
		\node[below] at (2+6,0) {$2$};
		\end{tikzpicture}
		
		\caption{The shaded area shows the evolution $(\eta _t ^{\alpha})_{t \geq 0}$ of $\nu ^{\alpha}$ through $(\mu _t)_{t \geq 0}$ for $\alpha = \frac{1}{3}$ (left) and  $\alpha = \frac{2}{3}$ (right), respectively. Here $\mu _t = \mathrm{Unif}_{[-1-t,1+t]}$ and $I_{\alpha} = [-\alpha,+\alpha]$. The measures are represented by their density functions w.r.t.\ the Lebesgue measure. The representation is in 3D-perspective with times evolving transversally to the page.}   
	\end{figure}
\end{center}

\subsection{Main results}\label{ss:main_results}

Our main results, Theorem \ref{thm:GenExist} and \ref{thm:LMAtlOptimal}, enlarge the perspective presented in Corollary \ref{thm:intro1} but are of the same nature. They in fact permit further parametrizations of $\mu_0$ and stress the optimal feature of our shadow martingales.

To state Theorem \ref{thm:GenExist} we need to introduce the objects that will replace the specific parametrization $({\mu_0}_{|I_\alpha})_{\alpha\in [0,1]}$ and property \eqref{eq:ThmIntro1}. We start with the definition of shadows, the concept that will replace \eqref{eq:ThmIntro1}. 
To not overload the introduction we give a preliminary (but correct) definition in Proposition \ref{def:shadow} and refer to Proposition \ref{prop:GeneralSchadow}  where it is extended to and proved in a more general setting.

As before a martingale measure is a probability measure under which the canonical process is a martingale w.r.t.\ the filtration generated by the process.

\begin{proposition}\label{def:shadow}
	For all peacocks $(\mu_t)_{t \geq 0}$, $t \geq 0$ and $\nu \leq_+ \mu _0$ with $\alpha = \nu (\mathbb{R}) > 0$ the set
	\begin{equation*}
	\left\{ \alpha \mathsf{Law}_{\pi}(X_t) : \begin{array}{l}
	\pi \text{ is a martingale measure, }   \nu=\alpha\mathsf{Law}_{\pi}(X_0) \\ \text{and } \alpha\mathsf{Law}_{\pi}(X_s) \leqp \mu _s \text{ for all } s \in [0,t] \end{array}	 \right\}
	\end{equation*}
	attains a minimum w.r.t.\ the convex order $\leqc$. This minimum is called the shadow of $\nu$ in $(\mu _s)_{s \in [0,t]}$ and is denoted by $\shadow{\mu _{[0,t]}}{\nu}$.
\end{proposition}

\label{def:convex}
We say that two finite measures $\mu$ and $\mu'$ on $\mathbb{R}$ with finite first moments are in convex-stochastic order, denoted by $\mu \leqcs \mu'$, if $\int \varphi \de \mu \leq \int \varphi \de \mu'$ for all convex and increasing functions $\varphi: \mathbb{R} \rightarrow \R$. A parametrization $(\nu^{\alpha})_{\alpha \in [0,1]}$ of $\mu_0$ is called $\leq  _{c,s}$-convex
if for all $\alpha_1 < \alpha_2  < \alpha_3 $ in $[0,1]$ it holds
\begin{equation} \label{eq:csConvexParam}
\frac{\nu ^{\alpha_2} - \nu ^{\alpha_1}}{\alpha_2 - \alpha_1} \leqcs  \frac{\nu ^{\alpha_3} - \nu ^{\alpha_2}}{\alpha_3 - \alpha_2}.
\end{equation}
Since property \eqref{eq:csConvexParam} can be interpreted as increasing slopes of secant lines for the functions $\alpha\mapsto \int \varphi \de \mu^\alpha$, $\varphi$ increasing and convex, this property is called $\leqcs$-convexity.
The following three parametrizations, that were introduced in \cite{BeJu21} for one step processes, are examples of $\leqcs$-convex parametrizations (cf.\ Lemma \ref{lem:para_convex} for the proof):
\begin{itemize}
	\item the left-curtain parametrization  \label{p:defLC}
	\begin{equation*}
	\nu _{\lc} ^{\alpha} = \mu _0|_{(- \infty, F_{\mu _0}^{-1}(\alpha))}  + (\alpha - \mu _0 (- \infty, F_{\mu _0}^{-1}(\alpha))) \delta _{F_{\mu _0} ^{-1}(\alpha)},
	\end{equation*}
	where $F^{-1}_{\mu_0}$ is the quantile function of $\mu_0$, i.e the generalized inverse of the cumulative distribution function $F_{\mu_0}$,
	
	\item the sunset parametrization $\nu _{\text{sun}} ^{\alpha} = \alpha  \mu _0$ for every $\alpha\in [0,1]$ and
	
	\item  the middle-curtain parametrization
	\begin{equation*}
	\nu _{\text{mc}} ^{\alpha} = \mu _{0| (q_{\alpha}, q_{\alpha}')}  + c_{\alpha} \delta _{q_{\alpha}} + c'_{\alpha} \delta _{q_{\alpha}'},
	\end{equation*}
	where $q_{\alpha} \leq q_{\alpha}'$ and $c_{\alpha}, c_{\alpha}' \in [0,1]$ are chosen such that $\nu_{\text{mc}}^{\alpha}(\mathbb{R}) = \alpha$ and $\int y \de 
	\nu_{\text{mc}}^{\alpha}(y) = \int y \de \mu _0(y)$.
\end{itemize}
We remark that if $\mu_0$ has no atoms the left-monotone and the middle-curtain parametrizations are special cases of the parametrization used in Corollary \ref{thm:intro1}. In particular, the parametrization in Figure \ref{fig:ExplMC} is the middle-curtain parametrization of the uniform measure on $[-1,1]$.

The final object that we need to introduce are martingale parametrizations. Their purpose is to allow conditioning on the initial behaviour of a martingale which is not of the form $\{X_0\in I_\alpha\}$ for some Borel set $I_{\alpha} \subset \mathbb{R}$. Again, as Definition \ref{def:shadow}, Definition \ref{def:mgpara} is a simplified version of the general Definition \ref{def:GenParametrization}. Notice that the notion of submeasure from page \pageref{eq:Submeasures}  is also well-defined for measurable spaces other than $\mathbb{R}$. Moreover, we denote by $\pi ^{\alpha}(X_t \sowh)$ the push-forward measure of $\pi ^{\alpha}$ via $X_t$ (it is a measure of mass $\alpha$). 

\begin{definition} \label{def:mgpara}
	Let  $\pi$ be the law of a martingale indexed by $[0,\infty)$. A family $({\pi}^{\alpha})_{\alpha \in [0,1]}$ of finite measures is called a martingale parametrization of  $\pi$  if
	\begin{enumerate}
		\item[(i)] for every $\alpha\in [0,1]$ one has $\pi^{\alpha}(\mathbb{R}^{[0,\infty)}) = \alpha$,
		\item[(ii)] we have $\pi^{\alpha} \leqp \pi ^{\alpha'}$ for all $\alpha \leq \alpha'$,
		\item[(iii)] for every $\alpha\in (0,1]$ the measure $\frac{\pi^{\alpha}}{\alpha}$ is a martingale measure, 
		\item[(iv)]  we have $\pi ^1 = \pi$.
	\end{enumerate}	
	The family $(\pi ^{\alpha})_{\alpha \in [0,1]}$ is called a martingale parametrization of $\pi$ w.r.t.\ a parametrization $(\nu^{\alpha})_{\alpha \in [0,1]}$ of $\mu _0$ if it additionally satisfies
	\begin{enumerate}
		\item [(v)] $\pi ^{\alpha}(X_0 \sowh) = \nu^{\alpha}$ for all $\alpha \in [0,1]$.
	\end{enumerate}

\end{definition}

As we discuss in Subsection \ref{sec:ParamProbMeas}, a martingale parametrization $(\pi ^{\alpha})_{\alpha \in [0,1]}$ w.r.t.\ $(\nu^{\alpha})_{\alpha \in [0,1]}$ is a convenient way of encoding that for each $\alpha \in [0,1]$ the martingale $\pi$ transports the submeasure $\nu^\alpha$ of $\mu _0$ according to $\pi^\alpha$, i.e.\ we may interpret $\pi ^{\alpha}$ formally as ``$\alpha \mathrm{Law}_{\pi}(X | X_0 \in \nu ^\alpha)$''. In particular, any martingale parametrization $(\pi ^{\alpha})_{\alpha \in [0,1]}$ w.r.t.\ $(\nu^{\alpha})_{\alpha \in [0,1]}$ induces a specific evolution of $\nu ^{\alpha}$ for every $\alpha\in [0,1]$, namely $\eta ^{\alpha}_t = \pi ^{\alpha}(X_t \sowh )$.
We would like to stress that  there might be several martingale parametrizations of $\pi$ w.r.t.\ $(\nu^\alpha)_{\alpha\in[0,1]}$ (cf.\ Example \ref{ex:two_para}).
We can now state our first main result:

\begin{theorem} \label{thm:GenExist}
	Let $(\mu _t)_{t \geq 0}$ be a peacock and $(\nu ^{\alpha})_{\alpha \in [0,1]}$  a $\leq_{c,s}$-convex parametrization of $\mu _0$. Then, there exists a unique pair $(\pi,(\pi ^{\alpha})_{\alpha \in [0,1]})$ where the martingale measure $\pi$ solves the peacock problem w.r.t.\ $(\mu _t)_{t \geq 0}$, $(\pi ^{\alpha})_{\alpha \in [0,1]}$ is a martingale parametrization of $\pi$ w.r.t.\ $(\nu ^{\alpha})_{\alpha \in [0,1]}$ and
	\begin{equation} \label{eq:ShadowCurveintro}
	\pi^{\alpha} (X_t \sowh) = \shadow {\mu _{[0,t]}} {\nu ^{\alpha}}
	\end{equation}
	for all $\alpha \in [0,1]$ and $t \geq 0$. We call $\pi$ the shadow martingale (measure) w.r.t.\ $(\mu _t)_{t \geq 0}$ and $(\nu^{\alpha})_{\alpha \in [0,1]}$. 
	
	The shadow martingale $\pi$ can be represented as $\pi = \mathrm{Law}\left(M^U\right)$ where
	$U$ is a $[0,1]$-valued random variable and $(M^a)_{a \in [0,1]}$ is a family of $\mathbb{R}^{[0, \infty)}$-valued random variables 
	such that
	\begin{enumerate}
		\item [(i)] $\{U\} \cup \{M^a : a \in [0,1]\}$ is a collection of independent random variables,
		\item [(ii)] the random variable $U$ is uniformly distributed on $[0,1]$ with $\mathrm{Law}(M^U _0| U \leq \alpha) =  \frac{1}{\alpha}\nu ^{\alpha}$ for all $\alpha \in (0,1]$ and
		\item [(iii)] for each $a \in [0,1]$, $(M^a_t)_{t \geq 0}$ is a Markov martingale which is uniquely determined by its (one-dimensional) marginal distributions, i.e.\ any martingale $(Y_t)_{t \geq 0}$ with $\mathrm{Law}(Y_t) = \mathrm{Law}(M_t ^a)$ for all $t \geq 0$ satisfies $\mathrm{Law}(Y) = \mathrm{Law}(M^a)$.
	\end{enumerate}
\end{theorem}

Note that the constraint on  $\pi$ given by \eqref{eq:ShadowCurveintro} only involves the evolution $(\pi ^{\alpha}(X_t \sowh))_{t \geq 0}$ of the submeasures $\nu ^{\alpha}$ and a priori no joint distributions.
Hence,  the theorem states that taking a $\leqcs$-convex parametrization of the initial marginal $\mu_0$ and fixing its evolution to be as concentrated as possible w.r.t.\ the convex order uniquely characterizes a martingale.

Specializing to the left-curtain parametrization $(\nu_{\lc}^\alpha)_{\alpha\in[0,1]}$ and additionally assuming that both $\mu_0$ has no atoms and $t \mapsto \mu _t$ is weakly right-continuous, we can give an alternative characterization of the associated shadow martingale which identifies it as a unique solution to a variant of an optimal transport problem, namely a peacock version of the martingale optimal transport problem:

\begin{theorem} \label{thm:LMAtlOptimal}
	Let $(\mu _t)_{t \geq 0}$ be a peacock  and $c$ a sufficiently integrable and regular cost function with $\partial_{xyy} c<0$ (cf. Theorem \ref{thm:GeneralOptimality}).
	The shadow martingale $\pi _{\lc}$ w.r.t.\ $(\mu _t)_{t \geq 0}$ and $(\nu^{\alpha}_{\lc})_{\alpha \in [0,1]}$ satisfies 
	\begin{equation}\label{eq:peacockSPM}
	\mathbb{E}_{\pi _{\lc}}[c(X_0,X_t)] = \inf \left\{ \mathbb{E}_{\rho}[c(X_0,X_t)] : \rho \text{ solves the PCOC problem w.r.t.} (\mu _t)_{t \geq 0} \right\}
	\end{equation}
	simultaneously for all $t \geq 0$.
	
	If  $\mu _0 (\{x\}) = 0$ for all $x \in \mathbb{R}$, $\pi_{\lc}$ is the only solution to the peacock problem w.r.t.\ $(\mu _t)_{t \geq 0}$ that satisfies \eqref{eq:peacockSPM} simultaneously for all $t \geq 0$.
\end{theorem}

The necessity of the assumption of no atoms is best seen by looking at the case $\mu_0=\delta_0$. In that case, since the marginals at time $t$ are given, each solution to the peacock problem w.r.t.\ $(\mu_t)_{t\in [0,1]}$ is a solution to the optimization problem \eqref{eq:peacockSPM}.

When considering only finitely many marginals $\{\mu_{t_0},\mu_{t_1},\ldots,\mu_{t_n}\}$ increasing in convex order, and a corresponding piecewise constant peacock, Theorem \ref{thm:LMAtlOptimal} reduces to a recent theorem by Nutz, Stebegg and Tan \cite[Theorem 7.16]{NuStTa17}, see also Corollary \ref{cor:thm83} where we give a short proof of \cite[Theorem 7.16]{NuStTa17} with the tools developed in this paper.

\begin{remark}  In this remark we discuss the necessity of the assumption of $\leqcs$-convexity for the existence and uniqueness part in Theorem \ref{thm:GenExist}.

The question of existence of a martingale measure satisfying \eqref{eq:ShadowCurveintro} for a \emph{general} parametrization is proven in Step 1 of Theorem \eqref{thm:MostGenExist} (for the easier case of a countable index set, see Proposition \ref{prop:CountableExistSM}). So for existence this assumption is not needed.

	The assumption of a $\leqcs$-convex parameterization is only needed for our proof of the \emph{uniqueness} statement of Theorem \ref{thm:GenExist}.  In Remark \ref{rem:ExplainAssumptionDetail} we explain how dropping this assumption would heavily affect our proof. However, it is worth highlighting that $\leqcs$ is not the only possible choice of order relation. In our proof, we only rely on the fact that the parametrization is $\preceq$-convex where $\preceq$ stands for any partial order relation that satisfies the following properties:
	\begin{itemize}
		\item  \emph{Compatibility} with the shadow: $\nu \preceq \nu'$ implies $\shadow{\mu_{[0,t]}}{\nu} \preceq \shadow{\mu_{[0,t]}}{\nu'}$. 
	
		\item \emph{Weaker} than $\leqc$: $\nu\leqc \nu'$ implies $\nu \preceq \nu'$.
			
		\item \emph{Representation:} There exists a convex function $\Phi$ such that $\nu \preceq \nu'$ implies $\int \Phi\ d\nu\leq \int \Phi\ d\nu'$ with equality if and only if $\nu=\nu'$. 
	\end{itemize}
	Our line of reasoning would work with any such partial order relation. These properties are satisfied by $\leqc$ and $\leqcs$. While $\leqc$-convex parameterizations already include the sunset and middle-curtain parameterization, it is required to consider $\leqcs$-parameterizations to include the (first introduced in the literature) left-curtain parametrization and therefore obtain the conclusions of Theorem \ref{thm:GenExist} (and first Theorem \ref{prop:GenExistCount}) for the corresponding shadow martingale. We can find a third partial order relation satisfying the properties above by a simple change of orientation. The order defined symmetrically to $\leqcs$ by replacing increasing convex by decreasing convex functions in \eqref{eq:OrderRelation}. This ordering, $\leqc$, and $\leqcs$ are the only three examples of ordering we know to satisfy the specifications listed above. For concreteness, we stick in this article with $\leqcs$.

	We would like to stress that the assumption of a $\leqcs$- (or $\preceq$-)convex parameterization is tightly linked with the infinite index set $T$. As we discuss in Section \ref{sec:contshadow}, for the case of a finite index set $T$ our methods can cope with any (not only $\leqcs$-convex) parametrization, see in particular Corollary \ref{cor:thm83}. Note that for infinite index sets $T$ and general parametrizations we did not furnish any counterexample in the conclusions of the main theorems. In fact we could not definitely exclude that Theorem \ref{thm:GenExist} could be improved in this direction --with different or better arguments as ours in the present paper.
\end{remark}

\subsection{Choquet representation and the PRP property}\label{sec:choquet}

There is another abstract point of view on our main result Theorem \ref{thm:GenExist}. Any martingale measure has a Choquet representation, i.e.\ it can be written as a superposition of martingale measures that are extremal elements of the convex set of all martingale measures. Such a representation is interesting because the extremality in the set of all martingale measures naturally relates to the predictable representation property (PRP). In stochastic analysis a martingale $M$ is said to satisfy the PRP if and only if any martingale $X$ adapted to the natural filtration of $M$ can be represented as a stochastic integral with respect to $M$. According to a theorem by Jacod and Yor a  martingale satisfies the PRP if and only if it's law is extremal in the convex set of all martingale measures  (cf.\ \cite{JaYo77,Yo78,Ja79}). Hence, any martingale measure is a superposition of martingales with the PRP. To the best of our knowledge, no concrete recipe for the construction of such a representation is known.
However, for shadow martingales there exists a natural Choquet representation. In fact, this natural Choquet representation is the driving force behind the proof of Theorem \ref{thm:GenExist}, especially the uniqueness part. 

Given a peacock $\mu = (\mu _t)_{t \geq 0}$, our construction of the uniquely determined shadow martingale starts with a representation of $\mu$ as a superposition of peacocks, i.e.\
\begin{equation}\label{eq:pcocrep}
\mu=\int_{[0,1]} \eta^a\ da.
\end{equation}
This representation is induced by the shadow and the choice of a proper parametrization $(\nu ^{\alpha})_{\alpha \in [0,1]}$ of $\mu_0$  (cf.\ Lemma \ref{lemma:RDofShadow}).
The peacocks $\eta^a$ in \eqref{eq:pcocrep} are in general not extremal in the convex set of all peacocks (in the sense that $2\eta^a=\eta'+\eta''$ implies $\eta'=\eta''=\eta^a$) so that \eqref{eq:pcocrep} cannot be called a Choquet representation of $\mu$. However, they satisfy a very similar property that we call \textit{non self-improvable} (NSI) (cf.\ \S \ref{sec:NSI}):
\begin{equation} \label{eq:ExtremalNSI}
2\eta^a=\eta'+\eta''\text{ and }\eta'_0=\eta''_0=\eta^a_0\quad \text{implies}\quad \eta'=\eta''=\eta^a.
\end{equation}
The main consequence of the NSI property is that for every peacock that satisfies this property there exists only one martingale measure that is associated with this peacock. Thus, the unique martingale measures $\pi ^a$ associated with $\eta ^a$ are extremal in the set of martingale measures with fixed initial distribution, i.e.\
\begin{align*}
\left\{ \begin{aligned}
&2\pi^a=\rho'+\rho\hspace{2cm}\text{and}\\
&\rho'(X_0 \sowh)= \rho''(X_0 \sowh) = \pi^a(X_0 \sowh)
\end{aligned}
\right.\quad\text{implies}\quad \rho'=\rho''=\pi^a \ .
\end{align*}
Indeed, if $\rho'$ and $\rho''$ are two martingale measures with $2\pi^a = \rho' + \rho ''$ that have the same initial distribution as $\pi ^a$, the marginal distributions of these three objects satisfy \eqref{eq:ExtremalNSI} and thus the marginal distributions of $\rho'$ and $\rho''$ are $\eta ^a$, i.e.\ they coincide with the marginal distributions of $\pi ^a$. But then the uniqueness of the associated martingale measure given by the NSI property yields ${\pi^a} = \rho = \rho'$. The superposition of these special martingale measures
\begin{equation} \label{eq:martrep}
\pi = \int _0 ^1 \pi ^a \de a
\end{equation}
is exactly the shadow martingale w.r.t.\ $\mu$ and $(\nu ^{\alpha})_{\alpha \in [0,1]}$. More precisely, $\pi ^a$ is the distribution of $M^a$ in the representation $\pi=\mathrm{Law}(M^U)$ of the shadow martingale that is described in the second part of Theorem \ref{thm:GenExist}.	

The representation of the shadow martingale in \eqref{eq:martrep} is in general not yet a Choquet representation because the martingale measures $\pi^a$ are only extremal in the set of martingale measures with fixed initial distribution. Nevertheless, we can directly obtain a Choquet representation from \eqref{eq:martrep} and then by Jacod and Yor's theorem we have a rather explicit representation of the shadow martingale as a superposition of martingales that satisfy the PRP. In the case of the left-curtain parametrization, this is particularly easy. The construction of \eqref{eq:pcocrep} is such that for each $\alpha\in[0,1]$ we have $\int_0^\alpha \eta_0^a da=\nu^\alpha_{\lc}$ where $(\nu ^{\alpha}_{\lc})_{\alpha \in [0,1]}$ is the left-curtain parametrization. Looking again at the definition of  $(\nu ^{\alpha}_{\lc})_{\alpha \in [0,1]}$ in Subsection \ref{p:defLC}, we see that $\eta^a_0$ is a Dirac measure for any $a\in [0,1]$. Hence, the peacocks $\eta^a$ and the associated martingale measures $\pi ^a$ are automatically extremal in the set of {\it all} peacocks and {\it all} martingale measures. Therefore, in the case of the left-curtain parametrization, \eqref{eq:martrep} is in fact already a Choquet representation of the shadow martingale $\pi$. More generally, given any $\leqcs$-convex parametrization $(\nu ^{\alpha})_{\alpha \in [0,1]}$ of $\mu_0$ we obtain a Choquet representation of the shadow martingale by  further disintegrating \eqref{eq:martrep} w.r.t.\ the initial marginal $\mu_0$, i.e.\ by conditioning on the starting value of $\pi$.

We want to emphasize that this Choquet representation of the shadow martingale is uniquely determined by the representation of the  peacock $\mu$ given in \eqref{eq:pcocrep}. This representation of $\mu$ is constructed using only the shadow and a $\leqcs$-convex parametrization of the initial distribution. To show that the peacocks $(\eta ^a)_{a \in [0,1]}$ in \eqref{eq:pcocrep} satisfy the NSI property (which is very similar to extremality, cf.\ \eqref{eq:ExtremalNSI}) is in fact a crucial part of our proof. Moreover, this abstract point of view of our result makes it apparent that the construction of the shadow martingale is purely based on its marginals as an object in the space of peacocks and thus these are intrinsic solutions to the peacock problem.

\subsection{Outline} \label{sec:MainIdeaProof}

There are several contributions to martingale optimal transport theory and the peacock problem that are related to our results and that we discuss in Section \ref{sec:RelatedLit}. In Section \ref{sec:Prelimin} we recall order relations for finite measures and important properties of the peacock problem. 

In Section \ref{sec:ImportantObjects} we introduce (martingale) parametrizations, (general obstructed) shadows and non self-improvable peacocks. These concepts are not only essential ingredients of our proof of Theorem \ref{thm:GenExist} but are interesting in themselves. 

In Section \ref{sec:GenExistCountable}, we prove a variant of Theorem \ref{thm:GenExist} in the case that the peacock is indexed by a countable set $S \subset [0,\infty)$ that contains $0$ and satisfies $\sup S \in S$. In this setup it is possible to avoid some of the technicalities needed to be able to handle the general case and concentrate on the key steps and ideas of the proof. Let us briefly sketch them in the following paragraphs:

We need to construct a family of measures $(\pi ^{\alpha})_{\alpha \in [0,1]}$ on $\mathbb{R}^S$ that satisfies both Definition \ref{def:mgpara} (i)-(iii) and property \eqref{eq:ShadowCurveintro}\footnote{Of course with $\R^{[0,\infty)}$ replaced by $\R^S$ and $[0,t]$ replaced by $S\cap [0,t]$.} and need to show that this family is uniquely determined by these two properties (note that \eqref{eq:ShadowCurveintro} already implies that $\pi := \pi ^1$ is a solution of the peacock problem w.r.t.\ $(\mu_t)_{t \in S}$ and  that condition (v) of Definition \ref{def:mgpara} is satisfied). To this end, we pursue the following approach:

\begin{itemize}
	\item \textsf{STEP 1:} Any family of measures $(\pi ^{\alpha})_{\alpha \in [0,1]}$ on $\mathbb{R}^S$ satisfies Definition \ref{def:mgpara} (i)-(iii) if and only if there exists a family of martingale measures $(\hat{\pi}^a)_{a \in [0,1]}$ such that 
	$ \pi ^{\alpha} = \int _0 ^{\alpha} \hat{\pi}^a \de a$
	for all $\alpha \in [0,1]$.
	In particular, the family $(\pi ^{\alpha})_{\alpha \in [0,1]}$ is uniquely determined by any such family $(\hat{\pi}^a)_{a \in [0,1]}$ and for Lebesgue-a.e.\ $a \in [0,1]$ it holds 
	\begin{equation}
	\hat{\pi}^a = \lim_{h \downarrow 0} \frac{\pi ^{a+h} - \pi ^a}{h}
	\end{equation}
	under an appropriate topology (cf.\ Subsection \ref{sec:Notation}).
	Thus, $(\pi ^{\alpha})_{\alpha \in [0,1]}$ satisfies property \eqref{eq:ShadowCurveintro} if and only if  the following two properties hold: For Lebesgue-a.e.\ $a\in [0,1]$ and all $t\in S$ the limit 
	\begin{equation}\label{eq:tu01}
	\hat{\eta}^a_t = \lim_{h \downarrow 0} \frac{\shadow{\mu _{[0,t]\cap S}}{\nu ^{a+h}} - \shadow{\mu _{[0,t]\cap S}}{\nu ^{a}} }{h}
	\end{equation}
	exists and the distribution of $X_t$ under $\hat{\pi}^a$ is $\hat{\eta}^a_t$. Step 1 is accomplished in Subsection \ref{sec:RDofParam}.
	
	\item \textsf{STEP 2:} Step 1 implies that there exists a family $(\pi ^{\alpha})_{\alpha \in [0,1]}$ with the desired properties, if there exists a family $(\hat{\pi}^a)_{a \in [0,1]}$ of martingale measures on $\mathbb{R}^S$ such that the canonical process under $\hat{\pi}^a$ has marginal distributions $(\hat{\eta} ^a _t)_{t \in S}$ for Lebesgue a.e.\ $a \in [0,1]$. By Kellerer's Theorem, for fixed $a \in [0,1]$ such a martingale measure $\hat{\pi}^a$ exists if $(\hat{\eta} ^a_t)_{t \in S}$ is a peacock. Using the calculus rules that we develop for general obstructed shadows, we show  in Subsection \ref{sec:CountableExistence} that for Lebesgue-a.e.\ $a \in [0,1]$ the limit in \eqref{eq:tu01} exists and that $(\hat\eta ^a_t)_{t \in S}$ is  a peacock.
	
	\item \textsf{STEP 3:} Another implication of Step 1 is that the family $(\pi ^{\alpha})_{\alpha \in [0,1]}$ constructed in Step 2 is uniquely determined if there exists only one martingale measure with marginal distributions $(\hat{\eta} _t ^a)_{t \in S}$ for Lebesgue a.e.\ $a \in [0,1]$.
	Unfortunately, just from the defining equation \eqref{eq:tu01}, the peacock $(\hat\eta^a_t)_{t\in S}$ does not need to satisfy this very restrictive property for all $a \in [0,1]$ where $(\hat\eta ^a _t)_{t \in S}$ is defined (cf.\ Example \ref{expl:RDNotNSI}).
	
	That being said, it is sufficient for us that $(\hat\eta^a_t)_{t \in S}$ is NSI for Lebesgue-a.e.\ $a\in [0,1]$ since the NSI property implies the uniqueness of a martingale associated with  $(\hat\eta^a_t)_{t\in S}$ (cf.\ Section \ref{sec:choquet}). To show this, we introduce an auxiliary optimization problem and establish a corresponding monotonicity principle (cf.\ Subsection \ref{sec:MonotonicityPrinciple}). The minimality of the shadow in conjunction with the $\leqcs$-convexity of $(\nu ^{\alpha})_{\alpha \in [0,1]}$ implies that $(\hat{\eta}^a )_{a \in [0,1]}$ is a minimizer of this optimization problem which in turn implies that $(\hat\eta^a _t)_{t \in S}$ is NSI  for Lebesgue-a.e\ $a$ as desired (see Subsection \ref{sec:coreArgument}).
\end{itemize}

If $S$ was finite, we could use the concept of Kellerer dilations as in \cite{BeJu21} to show that for all $a \in [0,1]$ where $(\hat{\eta}^a_t)_{t \in S}$ is well-defined there is only one martingale measure with these one-dimensional marginal distributions (cf.\ Remark \ref{rem:NSIandKellerer}). However, as shown in Example \ref{expl:RDNotNSI}, this is not true if $S$ is infinite. This major difference between the case of a finite index set and a countable infinite one, is the reason why we have to develop new tools and techniques and cannot extend methods used in \cite{BeJu21} and \cite{NuStTa17}.

In Section \ref{sec:ContTimeRC} we establish Theorem \ref{thm:GenExist} in the setting of a continuous time index set $T\subset [0,\infty)$ under the additional assumption that the given peacock is right-continuous. Martingale regularization techniques imply that martingale measures are uniquely determined by their behaviour on a countable index set. We will show in Subsection \ref{sec:ShadObstbyPCOC} that also the obstructed shadow and the NSI property are determined by the behaviour of the peacock $(\mu_t)_{t\in T}$ restricted to a well chosen countable index set $S\subset T$. This allows us to lift the results from Section \ref{sec:GenExistCountable} to the setting of $T\subset [0,\infty)$ with right-continuous peacock.

In Section \ref{sec:AbstractIndexSet} we show how we can pass to an abstract totally ordered index set without any assumptions on the peacock. In particular, this completes the proof of Theorem \ref{thm:GenExist} (recall that it was stated for the totally ordered space $T=[0,\infty)$). Moreover, we explain how Corollary \ref{thm:intro1} follows from Theorem \ref{thm:GenExist}.

The proof of Theorem \ref{thm:LMAtlOptimal} is contained in Section \ref{sec:LeftMonotone}.

Finally, in Section \ref{sec:Examples} we discuss counterexamples regarding shadows and NSI peacocks and provide explicit examples of shadow martingales.

\section{Related literature} \label{sec:RelatedLit}

The optimal transport theory dates back to Monge (1781) and Kantorovich (1939) and has a huge variety of different facets and applications (see e.g. \cite{Vi09}).  Martingale optimal transport is a relatively new subdomain, that has for instance applications in robust mathematical finance (see e.g.\ \cite{AcBePeSc16} or the book \cite{He17}). Given two probability measures $\mu_0$ and $\mu _1$ with $\mu _0 \leqc \mu _1$ and a cost function $c$, the goal is to minimize (or maximize)
\begin{equation*}
\pi \mapsto \mathbb{E}_{\pi}[c(X_0,X_1)]
\end{equation*}
over the set of couplings of $\mu _0$ and $\mu _1$ that additionally satisfy the martingale property. Among the solutions of the problem (for different cost functions) are the couplings presented by Hobson and Neuberger \cite{HoNe12}, Hobson and Klimmek \cite{HoKl15} and for other related problems the couplings recently introduced in \cite{JoMa20} by Jourdain and Margheriti and in \cite{BeJu21} by Beiglb\"ock and Juillet. Note that martingale optimal transport problems are a special case of a wider class of transport problems as weak optimal transport problems \cite{GRST17,GoJu20} or linear transfers \cite{BoGh19}.

The left-curtain coupling introduced by Beiglb\"ock and Juillet in \cite{BeJu16} is of particular importance for our approach of the peacock problem. Besides being the unique minimizer for a certain class of cost functions (cf.\ \cite{BeJu16}), the left-curtain coupling has several different characterizations, for instance concerning the geometry of its support \cite{Ju14} or in the context of the Skorokhod Embedding problem \cite{BeHeTo17, BeJu21}. Moreover, Hobson and Norgilas show in \cite{HoNo17} that it possesses a natural interpretation in Mathematical Finance.

The existence of the shadow (without coining this name) was established by Rost in \cite{Ro71} in the context of stopping times for Markov processes. Beiglb\"ock and Juillet later rediscovered this object  in the context of martingale optimal transport and established several important properties that are used in this paper.
In fact, the left-curtain coupling $\pi$ is introduced by
\begin{equation*}
\pi (X_0 \leq a, X_1\in \cdot ) = \shadow{\mu _1}{\mu _{0|(- \infty,a]}}.
\end{equation*}
for every $a\in \R$. The shadow martingale w.r.t.\ the left-curtain parametrization is a natural extension of this coupling (and hence also of its discrete time extension by Nutz, Stebegg and Tan in \cite{NuStTa17}) to the continuous time case. Moreover, shadow martingales extend the concept of shadow couplings introduced in \cite{BeJu21}.

The name peacock (alias PCOC) which is derived from the French term \textit{Processus Croissant pour l'Ordre Convexe} and likewise the peacock problem were introduced by Hirsch, Profeta, Roynette and Yor in their monograph \cite{HiPrRoYo11} and are therefore quite recent. However, the construction of martingales that match given marginal distributions at least goes back to the seminal work of Kellerer \cite{Ke72}. Since then a variety of solutions have been developed before it was subsumed under the name peacock problem. Most of these solutions make more or less restrictive additional assumptions on the peacock, e.g.\ assuming that the peacock satisfies the (IMRV) property (see \cite{MaYo02}) or consists of the marginal distributions of a solution to a certain class of SDE (see \cite{Gy86}). Moreover, the construction of fake Brownian motions (e.g.  \cite{Al08, HaKl07}) can be seen as solutions to a (very specific) peacock problem. The monograph \cite{HiPrRoYo11} provides a  comprehensive overview of solutions to the peacock problem that work with special classes of peacocks.

Recently there were several contributions that face generic peacocks without a rich additional structure. There is the solution of Hobson \cite{Ho17} which is based on the Skorokhod Embedding Problem and the one of Lowther \cite{Lo08} who constructs for continuous peacocks with connected support a solution under which the canonical process is a strong Markov process. The approach closest to our class of solutions is the one independently  studied by Henry-Labord\`ere, Tan and Touzi \cite{HeTaTo16} and Juillet \cite{Ju18}. The shadow martingale w.r.t.\ the left-curtain parametrization $(\nu ^{\alpha} _{\lc})_{\alpha \in [0,1]}$ is the limit of the discrete time simultaneous minimizer of
\begin{equation*}
\mathbb{E}_{\pi}[c(X_0, X_{t_k})] \hspace{1cm} \forall \ 1 \leq k \leq n
\end{equation*} 
among all martingale coupling of $\mu _0, \mu_{t_1},\ldots, \mu _{t_n}$ for all $c$ with $\partial_{xyy} c< 0$ as $n$ tends to infinity  for a suitable chosen sequence of nested finite partitions of $T$ whose mesh tends to zero.  In contrast, the solution of \cite{HeTaTo16,Ju18}-- when it exists-- is constructed as the limit of the concatenation of the discrete time simultaneous minimizers of
\begin{equation*}
\mathbb{E}_{\pi}[c(X_{t_{k-1}}, X_{t_k})] \hspace{1cm} \forall \ 1 \leq k \leq n
\end{equation*} 
for all $c$ with $\partial_{xyy}c < 0$.
Unsurprisingly, this solution behaves notably differently than the shadow martingale induced by the left-curtain parametrization (see Example \ref{expl:LCSM}).

Besides this article we are not aware of solutions to the peacock problem that are related to shadow martingales w.r.t.\ a parametrization which is not the left-curtain one.  Similarly, there are no results about uniquely constructing martingales by solely describing how parts of the initial distribution  evolve. In fact, the only approach in this direction that we are aware of is \cite{BoJu18} but in a non-martingale setup.

\section{Preliminaries} \label{sec:Prelimin}

In this section we introduce our notation and recall objects and properties that are well known in the context of martingale optimal transport and the peacock problem. Since we want to work at the level of probability distributions (of processes), we  sometimes choose a non-standard perspective on standard results.

\subsection{Notation} \label{sec:Notation}

We denote by $\MZ(\mathsf{X})$ (resp.\ $\PZ(\mathsf{X})$) the set of all finite measures (resp.\ probabilty measures) on some measurable space $\mathsf{X}$. The underlying space will mostly be the space of functions from $T$ to $\mathbb{R}$, denoted by  $\mathbb{R}^T$, for some totally ordered set $(T, \leq)$. In this case, the space $\mathbb{R}^T$ is equipped with the product topology and the corresponding Borel $\sigma$-algebra, the canonical process on $\mathbb{R}^T$ is denoted by $(X_t)_{t \in T}$, i.e.\ for all $t \in T$
\begin{equation*}
X_t : \mathbb{R}^T \ni \omega \mapsto \omega(t) \in \mathbb{R},
\end{equation*}
and $(\mathcal{F}_t)_{t\in T}$ is its natural filtration defined by $\mathcal{F}_t= \sigma ( X_s : s \leq t)$.

The set $\MO(\mathbb{R}^T)$ consists of all $\pi \in \MZ(\mathbb{R}^T)$ for which all one-dimensional marginal distributions  have a finite first moment. We equip $\MO(\mathbb{R}^T)$ with the initial topology generated by the functionals $(I_f)_{f \in \mathcal{G}_{0} \cup \mathcal{G}_1}$ where
\[I_f : \MO(\mathbb{R}^T) \ni \pi \mapsto \int _{\mathbb{R}^T} f \de \pi \in \mathbb{R}\]
and
\begin{align*}
\left\{
\begin{aligned}
\mathcal{G}_0 &= \left\{ g \circ (X_{t_1},\ldots,X_{t_n}) : n \geq 1, \ t_1,\ldots, t_n \in T, \ g \in C_b(\mathbb{R}^n) \right\}\\
\mathcal{G}_1 &= \left\{|X_t| : t \in T \right\}.
\end{aligned}
\right.
\end{align*} 
We denote this topology on $\MO(\mathbb{R}^T)$ by $\TO$. In contrast, we denote by $\TZ$ the initial topology on $\MZ(\mathbb{R}^T)$ that is generated by the functionals $(I_f)_{f \in \mathcal{G}_0}$ only. The subspace of probability measures in $\MO(\mathbb{R}^T)$ is denoted by $\PO(\mathbb{R}^T)$ and equipped with the inherited topology.  If $T$ is finite, the topology on $\PO(\mathbb{R}^T)$ is induced by the $1$-Wasserstein metric $\mathcal{W}_{1,l_1}$ corresponding to the $l_1$-metric on $\mathbb{R}^T$ (see Villani \cite[Theorem 6.9]{Vi09}). It is also not difficult to see, that, if $T$ is countable, $\MO(\mathbb{R}^T)$ is first countable and therefore continuity is equivalent to sequential continuity. 
Note now that we can reduce $\mathcal{G}_0$ in the definition of $\TZ$ to the following set of functions 
\begin{align*}
\mathcal{G}'_0 = \{\omega\in \R^T\mapsto 1\}\cup\left\{ g \circ (X_{t_1},\ldots,X_{t_n}) : n \geq 1, \ t_1,\ldots, t_n \in T, \ g \in C_c(\mathbb{R}^n) \right\}\ .
\end{align*}
To see this, recall that a sequence $(\pi _n)_{n \in \mathbb{N}}$ converges to $\pi$ w.r.t.\ some initial topology $\mathcal{T}$ defined by some set of functions $\mathcal{G}$ if and only if $(I_f(\pi _n))_{n \in \mathbb{N}}$ converges to $I_f(\pi)$ in $\mathbb{R}$ for all $f \in \mathcal{G}$. If $T$ is finite and $\mathcal{G}$ contains $\mathcal{C}_c(\R^T)$, the same limit is also satisfied for any continuous function $f$ that can be dominated by a linear combination generated with elements $f_i\in\mathcal{G}$, i.e such that $ |f|\leq \sum_i \zeta_i f_i.$
This is the reason why (i) the topology generated by $\mathcal{G}'_0$ is exactly $\TZ$ (also if $T$ is infinite), (ii) functions $g\circ(X_{t_1},\cdots,X_{t_n})$ where $g$ grows at most linearly at infinity are admissible for $\TO$.

We denote the push-forward of a measure $\pi\in \MO(\mathbb{R}^T)$ under some measurable map $f$ defined on $\mathbb{R}^T$ by $f_\#\pi$. If $\pi$ is a probability distribution, we refer to the push-forward as the law or distribution of $f$ under $\pi$ denoted by $\mathrm{Law}_{\pi}(f)$. Furthermore, the expression ``marginals of a probability measure $\pi$ on $\mathbb{R}^T$'', always refers to  all the one-dimensional marginal distributions of the canonical process under $\pi$, i.e.\ to the measures $\mathrm{Law}_{\pi}(X_t)$ for $t \in T$.   

Let $S$ be a subset of $T$ and $\mathrm{proj}^S : \mathbb{R}^T \rightarrow \mathbb{R}^S$ the projection on the index set $S$. The induced projection map 
\begin{equation*}
\MO(\mathbb{R}^T) \ni \pi \mapsto (\mathrm{proj}^S)_\# \pi \in \MO(\mathbb{R}^S)
\end{equation*}
is continuous w.r.t.\ $\TO$ on $\MO(\mathbb{R}^T)$ and $\MO(\mathbb{R}^S)$. We denote the measure $\pi$ projected on the coordinates in $S$, i.e. $(\mathrm{proj}^S)_\#\pi$, by the shorter notation $\pi_{|S}$ as if it were the restriction of a random vector.
Moreover, we denote the cumulative distribution function of a probability measure $\mu$ on $\mathbb{R}$ by $F_{\mu}$ and its quantile function is
\begin{equation*}
F_{\mu}^{-1}:\alpha\in[0,1]\mapsto \inf \{ x \in \mathbb{R} : F_{\mu}(x) \geq \alpha \}.
\end{equation*}
We also denote by $\lambda$ the Lebesgue measure on $[0,1]$, by $\mathrm{Unif}_{I}$ the uniform distribution on an interval $I \subset \mathbb{R}$ and by $\delta_x$ the Dirac measure at point $x$.

\subsection{Order relations and potential functions} \label{sec:OrderRelPF}

We use several partial order relations on $\MO(\mathbb{R})$. They can all be introduced in a parallel way saying that $\mu  \in \MO(\mathbb{R})$ is smaller than or equal to $\mu ' \in \MO(\mathbb{R})$  if
\begin{equation} \label{eq:OrderRelation}
\int _{\mathbb{R}} \varphi \de \mu \leq \int _{\mathbb{R}} \varphi \de \mu '
\end{equation} 
for every $\varphi$ in a certain positive cone of measurable test functions. These orders and the corresponding cones are:
\begin{itemize}
	\item \textit{The positive order:} $\mu \leqp \mu '$ if \eqref{eq:OrderRelation} holds for all non-negative $\varphi$.
	\item \textit{The convex order:} $\mu \leqc \mu '$ if \eqref{eq:OrderRelation} holds for all convex $\varphi$.
	\item \textit{The convex-positive order:} $\mu \leqcp \mu '$ if \eqref{eq:OrderRelation} holds for all non-negative convex $\varphi$.
	\item \textit{The convex-stochastic order:} $\mu \leqcs \mu '$ if \eqref{eq:OrderRelation} holds for all increasing convex $\varphi$.
\end{itemize}

Both non-negative and convex functions are bounded from below by an affine  function and thus, since the first moments are finite, the integrals in \eqref{eq:OrderRelation} are well-defined with values in $(- \infty, \infty ]$. The positive order is well-defined for finite measures on a measurable space (e.g.\ $\mathbb{R}^T$). 
Moreover, recall from the introduction that we call $\pi$ a submeasure of $\pi'$ if $\pi \leqp \pi'$.

The convex-positive order and the convex-stochastic order are less common in the literature. For a discussion of these ``combined'' partial order relations and the relationships between them we refer to \cite[Section 1]{Ju14} and especially Theorem 1.7 therein.

\begin{lemma} \label{lemma:relationOrders}
	Let $\mu$ and $\mu'$ be in $\MO(\mathbb{R})$.
	\begin{enumerate}
		\item[(i)] If $\mu \leqc \mu '$, then $\mu (\mathbb{R}) = \mu' (\mathbb{R})$ and $\int _{\mathbb{R}} y \de \mu (y) = \int _{\mathbb{R}} y \de \mu' (y)$.
		\item[(ii)] If $\mu \leq  _{c,+} \mu '$ and $\mu (\mathbb{R}) = \mu' (\mathbb{R})$ , then $\mu \leqc \mu '$.
		\item[(iii)] If $\mu \leq  _{c,s} \mu '$ and $\int _{\mathbb{R}} y \de \mu (y) = \int _{\mathbb{R}} y \de \mu' (y)$, then $\mu \leqc \mu'$.
	\end{enumerate}
\end{lemma}

\begin{proof}
	Item (i): The four functions $x\mapsto \pm 1$ and $x\mapsto \pm x$ are convex functions. 
	
	Item (ii): If $\mu \leq  _{c,+} \mu '$ and $\mu (\mathbb{R}) = \mu' (\mathbb{R})$, equation \eqref{eq:OrderRelation} is satisfied for every non-negative convex function and also for $x\mapsto -1$. Thus \eqref{eq:OrderRelation} is satisfied for any convex function that is bounded from below. Hence, for a general convex function $\varphi$ with $\int _{\mathbb{R}} \varphi \de \mu' < + \infty$, \eqref{eq:OrderRelation} holds for $\varphi_n=\varphi\vee(-n)$ and by the monotone convergence theorem it holds for $\varphi$ as well. 	
	
	Item (iii): We use a similar argument adding $x\mapsto -x$ to the set of nondecreasing convex functions. Any convex function is the pointwise increasing limit of a sequence of convex functions with limit slope bounded at $-\infty$.
\end{proof}

\begin{lemma} \label{lemma:preservingOrder}
	Let $\mu_n, \mu_n', \mu, \mu' \in \MO(\mathbb{R})$ for all $n \in \mathbb{N}$.
	\begin{enumerate}
		\item [(i)] Suppose $(\mu _n)_{n \in \mathbb{N}}$ and $(\mu' _n)_{n \in \mathbb{N}}$ converge to $\mu$ and $\mu'$ under $\TZ$. If $\mu _n \leqp \mu_n'$ for all $n \in \mathbb{N}$, then $\mu \leqp \mu'$.
		\item [(ii)] Suppose $(\mu _n)_{n \in \mathbb{N}}$ and $(\mu' _n)_{n \in \mathbb{N}}$ converge to $\mu$ and $\mu'$ under $\TO$. For any order relation  $\leqc$, $\leqcp$ or $\leqcs$ represented by $\leq$, the relations $\mu _n \leq \mu '_n$ for all $n \in \mathbb{N}$ imply $\mu \leq \mu'$. 
	\end{enumerate}
\end{lemma}

\begin{proof}	
Item (i): It is sufficient to test the positive order by indicator functions of closed intervals. For any such function $\varphi$ there exists a sequence $(\varphi _m)_{m \in \mathbb{N}}$ of continuous bounded functions such that $0 \leq \varphi _m \leq \varphi$ for all $m \in \mathbb{N}$ and $\int _{\mathbb{R}} \varphi \de \mu = \lim _{m \rightarrow \infty} \int _{\mathbb{R}} \varphi _m \de \mu$. Since convergence in $\TZ$ implies $\int _{\mathbb{R}} \varphi _m \de \mu = \lim _{n \rightarrow \infty} \int _{\mathbb{R}} \varphi _m \de \mu_n$ for all $m \in \mathbb{N}$, the claim follows. 
	
	Item (ii): For any (non-negative/increasing) convex function $\varphi \in L^1(\mu)$, there exists a sequence $(\varphi _m)_{m \in \mathbb{N}}$ of (non-negative/increasing) convex  functions with bounded slope at $\pm \infty$ such that $ \varphi _m \leq \varphi$ for all $m \in \mathbb{N}$ and $\int \varphi \de \mu = \lim _{m \rightarrow \infty} \int \varphi _m \de \mu$. Since the slope of $\varphi _m$ is bounded, there exist $a_m,b_m > 0$ such that $|\varphi_m(x)| \leq a_m |x| + b_m$ for all $x \in \mathbb{R}$ and hence the claim follows because the sequences converge w.r.t.\ $\TO$ (cf.\ Subsection \ref{sec:Notation}).
\end{proof}

Note that convergence in $\TZ$ does in general not preserve the order relations $\leqc$, $\leqcp$ and $\leqcs$. However, a sequence that is convergent under $\TZ$ and has a uniform upper bound in $\leqcp$, is in fact convergent under $\TO$: 

\begin{lemma}
	\label{lemma:convDom}
	Let $(\mu _n)_{n \in \mathbb{N}}$ be a sequence in $\MO(\R)$. If
	there exists a measure $\theta \in \MO(\mathbb{R})$ such that $\mu _n \leqcp \theta$ for all $n \in \mathbb{N}$, 	then the sequence $(\mu _n)_{n \in \mathbb{N}}$ is uniformly integrable, i.e.\ $$\lim _{N \rightarrow \infty} \sup _{n \in \mathbb{N}} \int _{\mathbb{R}} |x| \1_{[-N,N]^c} \de \mu _n(x) = 0.$$
	Hence,  $(\mu _n)_{n \in \mathbb{N}}$ converges under $\TZ$ if and only if  $(\mu _n)_{n \in \mathbb{N}}$ converges under $\TO$.
	Moreover, if $(\mu _n)_{n \in \mathbb{N}}$ converges under $\TZ$ to $\mu\in \MO(\R)$, then it holds $\int f \de \mu _n \rightarrow \int f d\mu$ for all continuous $f$ for which $|f|$ is dominated by some convex $\varphi \in L^1(\theta)$.
\end{lemma}

\begin{proof}
If $\varphi\in L^1(\theta)$ and $|f|\leq \varphi$ it holds that
\begin{equation*} 
	\int _{\mathbb{R}} |f| \1_{[-N,N]^c} \de \mu _n \leq  \int _{\mathbb{R}} (\varphi + |x| - N)^+ \de \theta	
	\end{equation*}
from which the results easily follows.
\end{proof}

Dealing with the convex order, potential functions are a very useful representation of finite measures on $\mathbb{R}$. They are defined as follows:

\begin{definition}
	Let $\mu \in \MO(\mathbb{R})$. The potential function of $\mu$ is the function
	\begin{equation*}
	U(\mu) : \mathbb{R} \ni x \mapsto \int _{\mathbb{R}} |y - x| \de \mu (y)\in \R^+.
	\end{equation*}
\end{definition}

Since elements of $\MO(\mathbb{R})$ have finite first moments, the potential function is always well-defined.
We collect a few important properties of potential functions below.

\begin{lemma}[{cf.\ \cite[Proposition 4.1]{BeJu16}}]
	\label{lemma:characPotF}
	Let $m \in [0,\infty)$ and $x^* \in \mathbb{R}$.
	For a function $u:\mathbb{R} \rightarrow \mathbb{R}$ the following statements are equivalent:
	\begin{enumerate}
		\item [(i)] There exists a finite measure $\mu \in \MO(\mathbb{R})$ with mass $\mu(\mathbb{R}) = m$ and barycenter $x^* = \int _{\mathbb{R}} x \de \mu (x)$ such that $U(\mu) = u$ .
		\item [(ii)] The function $u$ is non-negative, convex and satisfies
		\begin{equation} \label{eq:characPotF}
		\lim _{x \rightarrow \pm \infty} u(x) - m|x - x^*| = 0. 
		\end{equation}
	\end{enumerate}
	Moreover, for all $\mu, \mu' \in \MO(\mathbb{R})$ we have $\mu = \mu'$ if and only if $U(\mu) = U(\mu')$.
\end{lemma}

Convex ordering and convergence in $\MO(\mathbb{R})$ can be expressed via potential functions. 

\begin{lemma}
	\label{lemma:propertiesPotF}
	For all $\mu$, $\mu'$ and sequences $(\mu _n)_{n \in \mathbb{N}}$ in $\MO(\mathbb{R})$ with $\mu(\mathbb{R}) = \mu'(\mathbb{R}) = \mu _n(\mathbb{R})$ for all $n \in \mathbb{N}$, we have the following properties:
	\begin{enumerate}
		\item [(i)] It holds $\mu \leqc \mu'$  if and only if $U(\mu) \leq U(\mu')$.
		\item [(ii)] It holds $\mu _n \rightarrow \mu$ under $\TO$ if and only if $U(\mu _n) \rightarrow U(\mu)$ pointwise.
	\end{enumerate}
\end{lemma} 
 \begin{proof}
 	Since for every $x\in \R$  the function $f_x:y\mapsto |y-x|$ is convex the direct implication of (i) is obvious. The reverse implication is part of the folklore (see e.g Exercise 1.7 of \cite{HiPrRoYo11}). It can be proved as follows: let $C$ be the cone of real functions $f$ for which $\int f \de \mu \leq \int f \de \mu'$. It includes the constants and also the functions $f_x$, $x\in \R$. Considering both sequences $(f_{\pm n}-n)_{n\in \mathbb{N}}$, by the monotone convergence theorem we obtain $\pm x\in C$. Hence $C$ contains any piecewise (we mean with finitely many pieces) affine convex function. By the monotone convergence theorem again we see that every convex function is in $C$.
 	
 	Since $f_x$ is affine close to $\pm\infty$, the direct implication of (ii) is obvious.  For the reverse implication, since all measures have the same finite mass and $U(\mu_n)(0)\to U(\mu)(0)$ we have $\int f d\mu_n\to_{n\in \infty} \int f d\mu$ for $x\mapsto 1$ and $x\mapsto |x|$. Therefore it suffices to establish the convergence for  every continuous and compactly supported function $f$. Notice that the vectorial space spanned by the functions $f_x$ and the constant functions includes the continuous and piecewise affine functions with compact support. Hence we can conclude by their density in $\mathcal{C}_c(\R)$ for the uniform norm.
 \end{proof}

Specified to families monotonously increasing in convex-stochastic order, the second part of the previous lemma yields the following result.

\begin{corollary} \label{cor:MonotonicityPCOC}
	Let $T \subset \mathbb{R}$ and $(\mu _t)_{t \in T}$ be a family in $\MO(\mathbb{R})$ that is increasing in convex-stochastic order, i.e. $\mu _s \leqcs \mu _t$ for all $s \leq t$ in $T$.
	There exists a countable set $S \subset T$ such that $t \mapsto \mu _t$ is a continuous map from $T \setminus S$ to $\MO(\mathbb{R})$ under $\TO$.
\end{corollary}

\begin{proof}
	For all $q\in \mathbb{Q}$ the function 
	\begin{equation}\label{eq:mon int}
	t \mapsto U(\mu _t)(q) = \int _{\mathbb{R}} |y-q| \de \mu _t(y) = 2 \int _{\mathbb{R}} (y-q)^+ \de \mu _t(y) - \int _{\mathbb{R}} (y-q) \de \mu _t(y)
	\end{equation}
	is continuous except on a countable set $S_q$ because it is the difference of two functions that are monotonously increasing in $t$. Set $S = \bigcup _{q \in \mathbb{Q}} S_q$.  Observe that $\bar{u}_t := \lim _{s \downarrow t} U(\mu _s)$ is a well defined convex function as a pointwise limit of convex functions. This limit exists by monotonicity in $t$ of the integrals in \eqref{eq:mon int}.
	Also $u_t(x) := U(\mu_t)(x)$
	is a convex function 
	and we get $\bar{u}_t = u_t$ on $\mathbb{Q}$ for all $t \not \in S$. Since both $\bar{u}_t$ and $u_t$ are continuous as convex  functions this equality extends to $\mathbb{R}$. Similarly, it holds $\lim _{r \uparrow t} U(\mu _r)(x) = U(\mu _t)(x)$ for all $x \in \mathbb{R}$ and  $t \not \in S$.
	Thus, the map $t \mapsto U(\mu _t)(x)$ is continuous for every $x \in \mathbb{R}$ at any time $t \not \in S$. This transfers to the continuity of $t\mapsto \mu_t$ outside of $S$ by Lemma \ref{lemma:propertiesPotF} (ii).
\end{proof}

Despite monotonicity, a family increasing in convex-stochastic order does in general not admit left- and right-limits everywhere under $\TO$. For instance, the family $(\mu _t)_{t \in [0,1]}$ with
\begin{equation*}
\mu _t = \frac{1-t^2}{2-t^2} \delta _{-\frac{1}{1-t}} + \frac{1}{2-t^2} \delta _{1+t} \hspace*{1cm} \mu _1 = \delta _2
\end{equation*}
does not have a left-limit at $1$.

\subsection{Infimum and supremum in convex order}

\begin{definition}
	Let $\mathcal{A}$ be a set of measures in $\MO(\mathbb{R})$. If $\mathcal{A}$ possesses a smallest upper bound w.r.t.\ convex order, we call it the convex supremum of $\mathcal{A}$ and denote it by $\mathrm{Csup} \ \mathcal{A}$. It is then the unique measure $\zeta$ such that
	\begin{enumerate}
		\item [(i)] $\mu \leqc \zeta$ for all $\mu \in \mathcal{A}$ and 
		\item [(ii)] $\zeta \leqc \zeta'$ for all $\zeta'$ that satisfy (i).
	\end{enumerate}
	Similarly, we define $\mathrm{Cinf} \ \mathcal{A}$ as the convex infimum, if it exists.
\end{definition}

\begin{proposition} \label{prop:ExistenceCSup}
	Let $\mathcal{A}$ be a non-empty subset of $\MO(\mathbb{R})$ such that all measures in $\mathcal{A}$ have the same mass and the same barycenter.
	\begin{enumerate}
		\item [(i)]  The convex infimum $\mathrm{Cinf} \, \mathcal{A}$ exists. 
		\item[(ii)] 
		If there exists some $\theta \in \MO(\mathbb{R})$ such that $\mu  \leqcp \theta$ holds for all $\mu \in \mathcal{A}$, then the convex supremum  $\mathrm{Csup} \, \mathcal{A}$ exists.
	\end{enumerate}
	Moreover, their potential functions satisfy
	\begin{eqnarray*}
		U(\mathrm{Cinf} \, \mathcal{A}) = \mathrm{conv}\left( \inf _{\mu \in \mathcal{A}} \ U(\mu) \right) \ \text{and}	 \hspace{3mm}	U(\mathrm{Csup} \, \mathcal{A}) = \sup _{\mu \in \mathcal{A}} \ U(\mu).
	\end{eqnarray*}		
	where $\mathrm{conv}(f)$ denotes the convex hull of a function $f$, i.e. the largest convex function that is pointwise smaller than $f$.
\end{proposition}

\begin{proof}
	Item (i): Since the measures of $\mathcal{A}$ all have the same mass $m$ and barycenter $x$ and $U(m \delta _x)$ is convex, the set $\{ g : g \textit{ convex, } g \leq \inf _{\mu \in \mathcal{A}} U(\mu) \}$ is not empty. Let $u$ be defined by $u(x) = \sup \{ g(x) : g \textit{ convex, } g \leq \inf _{\mu \in \mathcal{A}} U(\mu) \}$. It is convex as the pointwise supremum of convex functions and $U(m \delta _x) \leq u \leq U(\mu)$ for any fixed $\mu \in \mathcal{A}$. Hence $u$ possesses the right behaviour at $\pm \infty$ in the sense of Lemma \ref{lemma:characPotF}. Therefore $u$ is a potential function and the corresponding measure satisfies the properties of a convex infimum by Lemma \ref{lemma:propertiesPotF} (i).
	
	Item (ii): According to \cite[Lemma 4.5]{BeJu16} applied to $m\delta_x$ and $\theta$ the ordering $m\delta_x\leqcp \theta$ implies that there exists a $\theta' \in \MO(\mathbb{R})$ with $m \delta _x \leq _c \theta'$ such that for all $\eta \in \MO(\mathbb{R})$ with $m\delta_x\leqc \eta\leqcp \theta$ we have $\eta\leqc \theta'$. More precisely if, if $\theta(\R)=m$ we choose $\theta'=\theta$. If not, there exists $a\leq a'$ and $b,b'$ with $b\in [0,\theta(\{a\})]$, $b'\in [0,\theta(\{a'\})]$ such that $\theta'=\theta_{|(-\infty,a)}+b\delta_a+b'\delta_{a'}+\theta_{|(a',+\infty)}$.
	
	In particular, we obtain $U(m \delta_x) \leq U(\mu)  \leq U(\theta')$ for every $\mu\in \mathcal{A}$ and therefore the convex function $u = \sup _{\mu \in \mathcal{A}} U(\mu)$ satisfies $U(m \delta_x) \leq u  \leq U(\theta')$. Thus, $u$ has the right behaviour at $\pm \infty$ in the sense of Lemma \ref{lemma:characPotF}.  Hence, by Lemma \ref{lemma:characPotF} $u$ is a potential function and the corresponding measure satisfies the properties of a convex supremum (see Lemma \ref{lemma:propertiesPotF} (i)).
\end{proof}

\begin{remark}
	The assumption that all measures in $\mathcal{A}$ have the same mass and barycenter is equivalent to the assumption that $\mathcal{A}$ has a lower bound w.r.t.\ the convex order. Hence, we don't need an additional lower bound in (i).
\end{remark}

\begin{lemma} \label{lemma:CSupSum}
	Let $(\mu _n) _{n \in \mathbb{N}}$ be a sequence in $\MO(\mathbb{R})$.
	\begin{enumerate}
		\item [(i)] If $\mu _m \leqc \mu _n$ for all $n \leq m$ in $\mathbb{N}$, then $(\mu _n) _{n \in \mathbb{N}}$ converges to $\mathrm{Cinf} \{ \mu _n : n \in \mathbb{N}\}$ under $\TO$.
		\item [(ii)] If $\mu _n \leqc \mu _m \leqcp \theta$ for all $n \leq m$ in $\mathbb{N}$ and some $\theta \in \MO(\mathbb{R})$, then $(\mu _n) _{n \in \mathbb{N}}$ converges to $\mathrm{Csup} \{ \mu _n : n \in \mathbb{N}\}$ under $\TO$.
		\item [(iii)] If $(\mu _n')_{n \in \mathbb{N}}$ is another sequence in $\MO(\mathbb{R})$ and both sequences are increasing in convex-oder and are uniformly bounded from above in convex-positive order, then
		\begin{equation*}
		\mathrm{Csup} \left\{ \mu _n + \mu' _n : n \in \mathbb{N}\right\} = \mathrm{Csup} \left\{ \mu _n : n \in \mathbb{N}\right\} +  \mathrm{Csup} \left\{ \mu' _n : n \in \mathbb{N}\right\}.
		\end{equation*}
	\end{enumerate}
\end{lemma}

\begin{proof}
	With Proposition \ref{prop:ExistenceCSup} and Lemma \ref{lemma:propertiesPotF} we can rewrite this statement in terms of sequences  of real functions and then the statement is well known.
\end{proof}

\begin{lemma} \label{lemma:CSupApprox}
	Let $\mathcal{A}$ be a non-empty subset of $\MO(\mathbb{R})$ such that all measures in $\mathcal{A}$ have the same mass, the same barycenter and are dominated by some $\theta \in \MO(\mathbb{R})$ in convex-positive order.
	If additionally for all $\mu_1, \mu _2 \in \mathcal{A}$ there exists some $\mu ' \in \mathcal{A}$ such that $\mu _1 \leqc \mu '$ and $\mu _2 \leqc \mu '$, then there exists an increasing sequence $(\mu _n)_{n \in \mathbb{N}}$ in $\mathcal{A}$ that converges to $\mathrm{Csup} \ \mathcal{A}$ under $\TO$.
\end{lemma}

\begin{proof}
	The potential function of $\mathrm{Csup}\, \mathcal{A}$ is given by $u = \sup _{\mu \in \mathcal{A}} U(\mu)$.  For any $q \in \mathbb{Q}$ there exists a sequence $(\nu ^{q}_k)_{k \in \mathbb{N}}$ of measures in $\mathcal{A}$ such that for the corresponding potential functions $u^q _k = U(\nu ^q _k)$ the sequence $(u_k ^q (q))_{k \in \mathbb{N}}$ converges to $u(q)$. 
	
	Let $(q_n)_{n \in \mathbb{N}}$ be an enumeration of $\mathbb{Q}$, set $\mu _1 = \nu ^{q_1} _1$ and choose a $\mu _{n} \in \mathcal{A}$ that is an upper bound in convex order to the finite set
	\begin{equation*}
	\{ \mu _{n-1} \} \cup \left\{ \nu ^{q_l} _k : 1 \leq k,l \leq n  \right\}
	\end{equation*}
	which is possible by assumption.	
	Thereby, we get an increasing sequence in $\mathcal{A}$ that satisfies $\lim _{n \rightarrow \infty} U(\mu _n)(q) = u(q)$ for all $q \in \mathbb{Q}$. 
	Since $(\mu _n)_{n \in \mathbb{N}}$ is increasing in convex order, $\lim _{n \rightarrow \infty} U(\mu _n)(x) = \sup _{n \in \mathbb{N}} U(\mu _n)(x)$ for all $x \in \mathbb{R}$.  Thus,   $ \sup _{n \in \mathbb{N}} U(\mu _n)$ and $u$ are  convex functions that agree on $\mathbb{Q}$ and, hence, on $\mathbb{R}$. Hence, $\lim _{n \rightarrow \infty} U(\mu _n)(x)=u(x)$ for all $x\in\R$  and we can apply Lemma \ref{lemma:propertiesPotF} (ii) to conclude that $(\mu_n)_{n \in \mathbb{N}}$ converges to $\mathrm{Csup}\, \mathcal{A}$ under $\TO$.
\end{proof}

\subsection{Peacocks and Kellerer's Theorem}

In this section we introduce notation regarding peacocks and martingale measures.

We fix  a totally ordered index set $(T, \leq)$. 
As already indicated in Subsection \ref{sec:Notation}, we are not working on the level of processes but with their distributions on the state space $\mathbb{R}^T$. However, we would like to introduce the martingale property and the Markov property that are typically formulated for processes indexed by $T$ and not probability measures on $\R^T$. 

\begin{definition} \label{def:MartSpaces}
	Let $\pi \in \PO(\mathbb{R}^T)$.
	\begin{enumerate}
		\item [(i)] We call $\pi$ a martingale measure if the canonical process $(X_t)_{t \in T}$ is a martingale w.r.t.\ its natural filtration under $\pi$, i.e.\ if
		\begin{equation*}
		\mathbb{E}_{\pi} \left[ X_t \ \vert \ \mathcal{F}_s \right] = X_s \hspace{1cm} \pi\text{-a.e.}
		\end{equation*}
		for all $s \leq t$ in $T$. The set of all martingale measures on $\mathbb{R}^T$ is denoted by $\MM_T$.
		\item[(ii)] The probability measure $\pi$ is said to be Markov if the canonical process $(X_t)_{t \in T}$ is a Markov process under $\pi$, i.e. if
		\begin{equation*}
		\mathbb{E}_{\pi} \left[ \1_{A}(X_t)  \vert  \mathcal{F}_s \right] = \mathbb{E}_{\pi} \left[ \1_{A}(X_t)  \vert  X_s \right] \hspace{1cm} \pi\text{-a.e.}
		\end{equation*}
		for all Borel sets $A \subset \mathbb{R}$ and $s < t$ in $T$.
	\end{enumerate}
\end{definition}

We equip $\mathsf{M}_T$ with  the topology inherited from $\TO$ on $\PO(\mathbb{R}^T)$ and the corresponding Borel $\sigma$-algebra. All subsets of $\mathsf{M}_T$ are equipped with the subspace topology and subspace $\sigma$-algebra.\footnote{Recall that the Borel $\sigma$-algebra of the subspace topology coincides with the subspace $\sigma$-algebra of the Borel $\sigma$-algebra corresponding to the topology on the ambient space.}
By Jensen's inequality, the marginal distributions $(\mathrm{Law}_{\pi}(X_t))_{t \in T}$ of a martingale measure $\pi$ form a family in $\PO(\mathbb{R})$ that is increasing in convex order. Recall from the introduction that those families are called peacocks:

\begin{definition}\label{def:PCOCspace}
	We call a family $(\mu _t)_{t \in T}$ in $\PO(\mathbb{R})$ a peacock, if $\mu _s \leqc \mu_t$ for all $s \leq t$ and we denote by $\PP_T$ the set of all peacocks indexed by $T$.
	Moreover, we say that a martingale measure $\pi$ is associated with a peacock $(\mu _t)_{t \in T}$ if $\mathrm{Law}_{\pi}(X_t) = \mu _t$ for all $t \in T$.
\end{definition}

Since all elements of a family of finite measures increasing in convex order have the same mass (not always $1$) by Lemma \ref{lemma:relationOrders} (i), they can therefore easily be rescaled to become peacocks. We equip $\PP_T$ with the inherited product topology on $\PO(\mathbb{R})^T$ where each factor $\PO(\mathbb{R})$ is equipped with $\TO$. The corresponding Borel $\sigma$-algebra is the product $\sigma$-algebra. 

\begin{definition}
	Let $S \subset T$ and $(\mu _t)_{t \in S}$ be a family in $\PO(\mathbb{R})$. By $\mathsf{M}_T((\mu _t)_{t \in S})$ we denote the set of all  martingale measures $\pi \in \MM _T$ satisfying
	$\mathrm{Law}_{\pi}(X_t) = \mu _t$ for all $t \in S$. 
\end{definition}

Thanks to the following result we know precisely when $\MM_T((\mu _t)_{t \in T})$ is not empty.

\begin{proposition}[Kellerer's Theorem \cite{Ke72, Ke73}] \label{prop:KellerersThm}
	Let $(\mu _t)_{t \in T}$ be a family in $\PO(\mathbb{R})$. The following are equivalent:
	\begin{itemize}
		\item [(i)] The family $(\mu _t)_{t \in T}$ is a peacock.
		\item [(ii)] There exists a  martingale measure $\pi \in \MM_T((\mu _t)_{t \in T})$ which can moreover be chosen to be Markov.
	\end{itemize}
\end{proposition}

The existence of solutions to the peacock problem is also true for martingales on $\R^d$ with $d\geq 2$ (cf.\ \cite{HiPrRoYo11}) but it is still an open problem whether in this case the martingale can be chosen to be Markov. An extension to partially ordered sets of indices is possible but only in certain cases (cf. \cite{Ju16}).

\section{Parametrizations, shadows, and NSI} \label{sec:ImportantObjects}

The goal of this section is to introduce the three concepts that are crucial on the one hand for the construction of the shadow martingales, namely parametrizations and obstructed shadows, and on the other hand for the uniqueness of the shadow martingale measure, namely the NSI property, cf.\ Subsection \ref{sec:MainIdeaProof}.
Throughout this section we fix a totally ordered set $(T,\leq)$.

\subsection{Parametrizations}

\begin{definition} \label{def:GenParametrization}
	Let $\mathsf{X}$ be a measurable space and $\mu \in \PZ(\mathsf{X})$. A family $(\mu ^{\alpha})_{\alpha \in [0,1]}$ in $\MZ(\mathsf{X})$ is called a parametrization of $\mu$ if 
	\begin{enumerate}
		\item [(i)] $\mu ^{\alpha}(\mathsf{X}) = \alpha$ for all $\alpha \in [0,1]$,
		\item [(ii)] $\mu ^{\alpha} \leqp \mu ^{\alpha'}$ for all $\alpha \leq \alpha'$ in $[0,1]$ and
		\item [(iii)] $\mu ^1 = \mu$.
	\end{enumerate}
\end{definition}

Each parametrization of a probability measure $\mu$ can be seen as an explicit coupling of $\mu$ with a uniformly distributed random variable on $[0,1]$ that is added to the probability space. Recall, that $\lambda$ denotes the Lebesgue measure on $[0,1]$.

\begin{remark} \label{rem:ParamToCpl}
	Let $\mathsf{X}$ be a measurable space, $\mu \in \PZ(\mathsf{X})$ and $(\mu ^{\alpha})_{\alpha \in [0,1]}$ a family of finite measures on $\mathsf{X}$. The following are equivalent:
	\begin{enumerate}
		\item [(i)] The family $(\mu ^{\alpha})_{\alpha \in [0,1]}$ is a parametrization of $\mu$.
		\item [(ii)] There exists a coupling $\xi$ of $\lambda$ and $\mu$ with $\xi ([0,\alpha] \times B) = \mu ^{\alpha}(B)$ for all $\alpha \in [0,1]$ and measurable sets $B \subset E$.
	\end{enumerate}
	Clearly, the coupling $\xi$ is uniquely determined by $(\mu ^{\alpha})_{\alpha \in [0,1]}$ and vice versa.
\end{remark}

\begin{lemma} \label{lemma:ParamDetermined}
	Let $\mathsf{X}$ be a measurable space, $\mu, \nu \in \PZ(\mathsf{X})$, $(\mu ^{\alpha})_{\alpha \in [0,1]}$ a parametrization of $\mu$ and $(\nu ^{\alpha})_{\alpha \in [0,1]}$ a parametrization of $\nu$. If $\mu ^{\alpha}  = \nu ^{\alpha}$ for all $\alpha$ in a dense subset $A$ of $[0,1]$, then $\mu ^{\alpha} = \nu ^{\alpha}$ for all $\alpha \in [0,1]$ and, in particular, $\mu = \nu$.
\end{lemma}

\begin{proof}
	Let $\alpha \in (0,1]$ and $(\alpha _n)_{n \in \mathbb{N}}$ a sequence in $A$ with $\alpha_n \uparrow \alpha$. Remark \ref{rem:ParamToCpl} yields that 
	\begin{equation*}
	\mu ^{\alpha}(B) = \lim _{n \rightarrow \infty} \mu ^{\alpha_n}(B) \hspace*{1cm} \text{and} \hspace*{1cm} \nu ^{\alpha}(B) = \lim _{n \rightarrow \infty} \nu ^{\alpha_n}(B)
	\end{equation*}
	for all measurable sets $B \subset \mathsf{X}$.
\end{proof}

Given a specific peacock, the degree of freedom in our construction (Theorem \ref{thm:GenExist}) is the choice of a parametrization of the initial marginal. Hence, our primary motivation to consider general  (non-interval based) parametrizations of probability measures is to enlarge the set of possible input choices. For instance, an initial distribution that contains atoms cannot satisfy condition (i) in Corollary \ref{thm:intro1}. The concept of parametrizations allows us to break these atoms into a continuum of quantiles.

\subsubsection{Convex parametrizations} \label{ss:convex_para}

\begin{definition}
	Let $\mu$ be in $\PO(\mathbb{R})$.
	A parametrization $(\nu ^{\alpha})_{\alpha \in [0,1]}$ of $\mu$ is said to be $\leq  _{c,s}$-convex if
	\begin{equation} \label{eq:csConvex}
	\frac{\nu ^{\alpha_2} - \nu ^{\alpha_1}}{\alpha_2 - \alpha_1} \leqcs  \frac{\nu ^{\alpha_3} - \nu ^{\alpha_2}}{\alpha_3 - \alpha_2}
	\end{equation}
	for all $\alpha _1 < \alpha _2  < \alpha _3 $ in $[0,1]$.
\end{definition}

Since both sides of inequality \eqref{eq:csConvex} can be interpreted as the slopes of secant lines of $\alpha \mapsto \nu^{\alpha}$ on $[\alpha_1, \alpha_2]$ and $[\alpha_2, \alpha_3]$, the inequality yields that $\alpha \mapsto \nu ^{\alpha}$ is convex in this sense.
Moreover, property \eqref{eq:csConvex} is equivalent to $\alpha\mapsto \int \varphi \mathrm{d}\nu^\alpha$ being convex for all increasing convex functions $\varphi: \mathbb{R} \rightarrow \mathbb{R}$.

\begin{lemma} \label{lemma:ExConvParam}
	Let $\mu \in \PO(\mathbb{R})$ and $(\nu ^{\alpha})_{\alpha \in [0,1]}$ be a parametrization of $\mu$. If there exists a sequence of nested intervals $(I_{\alpha})_{\alpha \in [0,1]}$ in $\mathbb{R}$ such that 
	\begin{enumerate}
		\item [(i)]  $\sup I_{\alpha} < + \infty$ and $\mathrm{supp}(\nu ^{\alpha}) \subset \overline{I_{\alpha}}$ for all $\alpha \in [0,1)$,
		\item [(ii)] $\mathrm{supp}(\nu ^{\alpha_2} - \nu ^{\alpha_1}) \subset \overline{{I_{\alpha_1}}^c}$ for all $\alpha_1 < \alpha _2$ in $[0,1]$ and
		\item [(iii)] $\alpha \mapsto  \int _{\mathbb{R}} y \de \nu ^{\alpha} (y)$ is convex,
	\end{enumerate}
	then the parametrization $(\nu ^{\alpha})_{\alpha \in [0,1]}$ is  $\leq_{c,s}$-convex.
\end{lemma}

\begin{proof}
	For all $\alpha_1 < \alpha_2 < \alpha _3$ in $[0,1]$ the measure $\bar{\nu} _{1,2} := \frac{\nu ^{\alpha _2} - \nu ^{\alpha_1}}{\alpha_2 - \alpha_1}$
	is concentrated on $\overline{I_{\alpha_2}}$ by (i) and $ 	\bar{\nu} _{2,3} := \frac{\nu ^{\alpha_3} - \nu ^{\alpha_2}}{\alpha_3 - \alpha_2} $
	is concentrated on the closure of the complement $\overline{{I_{\alpha_2}} ^c}$ by (ii). Moreover, both of theses measures are  probability measures and their barycenters satisfy
	\begin{eqnarray*}
		\int _{\mathbb{R}} y \de \bar{\nu} _{1,2} (y) \leq \int _{\mathbb{R}} y \de \bar{\nu} _{2,3}  (y)
	\end{eqnarray*}
	because $\alpha \mapsto \int _{\mathbb{R}} y \de \nu^{\alpha} (y)$ is convex by property (iii). 
	Let $\varphi: \mathbb{R} \rightarrow \mathbb{R}$ be a convex increasing function. Since $I_{\alpha_2}$ is bounded from above, there exists an increasing affine function $l(y) = ay + b$ with $\varphi \leq l$ on $\overline{I_{\alpha_2}}$ and $\varphi \geq l$ on $\overline{{I_{\alpha_2}}^c}$. Thus, by using $a\geq 0$ we obtain 
	\begin{align*}
		\int _{\mathbb{R}} \varphi \de \bar{\nu} _{1,2}   &\leq 	\int _{\mathbb{R}} l \de \bar{\nu} _{1,2}  = a \int _{\mathbb{R}} y \de \bar{\nu} _{1,2}(y) + b \\
		&\leq a \int _{\mathbb{R}} y \de \bar{\nu} _{2,3}(y) + b = \int _{\mathbb{R}} l \de \bar{\nu} _{2,3} \leq \int _{\mathbb{R}} \varphi \de \bar{\nu} _{2,3} .\qedhere
	\end{align*}
\end{proof}

\begin{lemma}\label{lem:para_convex}
	The following parametrizations of $\mu \in \PO(\mathbb{R})$ are $\leqcs$-convex:
	\begin{enumerate}
		\item[(i)] The left-curtain parametrization $(\nu ^{\alpha}_{\lc})_{\alpha \in [0,1]}$ with
		$$\nu _{\lc} ^{\alpha} = \mu _{|(- \infty, F_{\mu} ^{-1}(\alpha))}  + (\alpha - \mu [(- \infty, F_{\mu} ^{-1}(\alpha)]]) \delta _{F_{\mu} ^{-1}(\alpha)}.$$
		\item[(ii)] The middle-curtain parametrization $(\nu ^{\alpha}_{\mc})_{\alpha \in [0,1]}$ with
		$$\nu _{\mc} ^{\alpha} = \mu _{| (q_{\alpha}, q_{\alpha}')}  + c_{\alpha} \delta _{q_{\alpha}} + c'_{\alpha} \delta _{q_{\alpha}'}$$ 
		for $q_{\alpha} \leq q_{\alpha}'$ in $\mathbb{R}$ and $c_{\alpha}, c_{\alpha}' \in [0,1]$ such that $\nu _{\text{mc}} ^{\alpha}(\mathbb{R}) = \alpha$ and $\int y \de 
		\nu _{\text{mc}} ^{\alpha} = \int y \de \mu$.
		\item[(iii)] The sunset parametrization $(\nu ^{\alpha} _{\sun})_{\alpha \in [0,1]}$ with $\nu _{\text{sun}} ^{\alpha} = \alpha  \mu$.
	\end{enumerate}
\end{lemma}

\begin{proof}
	For item (iii), we have $\frac{1}{\alpha_2 - \alpha_1}(\nu _{\alpha_2} - \nu _{\alpha_1}) = \mu$ for all $\alpha_1 < \alpha_2$ in $[0,1]$.
	
	For item (i) and item (ii) we can apply Lemma \ref{lemma:ExConvParam} because both $\alpha \mapsto \int _{\mathbb{R}} y \de \nu ^{\alpha} _{\lc}(y)$ and $\alpha \mapsto \int _{\mathbb{R}} y \de \nu ^{\alpha} _{\text{mc}}(y)$ are convex functions. Indeed, it holds
	\begin{equation*}
	\frac{1}{\alpha_2  - \alpha_1} \left( \int _{\mathbb{R}} y \de \nu ^{\alpha_2} _{\lc} -  \int _{\mathbb{R}} y \de \nu ^{\alpha_1} _{\lc} \right) \leq F_{\mu} ^{-1} (\alpha_2) \leq 	\frac{1}{\alpha_3 - \alpha_2} \left( \int _{\mathbb{R}} y \de \nu ^{\alpha_3} _{\lc} -  \int _{\mathbb{R}} y \de \nu ^{\alpha_2} _{\lc} \right)
	\end{equation*}
	for all $\alpha _1 < \alpha_2 < \alpha_3$ in $[0,1]$ and $\alpha \mapsto \int _{\mathbb{R}} y \de \nu ^{\alpha} _{\text{mc}}(y) = \int _{\mathbb{R}} y \de \mu$ is constant.
\end{proof}

\subsubsection{Probability measures on $\mathbb{R}^T$} \label{sec:ParamProbMeas}

To be able to describe the evolution of a submeasure of the initial measure under some $\pi \in \MM_T$ we need to consider parametrizations of $\pi$ as well.

\begin{definition} \label{def:Parametrization}
	Let $\pi$ in $\PO(\mathbb{R}^T)$. 
	\begin{enumerate}
		\item [(i)] Let $(\nu ^{\alpha})_{\alpha \in [0,1]}$ be a parametrization 
		of the initial marginal $\mathrm{Law}_{\pi}(X_0)$. A family $(\pi ^{\alpha})_{\alpha \in [0,1]}$ in $\MO(\mathbb{R}^T)$ is called a parametrization of $\pi$ w.r.t.\ $(\nu ^{\alpha})_{\alpha \in [0,1]}$ if $(\pi ^{\alpha})_{\alpha \in [0,1]}$ is a parametrization of $\pi$ with
		$\pi ^{\alpha}(X_0 \in \cdot) = \nu ^{\alpha}$ for all $\alpha \in [0,1]$.
		\item [(ii)] A parametrization $(\pi^{\alpha})_{\alpha \in [0,1]}$ of $\pi$ is called a martingale parametrization of $\pi$, if $\frac{1}{\alpha} \pi ^{\alpha} \in \mathsf{M}_T$	for all $\alpha \in (0,1]$. 
	\end{enumerate}
\end{definition}

\begin{remark} \label{rem:uniqueParametrization}	
	Let $\pi$ be in $\PO(\mathbb{R}^T)$ and $(\nu ^{\alpha})_{\alpha \in [0,1]}$ be a parametrization of the initial marginal $\mathrm{Law}_{\pi}(X_0)$. Moreover, let $(\pi ^{\alpha})_{\alpha \in [0,1]}$ be a parametrization of $\pi$ w.r.t.\ $(\nu ^{\alpha})_{\alpha \in [0,1]}$. It is not difficult to prove that for any $\alpha \in [0,1]$, for which there exists a Borel set $A \subset \mathbb{R}$ with $\nu ^{\alpha} = (\mu _0)_{|A}$, it holds $\pi^{\alpha} = \alpha \mathrm{Law}_{\pi}(X | X_0 \in A)$.
\end{remark}

Remark \ref{rem:uniqueParametrization} suggests that we can interpret $\pi ^{\alpha}$ as the way $\nu ^{\alpha}$ is transported under $\pi$, i.e.\ we can see $\pi^{\alpha} (X_t\in \cdot)$ as a formal version of `$\alpha \mathrm{Law}_\pi(X_t|X_0 \in \nu ^{\alpha})$'. 
However, one has to be careful with this informal notation because contrarily to $\alpha \mathrm{Law}_{\pi}(X | X_0 \in A)$ that is uniquely defined, there can be several parametrizations $(\pi^\alpha)_{\alpha\in [0,1]}$ of the same measure $\pi$ w.r.t.\ $(\nu^\alpha)_{\alpha\in [0,1]}$, each giving another meaning to $\alpha \mathrm{Law}_\pi(X_t|X_0 \in \nu ^{\alpha})$. This is illustrated in Example \ref{ex:two_para} just below. Hence, the correct interpretation of the existence of a parametrization $(\pi ^{\alpha})_{\alpha \in [0,1]}$ of $\pi$ w.r.t.\ $(\nu ^{\alpha})_{\alpha \in [0,1]}$ is that $\nu ^{\alpha} = \mathrm{Law}_{\pi ^{\alpha}}(X_0)$ \textit{can} be transported  according to $\pi^{\alpha}$ as part of the dynamic given by $\pi$.

\begin{example}\label{ex:two_para}
	Let $(\mu _t)_{t \geq 0}$ be a peacock, $(\nu ^{\alpha}_{\sun})_{\alpha \in [0,1]}$ be the sunset parametrization of $\mu _0$ and let $\pi \in \PO(\mathbb{R}^{[0,\infty})$ be associated with $(\mu_t)_{t \geq 0}$. 
	Let $(\pi ^{\alpha})_{\alpha \in [0,1]}$ be a (martingale) parametrization of $\pi$ w.r.t.\ $(\nu ^{\alpha} _{\sun})_{\alpha \in [0,1]}$. 
	For $\alpha\in[0,1]$,  set $\rho ^{\alpha} = \pi - \pi ^{1 - \alpha}$. Assume that there is $\bar\alpha\in (0,1)$ such that $\pi^{\bar\alpha}\neq \rho^{\bar\alpha}$. Then, the family $(\rho ^{\alpha})_{\alpha \in [0,1]}$ is again a (martingale) parametrization of $\pi$ w.r.t.\ $(\nu ^{\alpha} _{\sun})_{\alpha \in [0,1]}$ but different from $(\pi^\alpha)_{\alpha\in[0,1]}$.
	For a concrete example one can choose $(\mu_t)_{t \geq 0}$, $\pi$ and $(\pi ^{\alpha})_{\alpha \in [0,1]}$ as in Example \ref{expl:SunSM}. 
\end{example}

\begin{remark}
	In the last example the assumption that there is a martingale parametrization satisfying $\pi^{\bar\alpha}\neq\rho^{\bar\alpha}$ for some $\bar\alpha\in [0,1]$ is always satisfied as soon as the peacock is not NSI (see $\S$ \ref{sec:NSI}).  NSI peacocks are  extremal elements in the set of peacocks with fixed initial marginal so that they are in a certain sense rare (see Lemma \ref{lemma:CharacNSI}).
\end{remark}

\subsection{Shadows} \label{sec:Shadow}

The concept of the shadow of a measure $\nu$ through a family of finite measures is at the center of our construction (cf.\ Section \eqref{eq:ShadowCurveintro}). After recalling previous results of Beiglb\"ock and Juillet \cite{BeJu16} and Nutz, Stebegg and Tan \cite{NuStTa17} for simple and finitely obstructed shadows, we establish in Proposition \ref{prop:GeneralSchadow} the existence of an obstructed shadow in the generality required for our setup.

\subsubsection{The simple shadows} \label{sec:SimpleShadow}

We start by recalling the original concept of (simple) shadows developed in \cite{BeJu16}. Given two finite measures $\nu$ and $\mu$ on $\mathbb{R}$, the shadow of $\nu$ in $\mu$ is defined as the minimum in convex order among all submeasure of $\mu$ that are in convex order larger than $\nu$. More precisely:

\begin{proposition}[cf.\ {\cite[Lemma 4.6]{BeJu16}}] \label{prop:SimpleShadow}
	Let $\nu, \mu \in \MO(\mathbb{R})$ satisfying $\nu \leqcp \mu$. There exists a unique finite measure $\eta$ such that
	\begin{enumerate}
		\item [(i)] $\nu \leqc \eta$,
		\item [(ii)] $\eta \leqp \mu$ and
		\item [(iii)] for all $\eta ' \in \MO(\mathbb{R})$ with $\nu \leqc \eta ' \leqp \mu$ it holds $\eta \leqc \eta '$.
	\end{enumerate}
	The measure $\eta$ is denoted by $\shadow{\mu}{\nu}$ and called the shadow of $\nu$ in $\mu$.
\end{proposition}

For a detailed proof we refer to \cite{BeJu16}. We only stress that the proof is based on potential functions and the potential function of the shadow has an explicit expression in terms of the potential functions of $\nu$ and $\mu$ stated in the following lemma:

\begin{lemma}
	\label{lemma:shadowPF}
	Let $\nu, \mu \in \MO(\mathbb{R})$ with $\nu \leqcp \mu$. It holds
	\begin{equation*}
	U(\shadow{\mu}{\nu}) = U(\mu) - \mathrm{conv}(U(\mu) - U(\nu))
	\end{equation*}
	where $\mathrm{conv}(f)$ denotes the convex hull of a function $f$, i.e. the largest convex function that is pointwise smaller than $f$. Moreover if $\mu\leqc\mu'$ it holds
	\begin{equation*}
	U(\shadow{\mu'}{\nu}) - U(\shadow{\mu}{\nu}) \leq U(\mu') - U(\mu).
	\end{equation*}
\end{lemma}

\begin{proof}
	The first formula has been brought to our attention by Mathias Beiglb\"ock, for a proof we refer to  \cite{BeHoNo22}.
	The second follows from an application of the first identity (for $\mu$ and $\mu'$) together with the inequality $\mathrm{conv}(U(\mu ') - U(\nu)) \geq \mathrm{conv}(U(\mu) - U(\nu))$ since $\mu \leqc \mu'$.
\end{proof}

This result allows one to explicitly calculate shadows. However, in simple situations one does not need to calculate potentials as the following example shows.

\begin{example} \label{expl:ShadowOfAtom} 
	Let $\nu \leqcp \mu$ in $\mathcal{M}^1(\mathbb{R})$ such that $\mu$ is atomless. If
	\begin{enumerate}
		\item[(i)] $\nu = \alpha \delta _x$ for some $\alpha \geq 0$ and $x \in \mathbb{R}$ or
		\item[(ii)] there exists an interval $I \subset \mathbb{R}$ with $\mathrm{supp}(\nu) \subset I$ and $\mathrm{supp}(\mu)\subset I^c$,
	\end{enumerate}
	then there exists an interval $J$ such that $\shadow{\mu}{\nu} = \mu _{|J}$. See \cite[Example 4.7]{BeJu16} for the proof of (i). For (ii) consider the shadow of $\alpha \delta _x$ in $\mu$ as in (i) where $\alpha,x$ are the mass and the barycenter of $\nu$, respectively. 
	Since $\mathrm{supp}(\nu) \subset I$, it holds $\nu \leq _c \shadow{\mu}{\alpha \delta _x}$ (see \cite[Example 4.2]{BeJu16}) and thus $\shadow{\mu}{\nu}=\shadow{\mu}{\alpha \delta _x} = \mu _J$ for some interval $J$. 
	
		For measures $\mu$ that posses atoms these examples can easily be adapted adding to $\mu _{|J}$ one or two atomic masses at the end points of the interval $J$. 	
\end{example}

The  calculation rule (ii) in Proposition \ref{prop:propertiesShadow} below is one of the key tools to deal with  shadows. Apart from its importance for proofs of more advanced properties of the shadow it provides us together with Example \ref{expl:ShadowOfAtom} with an alternative (to Lemma \ref{lemma:shadowPF}) and simple way  to calculate or approximate shadows iteratively in concrete examples.

\begin{proposition}
	\label{prop:propertiesShadow}
	Let $\nu \leqcp \mu$ in $\mathcal{M}^1(\mathbb{R})$.
	\begin{enumerate}
		\item [(i)] For all $\alpha > 0$ it holds $\alpha \nu \leqcp \alpha \mu$ and $\shadow{\alpha \mu}{\alpha \nu} = \alpha \shadow{\mu}{\nu}$.
		\item [(ii)] For all $\nu _1 + \nu _2 = \nu$ we have $\nu_2 \leqcp \mu - \shadow{\mu}{\nu_1}$ and $\shadow{\mu}{\nu} = \shadow{\mu}{\nu _1} + \shadow{\mu - \shadow{\mu}{\nu _1}}{\nu _2}$.
	\end{enumerate}
\end{proposition}

\begin{proof}
	Item (i) is clear by construction of the shadow. Item (ii) is \cite[Theorem 4.8]{BeJu16}.
\end{proof}

\begin{lemma} \label{lemma:orderShadows}
	Let $\nu, \nu ' \leqcp \mu$.
	\begin{enumerate}
		\item[(i)] If $\nu \leqc \nu '$, then $\shadow{\mu}{\nu} \leqc \shadow{\mu}{\nu'}$.
		\item[(ii)] If $\nu \leqp \nu '$, then $\shadow{\mu}{\nu} \leqp \shadow{\mu}{\nu'}$.
		\item[(iii)] If $\nu \leqcs \nu '$, then $\shadow{\mu}{\nu} \leqcs \shadow{\mu}{\nu'}$.
	\end{enumerate}
\end{lemma}

\begin{proof}
	Item (i) is a direct consequence of the minimality property of $\shadow{\mu}{\nu}$ and (ii) is a direct consequence of Proposition \ref{prop:propertiesShadow} (ii).
	If $\nu \leqcs \nu '$, then similar to the shadow, the set $\{\eta : \nu \leqcs \eta \leqp \mu \}$ has a minimal element w.r.t.\ the convex-stochastic order that we denote by $\eta ^*$ (cf.\ \cite[Lemma 6.2]{NuSt17} for decreasing instead of  increasing functions in the definition of $\leqcs$). The minimality implies both $\eta ^* \leqcs \shadow{\mu}{\nu}$ and $\eta ^* \leqcs \shadow{\mu}{\nu'}$. Moreover, we have
	\begin{equation*}
	\int _{\mathbb{R}} y \de \shadow{\mu}{\nu} (y) =  \int _{\mathbb{R}} y \de \nu (y) \leq   \int _{\mathbb{R}} y \de \eta^* (y).
	\end{equation*}
	Hence, $\eta ^* = \shadow{\mu}{\nu}$ by Lemma \ref{lemma:relationOrders} (iii) and we conclude. 
\end{proof}

\subsubsection{The obstructed shadow} \label{sec:ShadowGeneralConstruction}

We now turn to the definition of obstructed shadows. They can conveniently be constructed as a convex supremum over finitely obstructed shadows that were introduced by Nutz, Stebegg, and Tan in \cite{NuStTa17}.
Recall that $(T, \leq)$ is a totally ordered set. Moreover, we fix a family of measures $(\mu _t)_{t \in T}$ in $\MO(\mathbb{R})$. 
To keep the notation  compact we will
\begin{itemize}
	\item denote $(\mu _t)_{t \in S}$ by $\mu_S$ for all subsets $S \subset T$ and
	\item use the abbreviation  $T_t = \{s \in T : s \leq t \}$.
\end{itemize}
This notation will be used in all following sections. 

\begin{definition} \label{def:LeqCP}
	Let $\nu \in \MO(\mathbb{R})$ and $S \subset T$. We say $\nu \leqcp \mu_S$ if there exists a family $(\eta _t)_{t \in S}$ such that
	\begin{enumerate}
		\item [(i)] $\nu \leqcp \eta _t$ for all $t \in S$,
		\item [(ii)] $\eta _s \leqc \eta _t$ for all $s \leq t$ in $S$ and
		\item [(iii)] $\eta _t \leqp \mu _t$ for all $t \in S$.
	\end{enumerate}
\end{definition}

\begin{remark}
	If $T = \{ \star \}$ is a singleton, Definition \ref{def:LeqCP} coincides with the one of $\leqcp$ in Section \ref{sec:OrderRelPF} by choosing $\eta _{\star} = \nu$. Moreover, if $\nu \leqcp \mu _T$, then $\nu \leqcp \mu_S$ for all $S \subset T$.
\end{remark}

In the case that $T$ is finite, it was observed in \cite{NuStTa17} that one can recursively define an obstructed shadow through finitely many marginals.

\begin{lemma}[{\cite[Lemma 6.7]{NuStTa17}}] \label{lemma:FinteShadowMinimal}
	Let $R = \{r_1 \leq\ldots\leq r_n\}$ be a finite subset of $T$ and $\nu \leqcp \mu _R$. We define inductively the (obstructed) shadow of $\nu$ through $\mu_R$ by
	\begin{equation*}
	\shadow{\mu _{r_1},\ldots, \mu _{r_n}}{\nu} = \shadow{\mu _{r_n}}{\shadow{\mu _{r_1}, \ldots ,\mu _{r_{n-1}}}{\nu}}.
	\end{equation*}
	The measure $\shadow{\mu _{r_1}, \ldots , \mu _{r_n}}{\nu}$ is the unique minimal element of the set
	\begin{equation}
	\left\{ \eta _{r_n} \ \vert \ (\eta _r)_{r \in R} : \nu \leqc \eta _r \leqc \eta _{r'} \leqp \mu _{r'} \textit{ for all } r \leq r' \text{ in }R \right\} 
	\label{eq:minObSh}
	\end{equation}
	w.r.t. the convex order $\leqc$. In particular, $\shadow{\mu _{r_1}, \ldots , \mu _{r_n}}{\nu} \leqp \mu _{r_n}$.
\end{lemma}

\begin{proof}
	This can be easily shown by induction over $n \geq 2$ using Proposition \ref{prop:SimpleShadow} and Lemma \ref{lemma:orderShadows}. Alternatively, see \cite[Lemma 6.7]{NuStTa17}.
\end{proof}

By comparing \eqref{eq:minObSh} with Proposition \ref{prop:SimpleShadow}, we see that $\shadow{\mu _{r_1}, \ldots , \mu _{r_n}}{\nu}$ is the shadow of $\nu$ in $\mu _{r_n}$ obstructed by the additional finitely many marginals $\mu _{r_1}, \ldots , \mu _{r_{n-1}}$. Figure \ref{fig:ObstrShadow} illustrates that an additional obstructing marginal can force the shadow to ``spread out'' (in convex order).
\begin{center}
	\begin{figure}
		\begin{tikzpicture}[scale=0.8]
		\draw[-,dotted] (0,1) node[left]{$1$} -- (5.5,1);
		\draw[-,dotted] (0,-1) node[left]{$-1$} -- (5.5,-1);
		\draw[-,dotted] (0,0) node[left]{$0$} -- (5.5,0);
		\draw[-,dotted] (0,2) node[left]{$2$}-- (5.5,2);
		\draw[-,dotted] (0,-2) node[left]{$-2$} -- (5.5,-2);
		\draw[-,dotted] (0,3) node[left]{$3$} -- (5.5,3);
		\draw[-] (0,-2.5) -- (0,3.5);
		\draw[thick] (0,0) rectangle (1,1);
		\node at (0.5,3.5) {$\nu$};
		\draw[-] (2,-2.5) -- (2,3.5);
		\draw[thick] (2,-1) rectangle (3,0);
		\draw[thick] (2,1) rectangle (3,2);
		\node at (2.5,3.5) {$\mu_1$};
		\draw[-] (4,-2.5) -- (4,3.5);
		\draw[thick] (4,-2) rectangle (4.5,-1);
		\draw[thick] (4,0) rectangle (5,1);
		\draw[thick] (4,2) rectangle (4.5,3);
		\node at (4.5,3.5) {$\mu_2$};
		\draw[red,pattern=north east lines, pattern color=red] (2.02,-0.5) rectangle (2.98,-0.02);
		\draw[red,pattern=north east lines, pattern color=red] (2.02,1.02) rectangle (2.98,1.5);
		\node[red] at (3.2,-1.5) {$\mathcal{S}^{\mu _1}(\nu)$};
		\draw[red,pattern=north east lines, pattern color=red] (4.02,-1.317) rectangle (4.48,-1.02);
		\draw[red,pattern=north east lines, pattern color=red] (4.02,0.02) rectangle (4.98,0.342);
		\draw[red,pattern=north east lines, pattern color=red] (4.02,0.98) rectangle (4.98,1-0.342);
		\draw[red,pattern=north east lines, pattern color=red] (4.02,2.317) rectangle (4.48,2.02);
		\node[red] at (5.5,-1.5) {$\mathcal{S}^{\mu _1,\mu _2}(\nu)$};
		
		\draw[-,dotted] (7.5,1) node[left]{$1$} -- (11,1);
		\draw[-,dotted] (7.5,-1) node[left]{$-1$} -- (11,-1);
		\draw[-,dotted] (7.5,0) node[left]{$0$} -- (11,0);
		\draw[-,dotted] (7.5,2) node[left]{$2$}-- (11,2);
		\draw[-,dotted] (7.5,-2) node[left]{$-2$} -- (11,-2);
		\draw[-,dotted] (7.5,3) node[left]{$3$} -- (11,3);
		\draw[-] (7.5,-2.5) -- (7.5,3.5);
		\draw[thick] (7.5,0) rectangle (8.5,1);
		\node at (8,3.5) {$\nu$};
		\draw[-] (9.5,-2.5) -- (9.5,3.5);
		\draw[thick] (9.5,-2) rectangle (10,-1);
		\draw[thick] (9.5,0) rectangle (10.5,1);
		\draw[thick] (9.5,2) rectangle (10,3);
		\node at (10.5,3.5) {$\mu_2$};
		\draw[red,pattern=north east lines, pattern color=red] (9.52,0.02) rectangle (10.48,0.98);
		\node[red] at (10.5,-0.5) {$\mathcal{S}^{\mu _2}(\nu)$};
		\end{tikzpicture}
		
		\caption{The striped red area represents the obstructed shadow of $\nu$ in $\mu _1$ and $(\mu _1,\mu _2)$ (left) and the simple shadow of $\nu$ in $\mu _2$ (right).}\label{fig:ObstrShadow}
	\end{figure}
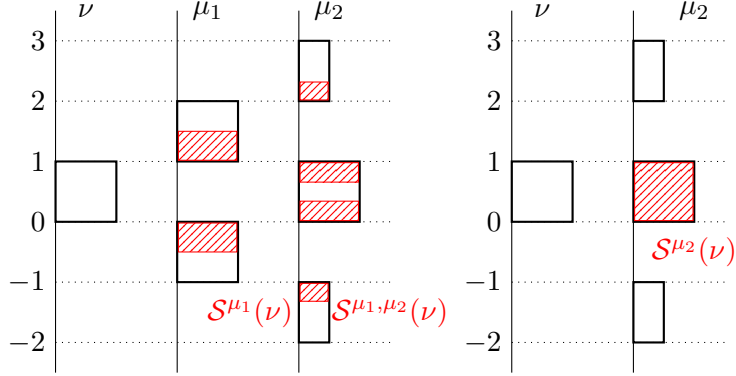
\end{center}

Taking the convex supremum over all choices of finite subsets $R\subset T$ yields the obstructed shadow of $\nu$ through $\mu_T$.

\begin{proposition} \label{prop:GeneralSchadow}
	Let $\nu \in \MO(\mathbb{R})$ with $\nu \leqcp \mu_T$ and suppose there exists $\theta \in \MO(\mathbb{R})$ such that $\mu_t \leqcp \theta$ for every $t\in T$. Then, the set 
	\begin{equation*}
	\{ \shadow{\mu_{r_1}, \ldots, \mu _{r_n}}{\nu} : \{r_1 < \ldots < r_n\} \subset T, \ n \geq 1\}
	\end{equation*}	
	admits a convex supremum. This is called the (general obstructed) shadow of $\nu$ through $\mu_T$  and is denoted by $\shadow{\mu_T}{\nu}$. Moreover, there exists a sequence $(R_n)_{n \in \mathbb{N}}$ of nested finite subsets of $T$ such that $(\shadow{\mu_{R_n}}{\nu})_{n \in \mathbb{N}}$ converges to $\shadow{\mu_T}{\nu}$ under $\TO$.
\end{proposition}

\begin{proof}
	By Lemma \ref{lemma:FinteShadowMinimal}, $\shadow{\mu _{r_1}, \ldots, \mu _{r_n}}{\nu}$ is increasing in convex order, if we add additional marginals as obstructions. Thus, for all finite $R=\{r_1,\ldots,r_n\}\subset T$ it holds that  $\shadow{\mu _R}{\nu} = \shadow{\mu _{r_1}, \ldots, \mu _{r_n}}{\nu}$, and therefore $\shadow{\mu _R}{\nu} \leqp \mu _{r_n} \leqcp \theta$. 
	Consequently, the convex supremum exists by Proposition \ref{prop:ExistenceCSup} (ii).  
	
	Again by Lemma \ref{lemma:FinteShadowMinimal},  it holds $\shadow{\mu _{R_1}}{\nu} \leqc \shadow{\mu_{R_1 \cup R_2}}{\nu}$ and $\shadow{\mu _{R_2}}{\nu} \leqc \shadow{\mu_{R_1 \cup R_2}}{\nu}$ for all finite $R_1, R_2 \subset T$. Thus,  by Lemma \ref{lemma:CSupApprox}, there exists a sequence of finite sets $(R_n)_{n \in \mathbb{N}}$ such that $(\shadow{\mu _{R_n}}{\nu})_{n \in \mathbb{N}}$ converges under $\TO$ to the convex supremum. Moreover, one can easily show that $R_n' = \bigcup _{i = 1} ^n R_i$ is a nested sequence of finite subsets of $T$ for which $(\shadow{\mu _{R'_n}}{\nu})_{n \in \mathbb{N}}$ converges to $\shadow{\mu _T}{\nu}$ because $\shadow{\mu _{R_n}}{\nu} \leqc \shadow{\mu _{R'_n}}{\nu} \leqc \shadow{\mu _{T}}{\nu}$ for all $n \in \mathbb{N}$ and the convex order is preserved under convergence w.r.t.\ $\TO$ (see Lemma \ref{lemma:preservingOrder}).
\end{proof}

Proposition \ref{prop:GeneralSchadow} extends the definitions in Lemma \ref{lemma:FinteShadowMinimal} and Proposition \ref{prop:SimpleShadow} in two ways: Firstly, it allows for infinitely, even uncountably, many obstructions. Secondly, the family $\mu_T$ does not have to be increasing in convex order.

For the remaining part of Section \ref{sec:ShadowGeneralConstruction} we will {\it always assume that there exists some $\theta \in \MO(\mathbb{R})$ with $\mu _t \leqcp \theta$ for all $t \in T$}. If $T$ has a maximal element and $\mu _T$ is a peacock this property is automatically satisfied. The following lemma collects some important consequences of Proposition \ref{prop:GeneralSchadow}:

\begin{lemma} \label{lemma:ShadowInclusion}
	Let $\nu \leqcp \mu _T$.
	\begin{enumerate}
		\item[(i)] For all $S \subset S' \subset T$ we have $\shadow{\mu _S}{\nu} \leqc \shadow{\mu _{S'}}{\nu}$.
		\item[(ii)] If $(S_n)_{n \in \mathbb{N}}$ is a sequence of subsets of $T$ such that $(\shadow{\mu_{S_n}}{\nu})_{n \in \mathbb{N}}$ converges to $\shadow{\mu _T}{\nu}$ under $\TO$, the same holds for any sequence $(S_n')_{n \in \mathbb{N}}$ of subsets of $T$ with $S_n \subset S_n'$ for all $n \in \mathbb{N}$.
		\item[(iii)] Let $(S_n)_{n \in \mathbb{N}}$ be a nested sequence of sets with $T = \bigcup _{n \in \mathbb{N}} S_n$. Then $(\shadow{\mu _{S_n}}{\nu})_{n \in \mathbb{N}}$ converges to $\shadow{\mu _T}{\nu}$ under $\TO$.
		\item[(iv)] It holds $\shadow{\mu _T}{\nu} = \mathrm{Csup} \left\{ \shadow{\mu _{T_u}}{\nu} : u \in T  \right\}$.	
		\item[(v)] For all $t \in T$ it holds $\shadow{\mu _{T_t}}{\nu} \leqp \mu _t$.
	\end{enumerate}
\end{lemma}

\begin{proof} 
	For the case of finite subsets $S,S',S_n,S'_n$ of $T$, we have already shown items (i) and (ii) in the proof of Proposition \ref{prop:GeneralSchadow}.  
	The items (i)--(v) are simple consequences of these two properties together with Proposition \ref{prop:GeneralSchadow} and Lemma \ref{lemma:preservingOrder}.
\end{proof}

Proposition \ref{prop:GeneralSchadow} does not require that $T$ admits a maximal element.
However, if such a maximal element exists, then we can recover an analogue of \eqref{eq:minObSh}:

\begin{proposition} \label{prop:MaximalElement}
	Let $\nu \leqcp \mu _T$. For all $u \in T$, 
	\begin{equation} \label{eq:ShadowAsCInf}
	\shadow{\mu_{T_u}}{\nu} = \mathrm{Cinf} \left\{ \eta _u : (\eta _t)_{t \in T_u} \textit{ with } \nu \leqc \eta _s \leqc \eta _t \leqp \mu _t  \textit{ for all } s \leq t \textit{ in }T_u \right\}
	\end{equation}
	and the infimum is attained by the family $(\eta_t)_{t\in T_u}=(\shadow{\mu_{T_t}}{\nu})_{t \in T_u}$.
\end{proposition}

\begin{proof}
	Set $\eta _t = \shadow{\mu_{T_t}}{\nu}$ for all $t \in T$. We know that $\nu \leqc \shadow{\mu_{T_t}}{\nu}$. Lemma \ref{lemma:ShadowInclusion} (i) shows that $\shadow{\mu_{T_s}}{\nu} \leqc \shadow{\mu_{T_t}}{\nu}$ for all $s \leq t$ and we have  $\shadow{\mu_{T_t}}{\nu} \leqp \mu _t$ by Lemma \ref{lemma:ShadowInclusion} (v). Thus, $\shadow{\mu _{T_u}}{\nu}$ has to be, in convex order, greater than or equal to the convex infimum on the right-hand side of \eqref{eq:ShadowAsCInf}  for all $u \in T$.
	
	Let $u \in T$ and take a sequence $(R_n)_{n \in \mathbb{N}}$ of finite subsets of $T_u$ given by Proposition \ref{prop:GeneralSchadow} such that $(\shadow{\mu_{R_n}}{\nu})_{n \in \mathbb{N}}$ converges to $\shadow{\mu_{T_u}}{\nu}$. Lemma \ref{lemma:ShadowInclusion} (ii) yields that for $R'_n = R_n \cup \{u\}$ the sequence of corresponding shadows converges to $\shadow{\mu_{T_u}}{\nu}$ as well. Any family $(\eta' _t)_{t \in T_u}$ with $\eta' \leqc \eta' _s \leqc \eta' _t \leqp \mu _t$ for all $s \leq t$ in $T_u$ satisfies
	\begin{eqnarray*}
		\shadow{\mu_{R'_n}}{\nu} = \mathrm{Cinf} \ \left\{ \tilde{\eta}_{u} \ \vert \ \exists (\tilde{\eta}_t)_{t \in R_n'} : \nu \leqc \tilde{\eta} _s \leqc \tilde{\eta} _t \leqp \mu _t  \textit{ for all } s \leq t \textit{ in }R_n'    \right\} \leqc \eta' _u
	\end{eqnarray*}
	where the equality is due to Lemma \ref{lemma:FinteShadowMinimal}. Passing to the limit under $\TO$, shows that  $\shadow{\mu_{T_u}}{\nu}$ is smaller in convex order than the right-hand side of \eqref{eq:ShadowAsCInf} by Lemma \ref{lemma:preservingOrder}.
\end{proof}

The following lemma generalizes Lemma \ref{lemma:orderShadows} to obstructed shadows.
\begin{lemma} \label{lemma:GenShadowOrder}
	Let $\nu, \nu ' \leqcp \mu_T$.
	\begin{enumerate}
		\item[(i)] If $\nu \leqc \nu '$, then $\shadow{\mu_T}{\nu} \leqc \shadow{\mu_T}{\nu'}$.
		\item[(ii)] If $\nu \leqp \nu '$, then $\shadow{\mu_T}{\nu} \leqp \shadow{\mu_T}{\nu'}$.
		\item[(iii)] If $\nu \leqcs \nu '$, then $\shadow{\mu_T}{\nu} \leqcs \shadow{\mu_T}{\nu'}$.
	\end{enumerate}
	Moreover, for any peacock $\mu '_T$ we have
	\begin{enumerate}
		\item [(iv)] if  $\mu _t \leq _+ \mu'_t$ for all $t \in T$, then $\nu \leqcp \mu _T ' $ and $\shadow{\mu' _T} {\nu} \leq _c \shadow{\mu _T} {\nu}$.
	\end{enumerate}
\end{lemma}

\begin{proof}
	Item (i)-(iii):
	By Proposition \ref{prop:GeneralSchadow} and Lemma \ref{lemma:ShadowInclusion} (ii), we can find a sequence of nested finite sets $(R_n)_{n \in \mathbb{N}}$ such that both $(\shadow{\mu _{R_n}}{\nu})_{n \in \mathbb{N}}$ converges weakly to $\shadow{\mu_T}{\nu}$ and $(\shadow{\mu _{R_n}}{\nu'})_{n \in \mathbb{N}}$ converges weakly to $\shadow{\mu_T}{\nu'}$. In any of the three cases we get the desired relation between $(\shadow{\mu _{R_n}}{\nu})_{n \in \mathbb{N}}$ and $(\shadow{\mu _{R_n}}{\nu'})_{n \in \mathbb{N}}$ by inductively applying Lemma \ref{lemma:orderShadows}. Since all of the three order relations are preserved under convergence in $\TO$ by Lemma \ref{lemma:preservingOrder}, we have shown the claim.
	
	Item (iv) is an immediate consequence of Proposition \ref{prop:MaximalElement} and Lemma \ref{lemma:ShadowInclusion} (iv).
\end{proof}

\begin{lemma} \label{lemma:ShadowFactor}
	Let $\nu \leqcp \mu_T$  and $\alpha > 0$.  Then $\alpha \nu \leqcp \alpha \mu_T$ and $\shadow{\alpha \mu _T}{\alpha \nu} = \alpha \shadow{\mu _T}{\nu}$,
	i.e.\ the convex supremum is positively $1$-homogeneous.
\end{lemma}

\begin{proof}
	For all $\alpha > 0$ and all $\eta, \eta ' \in \MO(\mathbb{R})$ it holds $\eta \leqc \eta '$ if and only if $\alpha \eta \leqc \alpha \eta '$. 
\end{proof}

\begin{proposition} \label{prop:ShadowAssoc}
	Let $\nu _1, \nu _2 \in \MO(\mathbb{R})$ with $\nu = \nu _1 + \nu _2 \leq_{c,+} \mu_T$. It holds $\nu _1 \leqcp \mu _T$, $\nu _2 \leqcp (\mu _t - \shadow{\mu _{T_t}}{\nu _1})_{t \in T}$ and
	\begin{equation*}
	\shadow{\mu_T}{\nu _1 + \nu _2} = \shadow{\mu_T}{\nu _1} +  \shadow{(\mu _t - \shadow{\mu_{T_t}}{\nu _1})_{t \in T}}{\nu _2}.
	\end{equation*} 
\end{proposition}

\begin{proof}
	First assume that $T$ is finite. In this case the claim follows from applying   Lemma \ref{prop:propertiesShadow} (ii) inductively $|T|$ times.
	
	Now suppose that $T$ has a maximal element, i.e. $T = T_u$ for some $u \in T$. Lemma \ref{lemma:GenShadowOrder} (ii) implies that $(\shadow{\mu _{T_t}}{\nu} - \shadow{\mu _{T_t}}{\nu _1})_{t \in T_u}$ is a well-defined family in $\MO(\mathbb{R})$. 
	We will show that this sequence is increasing in convex order and greater than or equal to $\nu _2$. To this end, let $s,t\in T_u$ with $s\leq t$ and let $(R_n)_{n\in N}$ be a sequence of finite sets such that all of the four sequences  $(\shadow{\mu_{(R_n)_s}}{\nu})_{n \in \mathbb{N}}$, $(\shadow{\mu_{(R_n)_s}}{\nu _1})_{n \in \mathbb{N}}$, $(\shadow{\mu_{(R_n)_t}}{\nu})_{n \in \mathbb{N}}$ and $\shadow{\mu_{(R_n)_t}}{\nu _1}$  converge to  $\shadow{\mu_{T_s}}{\nu}$, $\shadow{\mu_{T_s}}{\nu_1}$, $\shadow{\mu_{T_t}}{\nu}$ and $\shadow{\mu_{T_t}}{\nu_1}$ respectively. Again, this sequence can be constructed by using Proposition \ref{prop:GeneralSchadow} in conjunction with Lemma \ref{lemma:ShadowInclusion} (ii). For all $n \in \mathbb{N}$ we obtain by the initial considerations for finite index sets
	\begin{eqnarray*}
		&\shadow{\mu_{(R_n)_s}}{\nu} - \shadow{\mu_{(R_n)_s}}{\nu _1} = \shadow{(\mu_r - \shadow{\mu_{(R_n)_r}}{\nu_1})_{r \in (R_n)_s}}{\nu _2}\\
		&\leqc \shadow{(\mu_r - \shadow{\mu_{(R_n)_r}}{\nu_1})_{r \in (R_n)_t}}{\nu _2} = \shadow{\mu_{(R_n)_t}}{\nu} - \shadow{\mu_{(R_n)_t}}{\nu _1}.
	\end{eqnarray*}
	Letting $n$ tend to infinity, this proves that $(\shadow{\mu _{T_t}}{\nu} - \shadow{\mu _{T_t}}{\nu _1})_{t \in T_u}$ is increasing in convex order and shows $\nu_2\leqcp  (\shadow{\mu _{T_t}}{\nu} - \shadow{\mu _{T_t}}{\nu _1})_{t \in T_u}$. Since additionally $\shadow{\mu_{T_t}}{\nu} - \shadow{\mu_{T_t}}{\nu _1} \leqp \mu _t - \shadow{\mu_{T_t}}{\nu _1}$ for all $t \in T_u$, Proposition \ref{prop:MaximalElement} yields 
	\begin{equation*}
	\shadow{(\mu _t - \shadow{\mu_{T_t}}{\nu _1})_{t \in T_u}}{\nu _2} \leqc \shadow{\mu_{T_u}}{\nu} - \shadow{\mu_{T_u}}{\nu _1}.
	\end{equation*}
	Similarly, we can apply Proposition \ref{prop:MaximalElement} to see 
	\begin{equation*}
	\shadow{\mu_{T_u}}{\nu} \leqc \shadow{\mu_{T_u}}{\nu _1} +  \shadow{(\mu _t - \shadow{\mu_{T_t}}{\nu _1})_{t \in T_u}}{\nu _2}.
	\end{equation*}
	and therefore both sides are equal. 
	
	In the general case,	by Lemma \ref{lemma:ShadowInclusion} (iv) it holds 	
	\begin{eqnarray*}
		\shadow{\mu_T}{\nu} &=&  \mathrm{Csup} \left\{ \shadow{\mu_{T_u}}{\nu _1} +  \shadow{(\mu _t - \shadow{\mu_{T_t}}{\nu _1})_{t \in T_u}}{\nu _2} : u \in T  \right\} \\
		&=& \mathrm{Csup} \left\{ \shadow{\mu_{T_u}}{\nu _1} :u \in T \right\} + \mathrm{Csup} \left\{  \shadow{(\mu _t - \shadow{\mu_{T_t}}{\nu _1})_{t \in T_u}}{\nu _2} : u \in T  \right\} \\
		&=& \shadow{\mu_T}{\nu _1} +  \shadow{(\mu _t - \shadow{\mu_{T_t}}{\nu _1})_{t \in T}}{\nu _2}
	\end{eqnarray*}
	where the second equality follows from Lemma \ref{lemma:CSupSum} (iii) because both summands are increasing in convex order as $u$ increases (see Lemma \ref{lemma:ShadowInclusion} (i) and Lemma \ref{lemma:GenShadowOrder} (iv)).
\end{proof}

\begin{remark} \label{rem:RestIsPCOC}
	In the previous proof we have shown that $(\shadow{\mu _{T_t}}{\nu} - \shadow{\mu _{T_t}}{\nu _1})_{t \in T}$ is increasing in convex order with $\nu _2 \leqcp (\shadow{\mu _{T_t}}{\nu} - \shadow{\mu _{T_t}}{\nu _1})_{t \in T}$. If $\nu_1 + \nu _2 \leqc \mu _s \leqc \mu _t$ for all $s \leq t$ in $T$, since $\shadow{\mu _{T_t}}{\nu} = \mu_t$, 
	this implies that $(\mu _t - \shadow{\mu _{T_t}}{\nu _1})_{t \in T}$ is increasing in convex order with $\nu _2 \leq _{c} (\mu _t - \shadow{\mu _{T_t}}{\nu _1})_{t \in T}$.
\end{remark}

\subsection{Non self-improvable peacocks} \label{sec:NSI}

After parametrizations and shadows, non self-improvable peacocks are the last  conceptual ingredient that we need for the proof of Theorem \ref{thm:GenExist}.
Recall that $(T, \leq)$ is an abstract totally ordered set with minimal element $0 \in T$ and that we use the notation $T_r = \{s \in T : s \leq r\}$.

\begin{definition}\label{def:pcoc}
	A peacock $(\mu _t)_{t \in T}$ is called non self-improvable (NSI) if for all peacocks $(\eta _t)_{t \in T}$ with $\eta _0 = \mu _0$ and $\eta _t \leqp 2 \mu _t$ for all $t \in T$, it holds $\mu _t \leqc \eta _t$.

\end{definition}

The following lemma explains the term ``non self-improvable''. Indeed, Item (ii) in Lemma \ref{lemma:CharacNSI} shows that an NSI peacock is minimal in convex order for an operation that aims at reducing the peacock $(\mu _t)_{t \in T}$ in convex order at every $t\in T$ by rearranging the mass constrained to be a submeasure of $2\mu_t$ at every $t\in T$.
In this sense NSI peacocks cannot be ``self-improved''.

\begin{lemma} \label{lemma:CharacNSI}
	Let $(\mu _t)_{t \in T}$ be a peacock. The following are equivalent:
	\begin{enumerate}
		\item[(i)] $(\mu _t)_{t \in T}$ is NSI.
		\item[(ii)] For all $t \in T$ it holds $\shadow{(2 \mu _s)_{s \in T_t}}{\mu _0} = \mu _t$.
		\item[(iii)] $(\mu _t)_{t \in T}$ is an extreme point of the convex set
		\begin{equation*}
		K_{\mu _0} = \left\{ (\eta _t)_{t \in T} \ \vert \  (\eta _t)_{t \in T} \textit{ is a peacock with } \eta _0 = \mu _0 \right\}.
		\end{equation*}
	\end{enumerate}	 
\end{lemma}

\begin{proof}
	(i) $\Rightarrow$ (iii):
	Assume that $(\mu _t)_{t \in T}$ is NSI and $(\eta _t)	_{t \in T}$ and $(\eta' _t)	_{t \in T}$ are peacocks with $\eta _0 = \mu _0 = \eta '_0$ and $\mu _t = \frac{1}{2} \eta _t + \frac{1}{2} \eta' _t$ for all $t \in T$. Then both inequalities $\eta _t \leqp  2\mu _t$ and $\eta' _t \leqp  2\mu _t$ hold for all $t \in T$. Hence, $\mu _t \leqc \eta _t$ and $\mu _t \leqc \eta _t '$ by the NSI property.
	Combining these two inequalities with $\eta_t'=  2\mu _t - \eta _t$, yields  $\mu _t \leqc \eta _t \leqc \mu_t$ for all $t \in T$, i.e $\mu_t=\eta_t$. Hence, $(\mu _t)_{t \in T}$ is an extreme point of $K_{\mu _0}$.
	
	(iii) $\Rightarrow$ (ii):
	If $(\mu _t)_{t \in T}$ is an extreme point of $K_{\mu _0}$, applying first Lemma \ref{lemma:ShadowFactor} with $\alpha = 2$ and then Proposition \ref{prop:ShadowAssoc} with $\nu_1 = \nu_2 = \mu _0$, we can rewrite $\mu _t $ as
	\begin{equation*}
	\mu _t = \shadow{(\mu _t)_{s \in T_t}}{\mu _0} = \frac{1}{2} \shadow{(2 \mu _t)_{s \in T_t}}{2 \mu _0} = \frac{1}{2} \shadow{(2 \mu _s)_{s \in T_t}}{\mu _0} + \frac{1}{2} \left( 2 \mu _t - \shadow{(2 \mu _s)_{s \in T_t}}{\mu _0} \right) 
	\end{equation*}
	for all $t \in T$. Both $(\shadow{(2 \mu _t)_{t \in T}}{\mu _0})_{t \in T}$ and $(2 \mu _t - \shadow{(2 \mu _t)_{t \in T}}{\mu _0})_{t \in T}$ are elements of $K_{\mu _0}$ (see Remark \ref{rem:RestIsPCOC}) and thus extremality yields
	\begin{equation*}
	\shadow{(2 \mu _s)_{s \in T_t}}{\mu _0} =  2 \mu _t - \shadow{(2 \mu _s)_{s \in T_t}}{\mu _0}
	\end{equation*}
	for all $t \in T$ and hence $(\mu _t)_{t \in T}$ satisfies (ii).
	
	(ii) $\Rightarrow$ (i):
	Suppose $(\mu _t)_{t \in T}$ satisfies (ii) and $(\eta _t)_{t \in T}$ is a peacock with $\eta _0 = \mu _0$ and $\eta _t \leqp 2 \mu _t$ for all $t \in T$. Then $\mu _t = \shadow{(2 \mu _s)_{s \in T_t}}{\mu _0} \leq _c \eta _t$ for all $t \in T$ by Proposition \ref{prop:MaximalElement} and hence $(\mu _t)_{t \in T}$ is NSI.
\end{proof}

The key feature of non self-improvable peacocks is that there is only one martingale measure associated with such peacocks, see Proposition \ref{prop:NSItoUniq} below.
Even better this martingale measure is necessarily Markov. On the downside, the NSI property is not closed, as we show in Example \ref{expl:ShadowNotCont}.

\begin{lemma} \label{lemma:NSItoMarkov}
	If $(\mu _t)_{t \in T}$ is a peacock that is NSI, then under any $\pi \in \mathsf{M} _T((\mu _t)_{t \in T})$ the canonical process is a Markov process.
\end{lemma}

\begin{proof}
	Assume there exists $\pi \in \MM_T((\mu _t)_{t \in T})$ for which the canonical process is not a Markov process (in the sense of Definition \ref{def:MartSpaces} (ii)), i.e.\ there exist $r < u$ and a Borel set $A \subset \mathbb{R}$ such that $\mathbb{E}_{\pi}[\1_A | \mathcal{F}_r]$ is not $\pi$-a.e.\ equal to $ \mathbb{E}_{\pi}[\1_A | X_r]$. Since $\mathcal{F}_r$ is the product $\sigma$-algebra generated by the family $(\sigma(X_s))_{s \leq r}$,
	there exist $n \in \mathbb{N}$, $0 \leq r_1 < \ldots < r_n \leq r < u$ in $T$  such that 
	\begin{equation} \label{eq:KernelNEqu}
	\pi	[X_u \in A \ \vert X_{r_1}, \ldots, X_{r_n}, X_r] \neq \pi [X_u \in A \ \vert X_r].
	\end{equation}
	For $t \geq r$, let $k_{t}$ be a regular version of $\pi[X_t \in \cdot \ \vert X_{r_1}, \ldots, X_{r_n}, X_r]$ and $k'_{t}$ be a regular version of $\pi [X_t \in \cdot \ \vert X_r]$ which exist because $\mathbb{R}^{n+2}$ and $\mathbb{R}^2$ are Polish spaces. 
	The inequality in \eqref{eq:KernelNEqu} implies that there exists a convex function $\varphi$ s.t. the Borel-set 
	\begin{equation*}
	\left\{ x \in \mathbb{R}^{\{r_1,\ldots,r_n,r\}} : \int _{\mathbb{R}} \varphi(y) \de k_{u}(x,dy) \neq  \int _{\mathbb{R}} \varphi(y) \de k'_{u}(x_r,dy) \right\}
	\end{equation*}	
	has positive mass under	$\pi _{|\{r_1,\ldots,r_n,r\}}$. Suppose there exists $\varepsilon > 0$ such that the  Borel set 
	\begin{equation*}
	D = \left\{ x \in \mathbb{R}^{\{r_1,\ldots,r_n,r\}} : \int _{\mathbb{R}} \varphi(y) \de k_{u}(x,dy) \geq  \int _{\mathbb{R}} \varphi(y) \de k'_{u}(x_r,dy) + \varepsilon \right\}
	\end{equation*}
	has positive mass under $\pi _{|\{r_1,\ldots,r_n,r\}}$. We set
	\begin{equation*}
	\eta _t = \begin{cases}
	\mu _t & t < r \\
	\int _{D^c} k_{t}(x, \cdot) \de \pi _{|\{r_1,\ldots,r_n,r\}}(x) + \int _{D} k'_{t}(x_r, \cdot) \de \pi _{|\{r_1,\ldots,r_n,r\}}(x) & t \geq r
	\end{cases}
	\end{equation*}
	for all $t \in T$. Then $\eta _0 = \mu _0$,   $\eta _t \leqp 2 \mu _t$ for all $t \in T$ and $\eta_s \leqc \eta _t$ for all $s<t \leq r$. For all $r\leq s <t$ and any convex function $\psi$ it holds
	\begin{eqnarray*}
		&& \int \psi \de \eta _t = \mathbb{E}_{\pi}\left[ \mathbb{E}_{\pi} [\psi(X_t) | X_{r_1}, \ldots , X_{r_n}, X_r] \cdot \1_D + \mathbb{E}_{\pi} [\psi(X_t) | X_r] \cdot \1_{D^c}  \right] \\
		&\geq& \mathbb{E}_{\pi}\left[ \mathbb{E}_{\pi} [\psi(X_s) | X_{r_1}, \ldots , X_{r_n}, X_r] \cdot \1_D + \mathbb{E}_{\pi} [\psi(X_s) | X_r] \cdot \1_{D^c}  \right] = \int \psi \de \eta _s
	\end{eqnarray*}
	and thus $(\eta _t)_{t \in T}$ is a peacock.
	Then, since we have for $u>r$
	\begin{equation*}
	\int _{\mathbb{R}} \varphi \de \eta _u \leq \int _{\mathbb{R}} \varphi \de \mu _u - \varepsilon \cdot \pi _{|\{r_1,\ldots,r_n,r\}}[D]  < \int _{\mathbb{R}} \varphi \de \mu _u,
	\end{equation*}
	we get a contradiction to $(\mu _t)_{t \in T}$ being NSI. If such an $\varepsilon > 0$ does not exist, then there has to exist an $\varepsilon > 0$ such that the Borel set 
	\begin{equation*}
	D' = \left\{ x \in \mathbb{R}^{T_r} : \int _{\mathbb{R}} \varphi(y) \de k'_{u}(x,dy) \geq  \int _{\mathbb{R}} \varphi(y) \de k_{u}(x_r,dy) + \varepsilon \right\}
	\end{equation*}
	has positive mass under $\pi _{|\{r_1,\ldots,r_n,r\}}$. We define $\eta'$ with reversed roles of $k_{t}$ and $k'_{t}$ and obtain a contradiction as above. 
\end{proof}

The Markov property of a martingale measure associated with a NSI peacock allows for a short proof of the crucial uniqueness property. 

\begin{proposition} \label{prop:NSItoUniq}
	If $(\mu _t)_{t \in T}$ is a peacock that is NSI, then $\MM _T(\mu _T)$ consists of only one martingale measure and the canonical process is Markov under this measure.
\end{proposition}

\begin{proof}
	Let $\pi, \pi ' \in \mathsf{M}_T(\mu _T)$ and let $k _{s,t}$ (resp.\ $k'_{s,t}$) be regular versions of $\pi [X_t \in \cdot \ \vert X_s]$ (resp.\ $\pi' [X_t \in \cdot \ \vert X_s]$) for all $0 \leq s < t \leq 1$. These exist because $\mathbb{R}^2$ is a Polish space. Assume that $\pi \neq \pi'$. Lemma \ref{lemma:NSItoMarkov} yields that both $\pi$ and $\pi'$ are Markov processes and thus there must exist $0 < r < u < 1$ such that $k_{r,u}(x, \cdot) \neq k'_{r,u}(x, \cdot)$ for all $x$ in a Borel set with positive $\mu _r$-mass. Suppose there exists a convex function $\varphi$ and $\varepsilon > 0$ such that the Borel set 
	\begin{equation*}
	D = \left\{ x \in \mathbb{R} : \int _{\mathbb{R}} \varphi \de k_{r,u}(x,\cdot) \geq  \int _{\mathbb{R}} \varphi \de k'_{r,u}(x,\cdot) + \varepsilon    \right\}
	\end{equation*}
	has positive mass under $\mu _r$. We set
	\begin{equation*}
	\eta _t = \begin{cases}
	\mu _t &t < r \\
	\int _{D^c} k_{r,t}(x, \cdot) \de \mu _r  + \int _{D} k'_{r,t}(x, \cdot) \de \mu _r & t \geq r
	\end{cases}
	\end{equation*}
	for all $t \in T$. Then $\eta _0 = \mu _0$, $(\eta _t)_{t \in T} \in \PP_{T}$ and $\eta _t \leqp 2 \mu _t$. Therefore
	\begin{equation*}
	\int _{\mathbb{R}} \varphi \de \eta _u \leq \int _{\mathbb{R}} \varphi \de \mu _u - \varepsilon \cdot \mu _r[D]  < \int _{\mathbb{R}} \varphi \de \mu _u
	\end{equation*}
	is a contradiction to $(\mu _t)_{t \in T}$ being NSI. If there exist no convex function and $\varepsilon > 0$ such that $D$ has postive mass under $\mu _r$, there exists a convex function $\varphi$ and $\varepsilon > 0$ such that
	\begin{equation*}
	D' = \left\{ x \in \mathbb{R}^{T_r} : \int _{\mathbb{R}} \varphi(y) \de k'_{r,u}(x,dy) \geq  \int _{\mathbb{R}} \varphi(y) \de k_{r,u}(x_r,dy) + \varepsilon \right\}
	\end{equation*}
	has positive mass under $\mu _r$.  We define $\eta'$ by reversing roles of $k_{r,t}$ and $k'_{r,t}$ and obtain a contradiction as above. 
\end{proof}

\begin{remark} \label{rem:NSIandKellerer}
	Suppose that $T = \{0,1\}$. Then the NSI property is closely related to the concept of Kellerer dilations (cf. \cite{Ke73}) that we explain in the following: For any closed set $F\subset \mathbb{R}$ the Kellerer dilation is a kernel $P_F$ defined for every $x\in \R\cap[\min F,\max F]$ by
	\begin{equation*}
	P_F(x, \cdot) = \begin{cases}
	\frac{x - x^-}{x^+ - x^-} \delta _{x^+} + \frac{x^+ - x}{x^+ - x^-} \delta _{x^-} & x \not \in F\\
	\delta _x & x \in F
	\end{cases}
	\end{equation*}
	where $x^+ = \min (F \cap [x, \infty))$ and $x^- = \max (F \cap (- \infty, x])$.  Kellerer showed in \cite[Satz 25]{Ke73} that for any $\mu \in \PO(\mathbb{R})$ with $\supp(\mu)\subset [\min F,\max F]$ there is only one martingale measure with marginals $\mu(dx)$ and $\mu P_F:=\int P_F(x,dy)\mu(dx)$ and it is given by $\mu(dx)P_F(x,dy)$. In particular, as a consequence of Lemma \ref{lemma:CharacNSI} and  \cite[Lemma 2.8]{BeJu21} a peacock $(\mu _0, \mu _1)$ is NSI if and only if $\mu _1 = \mu _0 P_{\mathrm{supp}(\mu _1)}$. Then the law of the unique martingale is $\mu(dx)P_{\supp(\mu_1)}(x,dy)$.
	
	In Corollary \ref{cor:thm83}, we will use this connection to recover \cite[Theorem 8.3]{NuStTa17}.
\end{remark}

\section{Shadow martingales indexed by countable set} \label{sec:GenExistCountable}

In this section we prove Theorem \ref{prop:GenExistCount} that is an analogue of Theorem \ref{thm:GenExist} in the case of a countable index set $T \subset [0, \infty)$ with minimal element $0$ and that attains also a maximal element, i.e.\ $\sup T\in T$. The proof follows the outline explained in Steps 1--3 in Subsection \ref{sec:MainIdeaProof}. In Subsection \ref{sec:RDofParam}, we show that martingale parametrizations are almost everywhere differentiable. The existence part of Theorem \ref{prop:GenExistCount} is covered in Subsection \ref{sec:CountableExistence}. In Subsection \ref{sec:MonotonicityPrinciple}, we will introduce an auxiliary optimization problem over peacocks and establish a monotonicity principle for this optimization problem that will allow us in Subsection \ref{sec:coreArgument} to deduce that any optimizer is necessarily NSI. Finally, in Subsection \ref{sec:ShadowUniq}, we show that the family of right derivatives of shadow martingales is a  solution to our auxiliary optimization problem. This allows us to conclude.

\subsection{Right-derivatives of martingale parametrizations} \label{sec:RDofParam}

Recall that $T$ is at most countable. In this subsection we show that in the current setup any martingale parametrization  is $\lambda$-a.e.\ right-differentiable. Recall that  $\PP_T$ is the set of all peacocks indexed by $T$ (cf.\ Definition \ref{def:PCOCspace}) and $\MM_T$ is the set of all martingale measures on $\mathbb{R}^T$ (cf. Definition \ref{def:MartSpaces}).

\begin{lemma} \label{lemma:PcocMartPolish}
	The spaces $\PP_T$ and $\mathsf{M}_T$ are Polish.
\end{lemma}

\begin{proof}
	Sine $T$ is at most countable,  by Lemma \ref{lemma:preservingOrder}, $\PP_T$ is a closed subset of the Polish space $\PO(\mathbb{R})^T$  and thus Polish itself (cf.\ \cite[Exercise 3.3]{Ke95}). In the same way, the closed subspace $\mathsf{M} _T$ is Polish once we have shown that  $\PO(\mathbb{R}^T)$  is a Polish space, which can be seen as follows: the Wasserstein metric on the set $\PO(\mathbb{R}^T)$ induced by the metric $d(x,y) = \sum _{n \in \mathbb{N}} 2^{-n} \min \{ | x_{f(n)} - y_{f(n)} |, 1\}$ on $\mathbb{R}^T$ for some surjective $f: \mathbb{N} \rightarrow T$ is a complete separable metric that, moreover, induces the topology on $\PO(\mathbb{R}^T)$ that we defined in Section \ref{sec:Notation}.
\end{proof}

\begin{lemma} \label{lemma:RightDerivatives}
	Let  $\pi \in \PO(\mathbb{R}^T)$ and $(\pi ^{\alpha})_{\alpha \in [0,1]}$ a parametrization of $\pi$. The curve $\alpha \mapsto \pi ^{\alpha}$ is $\lambda$-a.e.\ right-differentiable, i.e. for $\lambda$-a.e.\ $a \in [0,1)$ the right-derivative
	\begin{equation} \label{eq:RDLimit}
	\hat{\pi}^{a} = \lim _{h \downarrow 0} \frac{\pi ^{a + h} - \pi ^{a}}{h}
	\end{equation}
	exists as a limit in $\PO(\mathbb{R}^T)$ under $\TO$.
	
	More precisely, by choosing $\hat{\pi}^a$ constant on the $\lambda$ null set where the limit in \eqref{eq:RDLimit} does not exist, $(\hat{\pi} ^{a})_{a \in [0,1]}$ is a disintegration of the corresponding coupling of $\lambda$ and $\pi$ in Remark \ref{rem:ParamToCpl} w.r.t.\ $\lambda$. In particular, $a \mapsto \hat{\pi} ^{a}$ is a measurable map from $[0,1]$ to $\PO(\mathbb{R}^T)$.
	
	If $\pi \in \MM_T$ and $(\pi ^{\alpha})_{\alpha \in [0,1]}$ is a martingale parametrization of $\pi$, then the right derivative $\hat{\pi} ^{a}$ is an element of $\mathsf{M}_T$ for $\lambda$-a.e.\ $a$.
\end{lemma}

\begin{proof}
	Since $(\pi ^{\alpha})_{\alpha \in [0,1]}$ is a parametrization of $\pi$, $\frac{\pi ^{a + h} - \pi ^{a}}{h}$ is an element of $\PO(\mathbb{R}^T)$ for all $a \in [0,1)$ and $h \in (0,1-a]$.
	Let $\xi$ be the coupling of $\lambda$ and $\pi$ on $[0,1] \times \mathbb{R}^T$ defined in Remark \ref{rem:ParamToCpl} and let $(\xi _x)_{x \in [0,1]}$ be a version of the disintegration of $\xi$ w.r.t.\ $\lambda$ (we can disintegrate the measure $\xi$ because both $[0,1]$ and $\mathbb{R}^T$ are Polish spaces, cf.\ Lemma \ref{lemma:PcocMartPolish}). 
	
	Recall that the topology $\TO$ on $\PO(\mathbb{R}^T)$ is generated by the convergence of the integrals of all functions in $\mathcal{G}_0 \cup \mathcal{G}_1$ where
	\begin{eqnarray*}
		\mathcal{G}_0 &=& \left\{ g \circ (X_{t_1},\ldots,X_{t_n}) : n \geq 1, \ t_1, \ldots, t_n \in T, \ g \in C_b(\mathbb{R}^n) \right\} \hspace{1cm} \text{ and} \\
		\mathcal{G}_1 &=&  \left\{|X_t| : t \in T \right\}.
	\end{eqnarray*}
	For all $h > 0$ and $f \in \mathcal{G}_0 \cup \mathcal{G}_1$ we have 
	\begin{equation}
	\label{eq:IntRightDer}
	\int _{\mathbb{R^T}} f \de \left( \frac{\pi ^{a + h} - \pi ^{a}}{h}\right) = \frac{1}{h} \int _{a} ^{a + h} \left( \int _{\mathbb{R^T}} f \de \xi _x  \right) \de \lambda (x)
	\end{equation}
	and since $x \mapsto \int _{\mathbb{R}} f \de \xi _x$ is measurable and in $L^1(\lambda)$, the 
	Lebesgue differentiation theorem yields that the integral converges for $\lambda$-a.e.\ $a$ to $\int _{\mathbb{R^T}} f \de \xi _{a}$ as $h \rightarrow 0$. 
	
	We claim that one can choose this $\lambda$ null set independent of $f$. Indeed, the set 
	\begin{equation*}
	A_c = \left\{ g \circ (X_{t_1},\ldots,X_{t_n}) : n \geq 1, \ t_1, \ldots, t_n \in T, \ g \in C_c(\mathbb{R}^n) \right\} \subset \mathcal{G}_0
	\end{equation*}
	is separable (w.r.t.\ the supremum norm on $C(\mathbb{R}^T)$) because $T$ is countable and $C_c(\mathbb{R}^n)$ is separable for all $n \in \mathbb{N}$. Let 
	$\mathcal{X}$ be a countable dense subset of $A_c$.
	Then there exists a $\lambda$ null set $L \subset [0,1]$ such that \eqref{eq:IntRightDer} converges to $\int _{\mathbb{R}^T} f \de \xi _{a}$
	as $h \rightarrow 0$ for all $f \in \mathcal{X} \cup \mathcal{G}_1$ and $a \not \in L$ ($\mathcal{G}_1$ is countable). Using the triangle inequality, it follows that this convergence holds for all $f \in A_c \cup \mathcal{G}_1$ and $a \not \in L$. Moreover, since $\xi _{a}$ is a probability measure for $\lambda$-a.e.\ $a$, we conclude that \eqref{eq:IntRightDer} converges as $h \rightarrow 0$ for all $f \in \mathcal{G}_0 \cup \mathcal{G}_1$ and $\lambda$-a.e $a$ because weak and vague convergence coincides for probability measures on a Polish space.
	
	Thus, we have shown that $\left(\frac{\pi ^{a + h} - \pi ^{a}}{h}\right)_{h > 0}$ converges under $\TO$ in $\PO(\mathbb{R}^T)$  to $\xi _a$ as $h \downarrow 0$ for all $a$ outside the $\lambda$ null set $L$. Since $(\xi_a)_{a \in [0,1]}$ is a disintegration, the map $a \mapsto \xi_a$ is measurable.
	
	Furthermore, if $(\pi^{\alpha})_{\alpha \in [0,1]}$ is a martingale parametrization, notice that for all $a \in [0,1)$ and $h \in (0,1-a]$ the quotient $\textstyle{
		\frac{\pi ^{a+h} - \pi ^{a}}{h}}$
	is a martingale measure and this property is preserved under convergence w.r.t.\ $\TO$.
\end{proof}

\begin{corollary} \label{cor:ParamDeterminedByRD}
	Let $\pi$ and $\rho$ be in $\PO(\mathbb{R}^T)$, $(\pi ^{\alpha})_{\alpha \in [0,1]}$ a parametrization of $\pi$ and $(\rho ^{\alpha})_{\alpha \in [0,1]}$ a parametrization of $\rho$. If the right-derivatives $(\hat{\pi}^a)_{a \in [0,1]}$ and $(\hat{\rho}^a)_{a \in [0,1]}$	coincide for $\lambda$-a.e.\ $a$, then $\pi ^{\alpha} = \rho ^{\alpha}$ for all $\alpha \in [0,1]$.	
\end{corollary}

\begin{proof}
	This is a direct consequence of Remark \ref{rem:ParamToCpl} and Lemma \ref{lemma:RightDerivatives}.
\end{proof}

\subsection{Existence of shadow martingales} \label{sec:CountableExistence}

In order to show the existence of a martingale parametrization $(\pi ^{\alpha})_{\alpha \in [0,1]}$ that satisfies 
\begin{equation} \label{eq:ShadowCurve61}
\pi ^{\alpha} (X_t \in \cdot) = \shadow{T_t}{\nu ^{\alpha}}
\end{equation}
for all $t \in T$ and $\alpha \in [0,1]$, we will construct an appropriate family of martingale measures $(\hat{\pi}^a)_{a \in [0,1]}$ such that $a \mapsto \hat{\pi}^a$ is a measurable function and $\pi ^{\alpha} = \int _0 ^{\alpha} \hat{\pi}^a \de a$
satisfies \eqref{eq:ShadowCurve61} for all $t \in T$ and $\alpha \in [0,1]$, i.e.\ we construct the right-derivatives of $\alpha \mapsto \pi ^{\alpha}$. To this end, it is necessary and sufficient that the marginal distribution of $\hat{\pi}^a$ at time $t$ coincides with the  right-derivatives of $\alpha \mapsto \shadow{T_t}{\nu ^{\alpha}}$ for $\lambda$-a.e.\ $a$.

Recall that $T$ is a countable subset of $[0,\infty)$ with $0 \in T$ and $\sup T\in T$.

\begin{lemma}\label{lemma:RDofShadow}
	Let $(\nu ^{\alpha})_{\alpha \in [0,1]}$ be a parametrization of $\mu _0$. 
	There exists a Borel set $A \subset [0,1]$ with $\lambda(A) = 1$ such that for every $a \in A$  the following holds:
	\begin{enumerate}
		\item [(i)] For all $t \in T$, the curve $\alpha \mapsto \shadow{\mu_{T_t}}{\nu ^{\alpha}}$ is right-differentiable at $a$ (in the sense of Lemma \ref{lemma:RightDerivatives}) and we denote this right derivative by $\hat{\eta} _t ^{a}$.
		\item [(ii)] The family $(\hat{\eta} _t ^{a})_{t \in T}$ is a peacock with initial value $\hat{\eta} ^a _0 = \hat{\nu} ^a$. Here, $\hat{\nu} ^a$ is the right-derivative of $\alpha \mapsto \nu ^{\alpha}$ at $a$.
	\end{enumerate}
	Moreover, setting $\hat{\eta}_t ^a = \delta _0$ for all $t \in T$ and $a \not \in A$,  $a \mapsto (\hat{\eta} _t ^{a})_{t \in T}$ is a measurable map from $[0,1]$ to $\PP_{T}$ and, for all $\alpha \in [0,1]$ and $t \in T$, it holds
	\begin{equation*}
	\shadow{\mu_{T_t}}{\nu ^{\alpha}} = \int ^{\alpha} _0 \hat{\eta} _t ^a \de a.
	\end{equation*}
\end{lemma}

\begin{proof}
	It is not difficult to see that $(\shadow{\mu_{T_t}}{\nu ^{\alpha}})_{\alpha \in [0,1]}$ is a parametrization of $\mu_t$. Hence, for all $t \in T$ Lemma \ref{lemma:RightDerivatives} yields that there exists a Borel set $A_t \subset [0,1]$ with $\lambda (A_t) = 1$ such that the map $\alpha \mapsto \shadow{\mu _{T_t}}{\nu ^{\alpha}}$ is right-differentiable for all $a \in A_t$. We set $A = \bigcap _{t \in T} A_t$ and denote the right derivatives by $\hat{\eta} _t ^{a}$ for all $a \in A$ and $t \in T$. Then, item (i) holds.
	
	Moreover, Proposition \ref{prop:ShadowAssoc} in conjunction with Lemma \ref{lemma:ShadowFactor} implies that
	\begin{equation*} 
	\hat{\eta} _t ^a = \lim _{h \downarrow 0} \shadow{\frac{1}{h} \left(\mu _s - \shadow{\mu _{T_s}}{\nu ^a} \right)_{s \in T_t}}{\frac{\nu ^{a +h} - \nu ^a}{h}}   
	\end{equation*}
	for all $a \in A$ and $t \in T$. Clearly, for all $a \in A$ it holds $\hat{\eta} ^a _0 = \hat{\nu} ^a$  and we obtain
	\begin{equation*}
	\shadow{\frac{1}{h} \left(\mu _s - \shadow{\mu _{T_s}}{\nu ^a} \right)_{s \in T_t}}{\frac{\nu ^{a +h} - \nu ^a}{h}} \leq _c \shadow{\frac{1}{h} \left(\mu _s - \shadow{\mu _{T_s}}{\nu ^a} \right)_{s \in T_u}}{\frac{\nu ^{a +h} - \nu ^a}{h}}
	\end{equation*}
	for all $t \leq u$ in $T$ and $h > 0$. Lemma \ref{lemma:preservingOrder} shows that the convex order is preserved under convergence in $\TO$, Item (ii) follows.
	
	Finally, note that Lemma \ref{lemma:RightDerivatives} implies that for all $t \in T$,  $a \mapsto \hat{\eta} _t ^{a}$ is a measurable map from $[0,1]$ to $\mathbb{R}$ and $(\hat{\eta}^a_t)_{a \in [0,1]}$ is a disintegration of the coupling $\xi ^t$ between $\lambda$ and $\mu _t$ w.r.t.\ $\lambda$ that corresponds to the parametrization $(\shadow{\mu_{T_t}}{\nu ^{\alpha}})_{\alpha \in [0,1]}$ in the sense of Remark \ref{rem:ParamToCpl} (setting $\hat{\eta}_t ^a = \delta _0$ for all $a \not \in A$).  Hence, $a \mapsto (\hat{\eta} _t ^{a})_{t \in T}$ is a measurable map from $[0,1]$ to $\PP_{T}$ and, for all $\alpha \in [0,1]$ and $t \in T$, it holds
	\begin{equation*}
	\int ^{\alpha} _0 \hat{\eta} _t ^a \de a = \xi ^t([0, \alpha]) = 
	\shadow{\mu_{T_t}}{\nu ^{\alpha}}. \qedhere
	\end{equation*}
\end{proof}

\begin{lemma} \label{lemma:MeasurableKellerer}
	There exists a measurable map $\PP_T \rightarrow \mathsf{M}_T$ such that the image of $(\mu _t)_{t \in T}$ is an element of $\mathsf{M}_T((\mu _t)_{t \in T})$.
\end{lemma}

\begin{proof}
	By Lemma \ref{lemma:PcocMartPolish},  $\PP_T$ and $\mathsf{M}_T$ are Polish spaces.
	Let $\Phi : \textsf{M}_T \ni \pi \mapsto (\mathrm{Law}_{\pi}(X_t))_{t \in T} \in \PP_T$. It is not difficult to see that $\Phi$ is continuous (thus measurable) and Proposition \ref{prop:KellerersThm} yields that $\Phi$ is a surjective map. 
	Since $T$ is countable with $\sup T \in T$, the set $\Phi ^{-1}(\{ (\mu _t) _{t \in T} \} ) = \MM_T((\mu _t) _{t \in T})$ is compact for all $(\mu _t)_{t \in T} \in \PP_T$ by \cite[Lemma 2.1]{BeHuSt16}.
	The measurable selection theorem of Dellacherie \cite{De73} shows that there exists a measurable right-inverse $\Phi ^{-1}$. 
\end{proof}

With respect to Lemma \ref{lemma:MeasurableKellerer} we would like to emphasize an impressive result by Lowther \cite{Lo08}.
For continuous peacocks with connected supports the measurable map of Lemma \ref{lemma:MeasurableKellerer} can be chosen to be continuous. For these peacocks, Lowther's map is the unique continuous map.

Now we are able to prove the existence part of Theorem \ref{thm:GenExist} for the countable index set $T$:

\begin{proposition} \label{prop:CountableExistSM}
	Let $(\nu ^{\alpha})_{\alpha \in [0,1]}$ be a parametrization of $\mu _0$.
	There exists a martingale measure $\pi \in \mathsf{M}_T((\mu _t)_{t \in T})$ and a martingale parametrization $(\pi ^{\alpha})_{\alpha \in [0,1]}$ of $\pi$ such that
	\begin{equation*}
	\pi^{\alpha} (X_t \in \cdot) = \shadow{\mu _{T_t}}{\nu ^{\alpha}}
	\end{equation*}
	for all $\alpha \in [0,1]$ and $t \in T$.
\end{proposition}

\begin{proof}
	Lemma \ref{lemma:RDofShadow} yields that there exists $A \subset [0,1]$ with $\lambda(A) = 1$ and a measurable map $a \mapsto (\hat{\eta} ^a _t)_{t \in T}$  from $[0,1]$ to $\PP_T$ such that $(\hat{\eta} ^a _t)_{t \in T}$ is the right-derivative of $\alpha \mapsto \shadow{\mu _{T_t}}{\nu ^{\alpha}}$ at $a$  for all $a \in A$ and $t \in T$.
	Thus, Lemma \ref{lemma:MeasurableKellerer} implies that there exists a measurable map $a \mapsto \hat{\pi}^a$ from $[0,1]$ to $\MM_T$ such that $\hat{\pi}^a \in \MM_T((\hat{\eta} ^a _t)_{t \in T})$ for all $a \in [0,1]$. We set $ \pi ^{\alpha} = \int _0 ^{\alpha} \hat{\pi} ^a \de a.$
	It is easy to check that $(\pi ^{\alpha})_{\alpha \in [0,1]}$ is a well-defined martingale parametrization of the martingale measure $\pi = \pi^1$ w.r.t.\ $(\nu ^{\alpha})_{\alpha \in [0,1]}$. Moreover, we have
	\begin{equation} \label{eq:ShadowPropertyCount}
	\pi ^{\alpha} (X_t \in \cdot) = \int _0 ^{\alpha} \hat{\pi}^a(X_t \in \cdot) \de a = \int _0 ^a \hat{\eta} _t ^a \de a = \shadow{\mu _{T_t}}{\nu ^{\alpha}}
	\end{equation}
	for all $\alpha \in [0,1]$ and $t \in T$. In particular, \eqref{eq:ShadowPropertyCount} implies for $\alpha = 1$ that $\pi$ is a solution to the peacock problem w.r.t.\ $(\mu _t)_{t \in T}$. 
\end{proof}

\begin{remark} Note that the existence of a shadow martingale does not require the parametrization to be $\leqcs$-convex.	
\end{remark}

\subsection{An auxiliary optimization problem} \label{sec:MonotonicityPrinciple}

Recall that we assume $T$ to be at most countable. In this subsection, we will introduce an auxiliary optimization problem over families of peacocks. The main result is a  monotonicity principle, i.e.\ a necessary pointwise optimality condition, similar to $c$-cyclical monotonicity in classical optimal transport, see e.g.\ \cite{Vi09}, or monotonicity principles in stochastic variants of the transport problem, e.g.\ \cite{BeJu21, BeCoHu17}. Our monotonicity principle is similar in spirit to the one recently proved for the weak transport problem \cite{BaBeHuKa20, GoJu20, BaBePa19,AlBoCh19}.

For a given peacock $(\mu _t) _{t \in T} \in \PP_{T}$ and a family of probability measures $(\tilde{\nu} ^a)_{a \in [0,1]}$ such that $a \mapsto \tilde{\nu} ^a$ is a measurable map from $[0,1]$ to $\PO(\mathbb{R})$, we set 
\begin{equation*}
\mathcal{A} = \left\{ (\theta ^{a})_{a \in [0,1]} \ \vert \
\begin{array}{l} \theta ^{a} \in \PP_{T}, a \mapsto \theta ^a \textit{ measurable, } \\ 
\theta ^a _0 = \tilde{\nu} ^a, \ \int _0 ^1 \theta ^{a} _t \de a = \mu _t 
\end{array}		 \right\}.
\end{equation*}
Let $c : [0,1] \times \PP_{T} \rightarrow [0, \infty)$ be a Borel measurable cost function that is linear and l.s.c.\ in the second component. (Note that starting from the next subsection we additionally ask $(\tilde{\nu} ^a)_{a \in [0,1]}$ to be increasing in the $\leqcs$-order.) We are interested in properties of solutions to the optimization problem
\begin{equation} \label{eq:OptRD}
\mathsf V_{\mathcal A}:=\inf _{ (\theta ^{a})_{a \in [0,1]} \in \mathcal{A}} \int _0 ^1 c(a, \theta^ {a}) \de a.	
\end{equation}

The following monotonicity principle will be essential in the next section.

\begin{proposition} \label{prop:GenMonotonicityPrinciple}
	Assume $\mathsf V_{\mathcal A}<\infty.$  If $(\theta ^{a})_{a \in [0,1]}$ is optimal for \eqref{eq:OptRD}, then there exists a measurable set $A \subset [0,1]$ with $\lambda(A) = 1$ such that for all $a < a'$ in $A$ we have
	\begin{equation*}
	c(a, \theta ^{a}) + c(a', \theta ^{a'}) \leq c(a, \theta ') + c(a', \theta '')
	\end{equation*}
	for any two peacocks $(\theta '_t)_{t \in T}$, $(\theta ''_t)_{t \in T}$ that satisfy $\theta' + \theta'' = \theta ^{a} + \theta ^{a'}$,  $\theta ^a_0 = \theta' _0$ and $\theta ^{a'}_0 = \theta ''_0$.
\end{proposition}

\begin{proof}
	This proof follows closely the proof of \cite[Proposition 4.1]{BaBeHuKa20}.
	Recall that $T$ is at most countable such that $\PP_{T}$ is a Polish space (cf.\ Lemma \ref{lemma:PcocMartPolish}) and thus
	\begin{equation*}
	\mathcal{B} = \left\{ (a,a',\theta', \theta'') \in [0,1]^2 \times {\PP_{T}}^2 : \begin{array}{l} \theta' + \theta'' = \theta ^{a} + \theta ^{a'}, \  \theta' _0 = \tilde{\nu}^a, \ \theta ''_0 = \tilde{\nu}^{a'}, \\ 
	c(a, \theta ^{a}) + c(a', \theta ^{a'}) > c(a, \theta ') + c(a', \theta '')
	\end{array}   \right\}
	\end{equation*}
	is an analytic set.\footnote{Recall that a subset of a Polish space is called analytic if it is the image of a continuous function defined on another Polish space. The countable intersection of analytic sets and the preimage of an analytic set under a Borel measurable map  are analytic. Since all Borel sets are analytic, we get that $\mathcal{B}$ is indeed an analytic set.} Likewise the projection of $\mathcal{B}$ onto $[0,1]^2$, denoted by $B$, is analytic. Furthermore Lusin's theorem \cite[Theorem 21.20]{Ke95} states that any analytic set is universally measurable and thus the mass of an analytic set and the integral of an analytically measurable function under any Borel measure are well-defined.
	We will show that $\xi(B) = 0$ for all couplings $\xi$ of $\lambda$ and $\lambda$. Then  \cite[Proposition 2.1]{BeGoMaSc09} (for analytic sets) yields that there exists a $\lambda$ null set $N$ with $B \subset (N \times [0,1]) \cup ([0,1] \times N)$ and the claim follows by choosing $A = N^c$.
	
	Suppose there exists a coupling $\xi$ of $\lambda$ and $\lambda$ with $\xi (B) > 0$. Then the symmetrized coupling $\xi' = \frac{1}{2}(\xi + s_{\#} \xi)$ where $s: (a,a') \mapsto (a',a)$ is again a coupling of $\lambda$ and $\lambda$ with $\xi '(B) > 0$. By the Jankov--von Neumann uniformization theorem (see \cite[Theorem 18.1]{Ke95}) there exists an analytically measurable map
	\begin{equation*}
	\tilde{\Phi} : B \rightarrow \PP^2_{T} \text{ such that } (a,a',\tilde{\Phi} _1(a,a'),\tilde{\Phi}  _2(a,a')) \in \mathcal{B} 
	\end{equation*}
	for all $ (a,a') \in B$. We extend $\tilde{\Phi} $ to an analytically measurable map $\Phi : [0,1]^2 \rightarrow \PP_{T}$ by setting 
	\begin{equation*}
	\Phi(a,a') = \begin{cases}
	\tilde{\Phi}(a,a') & (a,a') \in B \\
	(\theta ^a, \theta ^{a'}) & (a,a') \not \in B
	\end{cases}.
	\end{equation*}
	Denote the first (resp.\ second) component of $\Phi$ by $\Phi _1$ (resp.\ $\Phi _2$). Define
	\begin{equation*}
	\tilde{\kappa}_t ^a = \int _0 ^1 \Phi _1(a,a')_t \de \xi ' _a(a'), \  \check{\kappa} ^{a'}_t = \int _0 ^1 \Phi _2(a,a')_t \de \xi ' _{a'}(a)
	\end{equation*}
	for all $t \in T$ where $(\xi' _a)_{a \in [0,1]}$ is the disintegration of $\xi'$ w.r.t.\  $\lambda$ (Recall that $\xi'$ is symmetrized and thus the disintegrations w.r.t.\ the first and the second marginal distribution coincide). For $a\in [0,1]$ put $\kappa ^a _t = \frac{1}{2} \left( \tilde{\kappa}_t ^a + \check{\kappa} ^{a}_t   \right)$. Then for all $a \in [0,1]$ the family $(\kappa ^a _t)_{t \in T'}$   is a peacock with $\kappa ^a _0 = \tilde{\nu} ^a$. Furthermore, by definition of $B$,
	\begin{eqnarray*}
		\int _0 ^1 \kappa ^a _t \de a &=& \int _{[0,1]^2} \left( \frac{\Phi_1(a,a')_t + \Phi_2(a,a')_t}{2} \right) \de \xi' (a,a') \\ &=& \int _{[0,1]^2} \left( \frac{\theta ^a_t + \theta ^{a'} _t}{2} \right) \de \xi' (a,a') = \int _0 ^1 \theta _t ^a \de a = \mu _t
	\end{eqnarray*} 
	for all $t \in T$ and thus $(\kappa ^a)_{a \in [0,1]}$ is an element of $\mathcal{A}$, i.e.\ a competitor of $(\theta ^{a})_{a \in [0,1]}$ in the optimization problem \eqref{eq:OptRD}. 
	Since $c$ is linear and l.s.c.\ in the second component and  $\xi'(B) > 0$, it follows that
	\begin{eqnarray*}
		\int _0 ^1 c(a, \kappa ^a) \de a &=& \frac{1}{2} \int _0 ^1 c(a,  \tilde{\kappa} ^a) \de a + \frac{1}{2} \int _0 ^1 c(a', \check{\kappa} ^{a'}) \de a'  \\
		&\leq& \int _{[0,1]^2} \frac{c(a, \Phi_1(a,a')) + c(a', \Phi_2(a,a'))}{2} \de \xi '(a,a') \\
		&<& \int _{[0,1]^2} \frac{c(a, \theta ^a) + c(a', \theta ^{a'})}{2} \de \xi '(a,a') = \int _0 ^1 c(a, \theta ^a) \de a.
	\end{eqnarray*}
	This contradicts the fact that $(\theta ^{a})_{a \in [0,1]}$ is optimal. 
\end{proof}

\begin{remark}
	Of course, the same proof works if $c$ is only convex in the second component instead of being linear. 
\end{remark}

\subsection{NSI property for simultaneous optimizers} \label{sec:coreArgument}

As a last preparation for the proof of the uniqueness part of  Theorem \ref{prop:GenExistCount} in Section \ref{sec:ShadowUniq}  we will show how the NSI property is closely connected to optimizers of \eqref{eq:OptRD}. For $t \in T$ we set
\begin{equation*}
c_t(a, \theta) = (1-a) \int _{\mathbb{R}} x + \sqrt{1 + x^2} \de \theta _t.
\end{equation*}
Clearly, $c_t$ is an admissible cost function for \eqref{eq:OptRD} in Subsection \ref{sec:MonotonicityPrinciple}.
Throughout this section, we fix some family $(\tilde\nu^a)_{a\in [0,1]}$ of probability measures on $\R$ that is increasing in the $\leqcs$-order as input data for the optimization problem \eqref{eq:OptRD}.

The crucial observation, proved in Proposition \ref{prop:OptToNSI}, is that, if a family $(\theta ^a)_{a \in [0,1]}$ of peacocks  minimizes simultaneously the optimization problem \eqref{eq:OptRD} with cost function $c_t$  for all $t\in T$, then $\theta^a$ is NSI for $\lambda$-almost every $a$.

\begin{lemma} \label{lemma:ConStoSeparation}
	Let $(\theta_t)_{t \in T}$ and $(\theta'_t)_{t \in T}$ be two peacocks with $\theta _0 \leqcs \theta' _0$. There exist two peacocks $(\tilde{\theta}_t)_{t \in T}$ and $(\tilde{\theta}'_t)_{t \in T}$ with
	\begin{enumerate}
		\item[(i)] $\tilde{\theta} _0 = \theta _0$, $\tilde{\theta}' _0 = \theta' _0$, 
		\item[(ii)] $\theta _t + \theta '_t = \tilde{\theta}_t + \tilde{\theta}'_t$ for all $t \in T$ and
		\item[(iii)] $\tilde{\theta}_t \leqcs \theta _t$ and $\tilde{\theta}_t \leqcs \theta' _t$ for all $t \in T$.
	\end{enumerate}
\end{lemma}

\begin{proof}
	Set $\tilde{\theta}_t = \shadow{(\theta _s + \theta '_s)_{s \in T_t}}{\theta _0}$ and $\tilde{\theta}'_t = \theta _t + \theta '_t - \tilde{\theta}_t$ for all $t \in T$. Both $(\tilde{\theta} _t)_{t \in T}$ and $(\tilde{\theta}' _t)_{t \in T}$ are peacocks by Remark \ref{rem:RestIsPCOC}. They clearly satisfy properties (i) and (ii). Furthermore, for all $t \in T$ it holds $\tilde{\theta}_t \leqc \shadow{\theta _{T_t}}{\theta _0} =  \theta_t$
	by Lemma \ref{lemma:GenShadowOrder} (iv) and
	\begin{equation*}
	\tilde{\theta}_t \leqcs \shadow{(\theta _s + \theta '_s)_{s \in T_t}}{\theta ' _0} \leqc  \shadow{\theta' _{T_t}}{\theta _0'} = \theta'_t
	\end{equation*}
	by Lemma \ref{lemma:GenShadowOrder} (iii) and (iv).
\end{proof}

Property (ii) and (iii)  together imply that we also have $\theta _t \leqcs \tilde{\theta}'_t$ and $\theta' _t \leqcs \tilde{\theta}'_t$ for all $t \in T$. Thus, we have sandwiched $\theta$ and $\theta'$  between $\tilde{\theta}$ and $\tilde{\theta}'$ in convex-stochastic order.

\begin{lemma} \label{lemma:trivialCalculations}
	If $x,x',y,y' \in \mathbb{R}$ satisfy $x + x' = y + y'$ and $x < y$, then
	\begin{equation*}
	(1 - a)x + (1 - a')x' < (1-a)y + (1-a')y'
	\end{equation*}
	for all $a < a'$ in $[0,1]$.
\end{lemma}

\begin{proof}
	The inequality $(1 - a)x + (1 - a')x' < (1-a)y + (1-a')y'$ holds if and only if $(a' - a)(y - x) > 0$ because $y' - x' = x - y$.
\end{proof}

\begin{lemma} \label{lemma:OptToOrder}
Let  $(\theta ^a)_{a \in [0,1]}$ be a minimizer of \eqref{eq:OptRD} with finite cost $\mathsf{V}_{\mathcal A}$ w.r.t.\ $c_t$ simultaneously for all $t \in T$. Then there exists a Borel set $A \subset [0,1]$ with $\lambda (A) = 1$ such that for all $a  < a'$ in $A$ it holds
	\begin{equation} \label{eq:OptToOrder}
	2 \theta ^a _t - \shadow{(2 \theta ^a _s)_{s \in T_t}}{\theta _0} \leqcs \shadow{(2 \theta ^{a'} _s)_{s \in T_t}}{\theta _0},\qquad\text{for all }t\in T. 
	\end{equation}
\end{lemma}

The main idea of the proof is to show that whenever \eqref{eq:OptToOrder} is not satisfied for some $t \in T$, the pair $((a,\theta ^a),(a',\theta ^{a'}))$ violates the monotonicity principle in Proposition \ref{prop:GenMonotonicityPrinciple} for $c_t$. However, since the convex-stochastic order is not a total order relation on $\MO(\mathbb{R})$, the negation of \eqref{eq:OptToOrder} does not imply that the reversed order relation is true but the two measures might just be not comparable in convex-stochastic order. Thus, we use Lemma \ref{lemma:ConStoSeparation} to construct a new pair of competitors that are comparable and bring the essential improvement (cf. \eqref{eq:TheImportantImpro}).

\begin{proof}
	Recall that $T$ is a countable set. For every $t \in T$, there exists by Proposition \ref{prop:GenMonotonicityPrinciple} a Borel set $A_t \subset [0,1]$ with $\lambda(A_t) = 1$ such that for all $a < a'$ in $A_t$ we have
	\begin{equation} \label{eq:CompProp}
	c_t(a, \theta ^{a}) + c_t(a', \theta ^{a'}) \leq c_t(a, \theta ') + c_t(a', \theta '')
	\end{equation}
	where $(\theta '_t)_{t \in T}$ and $(\theta ''_t)_{t \in T}$ are any two peacocks with $\theta' + \theta'' = \theta ^{a} + \theta ^{a'}$,  $\theta ^a_0 = \theta' _0$ and $\theta ^{a'}_0 = \theta ''_0$.
	Put $A = \bigcap _{t \in T}  A_t$ and note that $\lambda (A) = 1$. For all $a \in A$ and $t \in T$ we define
	\begin{equation*}
	\theta ^{a-}_t := \shadow{(2 \theta ^a _s)_{s \in T_t}}{\theta _0^a} \hspace{0.5cm} \textit{and} \hspace{0.5cm}
	\theta ^{a+}_t := 2 \theta ^a _t - \theta ^{a-}_t.
	\end{equation*}
	We want to show that $\theta ^{a+}_t \leqcs \theta ^{a'-}_t$ for all $t \in T$ and all $a < a'$ in $A$. Suppose this is not the case for some $u \in T$ and $a < a'$ in $A$. Since $\theta^{a+} = \theta_0^a = \tilde{\nu}^a \leqcs \tilde{\nu}^{a'} = \theta_0^{a'} = \theta^{a'-}$ because the familiy $(\tilde{\nu}^a)_{a \in [0,1]}$ is monotonously increasing w.r.t.\ $\leqcs$, we may apply Lemma \ref{lemma:ConStoSeparation} to the peacocks $\theta^{a+}$ and $\theta^{a'-}$. Hence,
	 there exist two peacocks $\tilde{\theta}$ and $\tilde{\theta}'$ with $\tilde{\theta}_0 = \theta ^{a+}_0 = \theta ^a _0=  \tilde{\nu} ^a$, 
	$\tilde{\theta}'_0 = \theta ^{a'-}_0 = \theta ^{a'} _0 =  \tilde{\nu} ^{a'}$, 
	$\tilde{\theta} + \tilde{\theta}' = \theta ^{a+} + \theta ^{a'-}$, and
	\begin{equation*} 
	\tilde{\theta}_t \leq_{c,s} \theta^{a+}_t \hspace{0.5cm} \textit{and} \hspace{0.5cm} \tilde{\theta}_t \leq_{c,s} \theta^{a'-}_t 
	\end{equation*}
	for all $t \in T$. The inequality $\tilde{\theta}_u \leq_{c,s} \theta_u^{a+}$ cannot be an equality because this would imply that $\theta ^{a+}_u = \tilde{\theta}_u \leq_{c,s} \theta ^{a'-}_u$, which we supposed to be false. Hence, it holds
	\begin{equation} \label{eq:TheImportantImpro}
	c_u(a, \tilde{\theta}) < c_u(a, \theta^{a+})
	\end{equation}
	because $x \mapsto x + \sqrt{1 + x^2}$ is strictly increasing and strictly convex. Next, we use $(\tilde\theta,\tilde\theta')$ to construct a competitor $(\theta',\theta'')$ for $(\theta^a,\theta^{a'})$ in the sense of Proposition \ref{prop:GenMonotonicityPrinciple}. We set
	\begin{equation*}
	\theta ' = \frac{1}{2} \theta ^{a-} + \frac{1}{2} \tilde{\theta}\quad \textit{ and } \quad\theta '' = \frac{1}{2} \tilde{\theta}' + \frac{1}{2} \theta ^{a'+}.
	\end{equation*}
	The pair $(\theta', \theta'')$ of peacocks is  indeed a competitor   since $\theta' + \theta'' = \theta ^{a} + \theta ^{a'}$,  $\theta ' _0 = \theta ^a_0$ and $\theta'' _0 =  \theta ^{a'}_0$. Moreover, for $t =u$ it holds
	\begin{eqnarray*}
		&& c_u(a, \theta ') + c_u(a',\theta '') = \frac{1}{2} \left(c_u(a, \theta ^{a-}) + c_u(a, \tilde{\theta} )  + c_u(a', \tilde{\theta}') + c_u(a', \theta ^{a'+}) \right) \\
		&<& \frac{1}{2} \left(c_u(a, \theta ^{a-}) + c_u(a, \theta ^{a+} )  + c_u(a', \theta ^{a'-}) + c_u(a', \theta ^{a'+})\right) = c_u(a, \theta ^a) + c_u(a', \theta ^{a'})
	\end{eqnarray*}
	by the linearity of $c_u$ in the second component and Lemma \ref{lemma:trivialCalculations} in conjunction with \eqref{eq:TheImportantImpro}.
	This is a contradiction of \eqref{eq:CompProp}.
\end{proof}

\begin{proposition} \label{prop:OptToNSI}
	Let  $(\theta ^a)_{a \in [0,1]}$ be a simultaneous minimizer of \eqref{eq:OptRD} with finite cost $\mathsf{V}_{\mathcal A}$ with respect to $c_t$ for all $t \in T$. Then $(\theta ^a _t)_{t \in T}$ is NSI for a.e.\ $a \in [0,1]$.
\end{proposition}

\begin{proof}
	By Lemma \ref{lemma:OptToOrder} there exists a Borel set $A \subset [0,1]$ with $\lambda(A) = 1$ such that  for all $a  < a'$ in $A$ and $t \in T$ it holds
	\begin{equation}\label{eq:ImportantMonotonicity}
	2 \theta ^a _t - \shadow{(2 \theta ^a _s)_{s \in T_t}}{\theta^a _0} \leqcs \shadow{(2 \theta ^{a'} _s)_{s \in T_t}}{\theta^a _0}.
	\end{equation}
	Moreover, by Proposition \ref{prop:ShadowAssoc} and Lemma \ref{lemma:GenShadowOrder} (ii) it holds
	\begin{equation} \label{eq:CharNotNSI}
	\shadow{(2 \theta ^{a} _s)_{s \in T_t}}{\theta^a _0} \leqcs 2 \theta ^a _t - \shadow{(2 \theta ^a _s)_{s \in T_t}}{\theta^a _0} 
	\end{equation}
	for all $a \in  [0,1]$ and $t \in T$. Hence, the map $	a \mapsto \shadow{(2 \theta ^{a} _s)_{s \in T_t}}{\theta ^a _0} $ is increasing on $A$ in convex-stochastic order for all $t \in T$. If $\theta ^a$ is not NSI for some $a \in A$, there exists at least one $t \in T$ for which \eqref{eq:CharNotNSI} is not an equality (see Lemma \ref{lemma:CharacNSI}) and thus the map $a \mapsto \shadow{(2 \theta ^{a} _s)_{s \in T_t}}{\theta ^a _0}$ has a discontinuity at $a$ because of \eqref{eq:ImportantMonotonicity}.
	But since the map is increasing in convex-stochastic order, Corollary \ref{cor:MonotonicityPCOC} yields that this can only happen for countably many $a \in A$ for a given $t \in T$.  The set $T$ is countable, and hence we obtain that $\theta ^a$ is NSI for $\lambda$-a.e.\ $a$.
\end{proof}

\begin{remark}
	Referring back to \S \ref{sec:choquet}, 
	the last proposition establishes the decomposition of a peacock into NSI peacocks, cf.\ \eqref{eq:pcocrep} under the assumption that there is a suitable cost function $c$ such that for a given parametrization $(\nu^\alpha)_{\alpha \in [0,1]}$ there is an optimizer (with finite value) to \eqref{eq:OptRD}. For $\leqcs$-convex parametrizations this assumption will be established in the next subsection.
\end{remark}

\subsection{Uniqueness of the shadow martingale} \label{sec:ShadowUniq}

\begin{theorem} \label{prop:GenExistCount}
	Let $T \subset [0, \infty)$ be a countable index set with $0 \in T$ and $\sup T\in T$ and let $(\nu ^{\alpha})_{\alpha \in [0,1]}$ be a parametrization of $\mu _0$ that is $\leqcs$-convex. 
	There exists a unique pair $(\pi,(\pi ^{\alpha})_{\alpha \in [0,1]})$ where the martingale measure $\pi \in \mathsf{M}_{T}((\mu _t)_{t \in T})$ solves the peacock problem w.r.t.\ $(\mu _t)_{t \in T}$, $(\pi ^{\alpha})_{\alpha \in [0,1]}$ is a  martingale parametrization of $\pi$ w.r.t.\ $(\nu ^{\alpha})_{\alpha \in [0,1]}$ and for all $\alpha$ in $[0,1]$ and $t$ in $T$,
	\begin{equation} \label{eq:CountUniqSM}
	\pi^{\alpha} (X_t \in \cdot) = \shadow{\mu _{T_t}}{\nu ^{\alpha}}.
	\end{equation}

	Moreover, there exists a Borel set $A \subset [0,1]$ with $\lambda(A) = 1$ such that for all $a \in A$
	the map $\alpha \mapsto \pi ^{\alpha}$ is right-differentiable at $a$ and the marginals of the right-derivative $\hat{\pi}^a$ at $a$ form a NSI peacock. 
	In particular, $\hat{\pi} ^a$ is a Markov martingale measure uniquely defined by its marginal distributions.
\end{theorem}

\begin{proof}
	We have already proven the existence of a martingale measure $\pi$ and a corresponding martingale parametrization $(\pi ^{\alpha})_{\alpha \in [0,1]}$ that satisfies \eqref{eq:CountUniqSM} in Proposition \ref{prop:CountableExistSM}. Hence, it remains to prove the uniqueness of the family $(\pi ^{\alpha})_{\alpha \in [0,1]}$. 
	
	Let $(\rho ^{\alpha})_{\alpha \in [0,1]}$ be a martingale parametrization  that satisfies \eqref{eq:CountUniqSM}. Lemma \ref{lemma:RightDerivatives}  yields that $\alpha \mapsto \rho ^{\alpha}$ is a.e.\ right-differentiable and  the right-derivatives $(\hat{\rho} ^a)_{a \in [0,1]}$ are a family in $\mathsf{M}_{T}$. These right-derivatives determine $(\rho ^{\alpha})_{\alpha \in [0,1]}$  uniquely (see Corollary \ref{cor:ParamDeterminedByRD}) and their marginal distributions $(\mathrm{Law}_{\hat{\rho} ^a}(X_t))_{t \in T}$ are determined by \eqref{eq:CountUniqSM}. Hence, the marginal distributions of $\hat{\rho}^a$ coincide with those of $\hat{\pi}^a$ which are denoted by $(\hat{\eta} _t ^a)_{t \in T} \in \mathsf{P}_T$ (cf.\ Lemma \ref{lemma:RDofShadow}). Thus, if $(\hat{\eta} _t ^a)_{t \in T}$ is NSI for $\lambda$-a.e.\ $a \in [0,1]$, Proposition \ref{prop:NSItoUniq} implies that $\hat{\pi}^a = \hat{\rho}^a$  for $\lambda$-a.e.\ $a \in [0,1]$ and therefore we obtain $\pi^{\alpha} = \rho ^{\alpha}$ for all $\alpha \in [0,1]$ by Corollary \ref{cor:ParamDeterminedByRD}. Moreover, this would imply that $\hat{\pi}^a$ is Markov for $\lambda$-a.e.\ $a \in [0,1]$ proving the theorem.
	
	Hence, we need to show that $(\hat{\eta} _t ^a)_{t \in T}$ is a NSI peacock for $\lambda$-a.e.\ $a$. Since the parametrization $(\nu^\alpha)_{\alpha \in [0,1]}$ is $\leqcs$-convex, the familiy $(\hat{\nu}^a)_{a \in [0,1]}$ of right-derivatives is monotonously increasing w.r.t.\ $\leqcs$  and hence by Proposition \ref{prop:OptToNSI}, it is sufficient to show that 
	$((\hat{\eta}_t^a)_{t \in T})_{a\in[0,1]}$ is a solution to the optimization problem \eqref{eq:OptRD} w.r.t.\ $c_t$ and $(\tilde\nu^a)_{a\in[0,1]}:=(\hat{\nu}^a)_{a\in [0,1]}$ simultaneously for all $t \in T$. In the current setup we then have
	\begin{equation*}
	\mathcal{A} = \left\{ (\theta ^{a})_{a \in [0,1]} \ \vert \
	\begin{array}{l} \theta ^{a} \in \PP_{T}, a \mapsto \theta ^a \textit{ measurable, } \\ 
	\theta ^a _0 = \hat{\nu} ^a, \ \int _0 ^1 \theta ^{a} _t \de a = \mu _t 
	\end{array}		 \right\}.
	\end{equation*}
	It is easy to see that 	$(\hat{\eta}_t^a)_{t \in T} \in \mathcal{A}$ and 
	$\int _0 ^{\alpha} \hat{\eta} ^a _t \de a = \shadow{\mu _{T_t}} {\nu ^{\alpha}}$
	for all $t \in T$ and $\alpha \in [0,1]$. By the minimality of shadows (cf.\ Proposition \ref{prop:MaximalElement}), any competitor $(\tilde{\theta} ^a)_{a \in [0,1]} \in \mathcal{A}$ satisfies $\int _0 ^{\alpha} \tilde{\theta}^a_t \de a\leqc\int _0 ^{\alpha} \hat{\eta} ^a _t \de a$. Hence, it follows
	\begin{eqnarray*}
		\int _0 ^1 c_t (a,\tilde{\theta} ^a) \de a  
		&=& \int _0 ^1  \left( \int _a ^{1} \int _{\mathbb{R}} \frac{x + \sqrt{1 + x^2}}{2} \de \tilde{\theta}^{a} _t \de \alpha \right) \de a \\
		&=& \int _0 ^1  \left(  \int _{\mathbb{R}} \frac{x + \sqrt{1 + x^2}}{2} \de\left(\int _0 ^{\alpha}  \tilde{\theta}^{a} _t \de a\right)  \right) \de \alpha \\
		&\geq&  \int _0 ^1  \left(  \int _{\mathbb{R}} \frac{x + \sqrt{1 + x^2}}{2} \de\left(\int _0 ^{\alpha} \hat{\eta}^{a} _t \de a\right) \right) \de \alpha = \int _0 ^1 c_t (a,\hat{\eta}  ^a) \de a
	\end{eqnarray*}
	for all $t \in T$. This proves the claim. 
\end{proof}

\begin{remark} \label{rem:OtherFormulation}
	Let $(\Omega _a)_{a \in [0,1]}$ be uncountably many copies of $\mathbb{R}^T$ and set 
	\[\Omega = [0,1] \times \prod _{a \in [0,1]} \Omega _a.\]
	We equip $\Omega$ with the product $\sigma$-algebra and denote by $\mathbb{P}$ the product measure on $\Omega$ generated by $\lambda$ on $[0,1]$ and $\hat{\pi}^a$ on $\Omega _a$ for all $a \in [0,1]$ (on the Lebesgue null set where the right-derivative $\hat{\pi}^a$ are not defined we choose the Dirac mass of the null-path). It is easy to see that the random variables 
	\begin{equation*}
	U(\omega_0, (\omega _a)_{a \in [0,1]}) = \omega _0 \hspace{0.5cm} \text{ and }  \hspace{0.5cm}  M^a (\omega_0, (\omega _a)_{a \in [0,1]}) = \omega _a, \ a \in [0,1],
	\end{equation*} 
	satisfy the assertions of Theorem \ref{thm:GenExist}.
\end{remark}

\begin{remark} \label{rem:ExplainAssumptionDetail}
	
	We mentioned before that the $\leqcs$-convexity of the parametrization is crucial for our proof of the \emph{uniqueness} of the shadow martingale (Theorem \ref{prop:GenExistCount}). Note that the $\leqcs$-convexity is not mentioned in subsections prior to \ref{sec:coreArgument}. 
In the proof of Theorem \ref{prop:GenExistCount} we only use the $\leqcs$-convexity to justify the application of Proposition \ref{prop:OptToNSI} (``Optimality implies NSI'' for a family of peacocks $(\theta^a)_{a\in [0,1]}$). Since for almost every $a\in [0,1]$ the ``derivative peacock'' $(\hat\eta^a_t)_{t \in T}$ is NSI, Proposition \ref{prop:NSItoUniq} shows that the associated martingale is uniquely determined.

The proof of Proposition \ref{prop:OptToNSI} and the whole strategy of Subsection \ref{sec:coreArgument} boils down to the following observation: If the family $(\theta^a)_{a \in [0,1]}$ is a simultaneous minimizer of \eqref{eq:OptRD} with respect to the cost functions $c_t$ for all $t \in T$, we have
	\begin{equation} \label{eq:MainInequality}
		\shadow{(2 \theta ^{a} _s)_{s \in T_t}}{\theta^a _0} \leqcs 2 \theta ^a _t - \shadow{(2 \theta ^a _s)_{s \in T_t}}{\theta^a _0}  \leqcs \shadow{(2 \theta ^{b} _s)_{s \in T_t}}{\theta^b _0}
	\end{equation}
	for all $t$ and for $\lambda$-a.e.\ $a < b$ in $[0,1]$. Whereas the left inequality is a direct consequence of the defining properties of the shadow, the right inequality is a highly non-trivial statement for which we carefully study the monotonicity principle associated to \eqref{eq:OptRD}. For details we refer to the proof of Lemma \ref{lemma:OptToOrder}. So far we cannot hope that \eqref{eq:MainInequality} or even the weaker statement
	\begin{equation*}
		\shadow{(2 \theta ^{a} _s)_{s \in T_t}}{\theta^a _0} \leqcs \shadow{(2 \theta ^{b} _s)_{s \in T_t}}{\theta^b _0}
	\end{equation*}
	is satisfied for $\lambda$-a.e.\ $a<b$ if the map  $a \mapsto \theta_0^a $ is  not monotonously increasing w.r.t.\ $\leqcs$. In the setting of Theorem \ref{prop:GenExistCount} we have $\theta_0^a = \hat{\nu}^a$ where $(\hat\nu^a)_{a \in [0,1]}$ is the family of right-derivatives of the parametrization $(\nu^{\alpha}) _{\alpha \in [0,1]}$ and $(\hat\nu^a)_{a \in [0,1]}$ is monotonously increasing w.r.t.\ $\leqcs$ if and only if $(\nu^{\alpha}) _{\alpha \in [0,1]}$ is $\leq_{c,s}$-convex.
	
	Of course, there might exists a way to show that we have
	\begin{equation*} 
		\shadow{(2 \theta ^{a} _s)_{s \in T_t}}{\theta^a _0} = 2 \theta ^a _t - \shadow{(2 \theta ^a _s)_{s \in T_t}}{\theta^a _0} 
	\end{equation*}
--i.e the NSI property for $(\theta^a_t)_{t\in T}$-- for $\lambda$-a.e.\ $a \in [0,1]$  without relying on \eqref{eq:MainInequality}. However, for infinite $T$ and without continuity of $\mu _{S} \mapsto \shadow{\mu _S}{\nu}$ (see Subsection \ref{sec:contshadow} and Example \ref{expl:ShadowNotCont}) we don't see another possible line of reasoning so far.
\end{remark}

\section{C\`adl\`ag shadow martingales indexed by a continuous time set} \label{sec:ContTimeRC}

In this section we show how the results of the previous section can be lifted to the setting of a continuous time index set $T \subset [0, \infty)$ with minimal element $0 \in T$ under the additional assumption that the given peacock $(\mu _t)_{t \in T}$ is right-continuous, i.e.\ the map $t \mapsto \mu_t$ is a right-continuous map from $T$ to $\PO(\mathbb{R})$ (under $\TO$). 

The key observation is that in the current setup martingale measures $\pi$ (similarly for martingale parametrizations) are uniquely determined by the restriction to a well chosen countable index set $S$, i.e.\ there exists a unique martingale measure $\pi$ extending $\pi_{|S}$ to the index set $T$, cf.\ Lemma \ref{lemma:RCExtensionMartingale}. This will be established in Section \ref{sec:ContTimeMartMeas}. In Subsection \ref{sec:ShadObstbyPCOC} we will show that also the obstructed shadow only depends on $\mu_S$ for some countable family $S$ if $\mu_T$ is a peacock. Consequently, also the NSI property only depends on $\mu_S$ by Lemma \ref{lemma:CharacNSI}. These results will allow us to provide a proof of Theorem \ref{thm:GenExistRC}, a variant of Theorem  \ref{thm:GenExist} in the case of a continous time index set $T$ and a right-continuous peacock $\mu_T$ in Subsection \ref{sec:GenExistRC}.

\subsection{Continuous time martingale measures} \label{sec:ContTimeMartMeas}

We fix a subset $T \subset[0,\infty)$ with $0 \in T$ and we equip $T$  with the inherited standard topology. Recall that a modification of the canonical process $(X_t)_{t \in T}$ under $\pi \in \PO(\mathbb{R}^T)$ is a process $\tilde{X}: \mathbb{R}^T \rightarrow \mathbb{R}$ such that $\pi(\tilde{X}_t = X_t) = 1$ for all $t \in T$.  Note that $\mathrm{Law}_{\pi}(X) = \mathrm{Law}_{\pi}(\tilde{X}) = \pi$ and, if $T$ is countable, we get $\tilde{X} = X$ $\pi$-a.e.

\begin{definition}
	\begin{enumerate}
		\item [(i)] We call a peacock $(\mu _t)_{t \in T}$  right-continuous, if $t \mapsto \mu _t$ is right-continuous from $T$ to $\PO(\mathbb{R})$ w.r.t.\ $\TO$. We denote by  $\PPC_T\subset \PP_T$ the set of these peacocks.

		\item [(ii)] We call a martingale measure $\pi \in \MM _{T}$  a c\`adl\`ag martingale measure, if there exists a modification $(\tilde{X}_t)_{t \in T}$ of the canonical process under $\pi$ such that $t \mapsto \tilde{X}_t(\omega)$ is a c\`adl\`ag function for all $\omega \in \mathbb{R}^T$. We denote the set of all c\`adl\`ag martingale measures by $\MMC_T$.
		
		\item [(iii)] We call a martingale parametrization $(\pi ^{\alpha})_{\alpha \in [0,1]}$ of a c\`adl\`ag martingale measure $\pi \in \MMC_T$  c\`{a}dl\`{a}g, if $\frac{1}{\alpha} \pi ^{\alpha} \in \MMC_T$ for all $\alpha \in (0,1]$.
	\end{enumerate}
\end{definition}

\begin{remark} \label{rem:ObviuousRCMart}
	\begin{enumerate}
		\item [(i)] Lemma \ref{lemma:convDom} shows that a peacock is right-continuous (w.r.t.\ $\TO$)  if and only if $t \mapsto \mu _t$ is right-continuous w.r.t.\ $\TZ$ because for all $t \in T$ and $t_n \downarrow t$ the measures $\mu_t$ and $(\mu _{t_n})_{n \in \mathbb{N}}$ are bounded from above in convex order by $\mu _{t_1}$. 
		\item [(ii)] Let $\pi \in \MMC _T$ and $(\tilde{X}_t)_{t \in T}$ be a modification of the canonical process $(X_t)_{t \in T}$ under $\pi$. Then $(\tilde{X}_t)_{t \in T}$ is an $(\tilde{\mathcal{F}}_t)_{t \in T}$-martingale under $\pi$ where $\tilde{\mathcal{F}}_t = \sigma(\{ \tilde{X}_s :  s \in [0,t]\cap T\})$ because $\pi(\tilde{X}_t = X_t) = 1$ for all $t \in T$. 
	\end{enumerate}
\end{remark}

We equip both $\PPC_T$ and $\MMC_T$ with the subspace topology inherited from the product topology on $(\PO(\mathbb{R}))^T$ and $\TO$ on $\MM_T$. Moreover, we use the notation $\MMC_T((\mu _t)_{t \in T})$ for c\`adl\`ag martingale measures associated with a peacock $(\mu _t)_{t \in T}$.
Note that the right-continuity of the peacock corresponds to the c\`{a}dl\`{a}g property of the martingale measure:

\begin{lemma} \label{lemma:RCandcadlag}
	Let  $(\mu _t)_{t \in T} \in \PP_T$ and $\pi \in \mathsf{M}_T$ be associated with $(\mu _t)_{t \in T}$.
	\begin{enumerate}
		\item [(i)] If $\pi$ is a c\`adl\`ag martingale measure, then $(\mu _t)_{t \in [0,1]} \in \PPC_T$.
		\item [(ii)] If $(\mu _t)_{t \in T}$ is a right-continuous peacock, then $\pi \in  \MMC_T$.
	\end{enumerate}
\end{lemma}

\begin{proof}
This is a consequence of Remark \ref{rem:ObviuousRCMart} and of standard results on modifications for continuous martingales, see e.g.\ \cite[Theorem 2.8]{ReYo99}.
\end{proof}

\begin{lemma} \label{lemma:RDenseSubset}
	There exists a countable set $S \subset T$ that is right-dense in $T$, i.e.\ for all $t \in T$ there exists a sequence $(s_n)_{n \in \mathbb{N}}$ in $\{s \in S : s \geq t\}$ that converges to $t$. Similarly, there exists a countable left-dense set and the union of both is a countable subset of $T$ that is both right- and left-dense.
\end{lemma}

\begin{proof} 
	Let $S_1 = \{ t \in T \ \vert \ \exists\ \varepsilon > 0,\  T \cap (t,t+ \varepsilon)= \emptyset \}$ be the set of points that are ``right-isolated'' in $T$ and $S_2$ be a countable dense subset of $T$. Since $S_1$ is countable so is $S:=S_1\cup S_2$. Moreover, it is not difficult to check that any $t \in T$ is the limit of a sequence in $S\cap [t,\infty)$. Thus, $S$ is a countable right-dense subset of $T$.
\end{proof}

The following  Lemmas \ref{lemma:RCExtensionMartingale}--\ref{lemma:ParamExtension}  link c\`adl\`ag martingale measures and c\`adl\`ag martingale parametrizations to their restrictions to a suitable right-dense set: 

\begin{lemma} \label{lemma:RCExtensionMartingale}
	Let $S$ be a  countable right-dense subset of $T$.
	\begin{enumerate}
		\item[(i)] If $\pi, \rho \in \MMC_T$ satisfy $\pi _{|S} = \rho _{|S}$, then $\pi = \rho$.
		\item[(ii)] For all $\pi' \in \MMC_S$, there exists  $\pi \in \MMC_T$ such that $\pi _{|S} = \pi'$ (uniquely determined by (i)). 
	\end{enumerate}
\end{lemma} 

\begin{proof}
	Item (i): Let $n \in \mathbb{N}$, $t_1, \ldots, t_n \in T$ and $\varphi \in C_b(\mathbb{R}^n)$. We can find sequences $(s^i_k)_{k \in \mathbb{N}}$ in $S$, $1 \leq i \leq n$, such that $s ^i _k \searrow t_i$ for all $i \in \{1, \ldots,n\}$. Since there exist c\`{a}dl\`{a}g modifications of the canonical process under $\pi$ and $\rho$ and $\pi _{|S} = \rho _{|S}$,  we obtain that $\int _{\mathbb{R}^T} \varphi (X_{t_1}, \ldots, X_{t_n}) \de \pi = \int _{\mathbb{R}^T} \varphi (X_{t_1}, \ldots, X_{t_n}) \de \rho$.
	
	Item (ii): Since $\pi' \in \MMC_S$, there exists a modification $(\tilde{X}_s)_{s \in S}$ of the canonical process $(X_s)_{s \in S}$ on $\mathbb{R}^S$ under $\pi'$ such that any path $s \mapsto X_s(\omega)$ is  c\`{a}dl\`{a}g. Furthermore, also the limits $\lim _{s \uparrow t, s \in S} \tilde{X}_s$ and $\lim _{s \downarrow t, s \in S} \tilde{X}_s$ for $t \in T \setminus S$ exist (see \cite[Proposition 1]{BeHuSt16}).	We define \begin{equation}
	\mathbb{R}^S \ni \omega \mapsto Y_t(\omega) = \begin{cases} \tilde{X}_t(\omega) &t \in S \\
	\lim _{s \downarrow t} \tilde{X}_s(\omega) &t \not \in S .
	\end{cases}
	\end{equation}
	The family $(Y_t)_{t \in T}$ is a well-defined process on the probability space $(\mathbb{R}^S, \mathcal{F}^S_{\infty}, \pi ')$ where $\mathcal{F}^S_{\infty}$ denotes the $\sigma$-algebra $\mathcal{F} ^S_{\infty} = \bigvee _{s \in S} \mathcal{F}^S_s$.  Here, for definitess, we denote the canonical filtration on $\R^S$ by $(\mathcal F_s^S)_{s\in S}.$ Moreover, $Y$ is a martingale w.r.t.\ the right-continuous filtration $(\overline{\mathcal{F}}^+_t)_{t \in T}$  on $\R^T$ given by $\overline{\mathcal{F}}^+_t = \bigcap _{s > t, s \in S}  \mathcal{F}^S_s$ for all $t \in T$ that are not the maximal element and $ \overline{\mathcal{F}}^+_{\tmax} = \mathcal{F}^S_{\tmax}$ if there exists a maximal element $\tmax \in T$.	
	Set $\pi = \mathrm{Law}_{\pi'}(Y) \in \PO(\mathbb{R}^T)$. Then $\pi$ is a martingale measure
	because $Y$ is a martingale under $\pi'$.
	By the c\`{a}dl\`{a}g property of $Y$, the marginal distributions of the canonical process on $\mathbb{R}^T$ under $\pi$ are right-continuous. Hence, $\pi\in \MMC_T$ by Lemma \ref{lemma:RCandcadlag}.
	By construction, it holds $\pi _{|S} = \pi'$.
\end{proof}

\begin{lemma}  \label{lemma:CompactnessTD}
	Assume $T$ has a maximal element $\tmax$ and $(\pi ^n)_{n \in \mathbb{N}}$ is a sequence in $\MMC_{T}$.
	\begin{enumerate}
		\item [(i)] Let $\pi \in \MMC_{T}$ and $S$ be a countable right-dense subset of $T$. If $({\pi ^n}_{|S})_{n \in \mathbb{N}}$ converges to $\pi _{|S}$ in $\MMC_S$ (under $\TO$) and $(\mathrm{Law}_{\pi ^n}(X_t))_{n \in \mathbb{N}}$ converges to $\mathrm{Law}_{\pi}(X_t)$ in $\PO(\mathbb{R})$ for all $t \in T$, then $(\pi _n)_{n \in \mathbb{N}}$ converges to $\pi$ in $\MMC_T$.
		
		\item [(ii)] If the sequence of right-continuous peacocks $((\mathrm{Law}_{\pi^n}(X_t))_{t \in T})_{n \in \mathbb{N}}$ is convergent in $\PPC_{T}$ with limit $(\mu _t)_{t \in T}$, then there exists a convergent subsequence of  $(\pi ^n)_{n \in \mathbb{N}}$ in $\MMC_T$ with limit $\pi \in \MMC_{T}((\mu _t)_{t \in T})$.
	\end{enumerate}
\end{lemma}

\begin{proof}
	Let $S' \subset T$ with $\tmax \in S'$ and suppose that $\mathcal{A}$ is a compact subset of $\PO (\mathbb{R})$ under $\TO$. A slight modification of \cite[Lemma 1]{BeHuSt16} shows that the set 
	\begin{equation} \label{eq:CompactSet}
	\left\{ \rho \in \MM_{S'} \ \vert \  \mathrm{Law} _{\rho} (X_{\tmax}) \in \mathcal{A} \right\}
	\end{equation}
	is a compact subset of $\MM_{S'}$ (under $\TO$).
	
	Item (i):  By the definition of the topology $\TO$ on $\mathcal{P}(\R^T)$, the sequence $(\pi _n)_{n \in \mathbb{N}}$ converges to $\pi$ under $\TO$ if and only if $(\pi ^n _{|S \cup R})_{n \in \mathbb{N}}$ converges to $\pi _{|S \cup R}$ in $\MMC_{S \cup R}$ under $\TO$ for all finite $R \subset T$. Thus, it is sufficient to show that for any finite subset $R \subset T$, any subsequence of $(\pi ^n _{|S \cup R})_{n \in \mathbb{N}}$ has a subsequence that converges to $\pi _{|S \cup R}$ under $\TO$.

	Let $R \subset T$ finite and note that $S$ contains $\tmax$. 
	Let $(\pi ^{n_k} _{|S \cup R})_{k \in \mathbb{N}}$ be an arbitrary subsequence of $(\pi ^{n} _{|S \cup R})_{n \in \mathbb{N}}$. 
	Since $(\mathrm{Law}_{\pi ^n}(X_{\tmax}))_{n \in \mathbb{N}}$ converges to $\mathrm{Law}_{\pi}(X_{\tmax})$ under $\TO$, the set $\mathcal A_1:=\{ \mathrm{Law}_{\pi ^n}(X_{\tmax}) : n \in \mathbb{N}\} \cup \{\mathrm{Law}_{\pi}(X_{\tmax})\} \subset \PO(\mathbb{R})$ is compact w.r.t.\ $\TO$. Hence, the set in \eqref{eq:CompactSet} with  $\mathcal A :=\mathcal{A}_1$ and $S' := S \cup R$ is a compact subset of $\MM_{S \cup R}$ under $\TO$. Consequently, $(\pi ^{n_k} _{|S \cup R})_{k \in \mathbb{N}}$ has itself a convergent  subsequence with limit $\rho ^* \in \mathsf{M}_{S \cup R}$. 
	By assumption, the marginal distributions of $\rho ^*$ have to coincide with the right-continuous marginal distributions of $\pi _{|S \cup R}$, and thus $\rho^* \in \MMC_{S \cup R}$ by Lemma \ref{lemma:RCandcadlag} (ii).
	Since $(\pi ^n _{|S})_{n \in \mathbb{N}}$ is convergent to $\pi_{|S}$, it holds $\rho ^* _{|S} = \pi _{|S}$ and therefore Lemma \ref{lemma:RCExtensionMartingale} (i) yields $\rho^* = \pi _{|S \cup R}$.
	
	Item (ii): Let $S$ be a countable right-dense subset of $T$ that includes $\tmax$.
	Since $(\mathrm{Law}_{\pi ^n}(X_{\tmax}))_{n \in \mathbb{N}}$ converges to $\mu _{\tmax}$ under $\TO$, the set $\mathcal A_2:=\{ \mathrm{Law}_{\pi ^n}(X_{\tmax}) : n \in \mathbb{N}\} \cup \{\mu_{\tmax}\} \subset \PO(\mathbb{R})$ is compact w.r.t.\ $\TO$. Hence, the set in \eqref{eq:CompactSet} with  $\mathcal A :=\mathcal{A}_2$ and $S' := S$ is a compact subset of $\MM_{S}$ under $\TO$.
	Therefore, there exists a convergent subsequence $({\pi ^{n_k}} _{|S})_{k \in \mathbb{N}}$ with limit $\pi ^S\in \mathsf{M}_S$. The marginal distributions of $\pi ^S$ are $(\mu _t)_{t \in S}$. Since $(\mu _t)_{t \in T}$ is a right-continuous peacock we can extend $\pi ^S$ to some $\pi  \in \MMC_{T}((\mu _t)_{t \in T})$ with ${\pi}_{|S} = \pi ^S$ by Lemma \ref{lemma:RCExtensionMartingale} (ii). Item (i) yields that  $(\pi ^{n_k})_{k \in \mathbb{N}}$ converges to $\pi$ in $\MMC_T$.
\end{proof}

\begin{lemma} \label{lemma:ParamExtension}
	Let $T \subset [0,\infty)$ and $S$ be a countable right-dense subset of $T$. 
	\begin{enumerate}
		\item[(i)] If $(\pi ^{\alpha})_{\alpha \in [0,1]}$ is a  martingale parametrization of the c\`adl\`ag martingale measure $\pi$, then $(\pi _{|S} ^{\alpha})_{\alpha \in [0,1]}$ is a c\`adl\`ag martingale parametrization of $\pi _{|S}$.
		\item[(ii)] Let $(\tilde{\pi} ^{\alpha})_{\alpha \in [0,1]}$ be a c\`adl\`ag martingale parametrization of $\tilde{\pi} \in \MMC_{S}$ and $\pi \in \MMC_{T}$  the unique extension of $\tilde\pi$ given by Lemma \ref{lemma:RCExtensionMartingale}, i.e.\ $\pi _{|S} = \tilde{\pi}$. There exists a unique c\`adl\`ag martingale parametrization  $(\pi ^{\alpha})_{\alpha \in [0,1]}$ of $\pi$ such that $\pi^{\alpha} _{|S} = \tilde{\pi}^{\alpha}$ for all $\alpha \in [0,1]$.
		\item [(iii)] Let $(\pi ^{\alpha})_{\alpha \in [0,1]}$ be a  martingale parametrization of $\pi \in \MMC_T$. If $\alpha \mapsto \pi ^{\alpha} _{|S}$ is right-differentiable at $a \in [0,1)$ (in the sense of Lemma \ref{lemma:RightDerivatives}) and there exists  $(\eta _t)_{t \in T} \in \PPC_T$ such that the law of $X_t$ under $\frac{1}{h}(\pi ^{a + h} - \pi ^a )$ converges under $\TO$ to $\eta _t$ for all $t \in T$, then the map $\alpha \mapsto \pi ^{\alpha}$ is right-differentiable at $a$, i.e.\ 
		\begin{equation*}
		\hat{\pi}^{a} = \lim _{h \downarrow 0} \frac{\pi ^{a + h} - \pi ^{a}}{h}
		\end{equation*}
		exists as a limit in $\PO(\mathbb{R}^T)$ under $\TO$.
		Moreover, the right-derivative of $\alpha \mapsto \pi ^{\alpha}$ at $a$ is an element of $\MMC_T((\eta _t)_{t \in T})$ and its restriction to $S$ is the right-derivative of $\alpha \mapsto \pi ^{\alpha} _{|S}$ at $a$. 
	\end{enumerate}
\end{lemma}

\begin{proof}
	Item (i): It is straightforward to check that $(\pi _{|S} ^{\alpha})_{\alpha \in [0,1]}$ is a (c\`adl\`ag)  martingale parametrization of $\pi _{|S} \in \MMC_S$. 
	
	Item (ii): 
	For all $\alpha \in [0,1]$ let $\frac{1}{\alpha}\pi^{\alpha} \in \MMC_{T}$ be the unique extension of $\frac{1}{\alpha}\tilde{\pi}^{\alpha}$ given by Lemma \ref{lemma:RCExtensionMartingale} (ii). Then $\pi ^{\alpha}(\mathbb{R}^{T}) = \alpha$ and $\pi ^1 = \pi$. 
	Finally, $\pi^\alpha\leq_+\pi^\beta$ follows by  
	considering nonnegative cylinder functions (because they generate the $\sigma$-algebra on $\R^T$).
	
	Item (iii): Recall that the martingale property of a sequence in $\PO(\mathbb{R}^T)$ is preserved under convergence in $\TO$. The claim follows from Lemma \ref{lemma:CompactnessTD} (i).
\end{proof}

\subsection{Shadows obstructed by peacocks and the NSI property} \label{sec:ShadObstbyPCOC}
In this section we consider shadows in the special -- and for us most important -- case  in which $(\mu_t)_{t\in T}$ is a peacock. A particular consequence of this assumption, which is not true without the peacock assumption, is that the shadow is uniquely determined by a countable subset of obstructions, i.e.\ by marginal constraints $(\mu_t)_{t\in S}$ for a countable set $S\subset T$. By combining with Lemma \ref{lemma:CharacNSI} we see that the NSI property is determined by a well chosen countable subset of obstructions.

\begin{lemma} \label{lemma:ShadowRightCont}
	Let $\nu \leqcp \mu _{T}$ and $s \in T$. In the topology $\TO$ we have the following:
	\begin{enumerate}
		\item [(i)] If $t \mapsto \mu _t$ is left-continuous at $s$, then $t \mapsto \shadow{\mu _{T_t}}{\nu}$ is left-continuous at $s$.
		\item [(ii)] If $t \mapsto \mu _t$ is right-continuous at $s$, then $t \mapsto \shadow{\mu _{T_t}}{\nu}$ is right-continuous at $s$.
	\end{enumerate}
\end{lemma}

Recall that since $\mu _T$ is a peacock, left- and right-continuity of $t \mapsto \mu _t$ is independent from the choice of $\TO$ or $\TZ$ (cf.\ Remark \ref{rem:ObviuousRCMart} (i)).

\begin{proof}
	Item (i):  Let $(t_n)_{n \in \mathbb{N}}$ be a sequence that converges to $s$ from below. We define the family $(\eta _t)_{t \in T_s}$ in $\MO(\mathbb{R})$ by
	\begin{equation*}
	\eta _t = \begin{cases}
	\shadow{\mu _{T_t}}{\nu} & t < s \\
	\lim _{t \uparrow s} \shadow{\mu _{T_t}}{\nu} & t = s
	\end{cases}.
	\end{equation*}
	The  limit $\lim _{t \uparrow s} \shadow{\mu _{T_t}}{\nu}$ exists by Lemma \ref{lemma:CSupSum} (ii) in conjunction with Lemma \ref{lemma:ShadowInclusion} (i). Clearly, $\nu \leqc \eta _t \leqc \eta _u$ for all $t \leq u$ in $T_s$ and $\eta _t \leqp \mu _t$ for all $t < s$ in $T$. Moreover, since $\eta _s = \lim _{n \rightarrow \infty} \eta _{t_n}$ and $\eta _{t_n} \leqp \mu _{t_n}$ by Lemma \ref{lemma:ShadowInclusion} (v), we obtain $\eta _s \leqp \mu_s$ from  Lemma \ref{lemma:preservingOrder}. The claim follows by Proposition \ref{prop:MaximalElement}.
	
	Item (ii): Let $(t_n)_{n \in \mathbb{N}}$ be a sequence that converges to $s$ from above. 
	It holds $\shadow{\mu _{T_{t_n}}}{\nu} = \shadow{\mu _{T_{s, t_n}}}{\shadow{\mu _{T_s}}{\nu}}$ for all $n \in \mathbb{N}$ where $T_{s,t_n} = \{ t \in T: s < t \leq t_n \}$.
	Proposition \ref{prop:GeneralSchadow} states that there exists a sequence of finite sets $(R_k)_{k \in \mathbb{N}}$ such that $(\shadow{\mu _{R_k}}{\shadow{\mu _{T_s}}{\nu}})_{k \in \mathbb{N}}$ converges to $\shadow{\mu _{T_{s, t_n}}}{\shadow{\mu _{T_s}}{\nu}}=\shadow{\mu _{T_{t_n}}}{\nu}$ and a well-chosen telescopic application of Lemma \ref{lemma:shadowPF} to $(\shadow{\mu _{R_k}}{\shadow{\mu _{T_s}}{\nu}})_{k \in \mathbb{N}}$  in conjunction with Lemma \ref{lemma:preservingOrder} implies
	\[U(\shadow{\mu _{R_k}}{\shadow{\mu _{T_s}}{\nu}})-U(\shadow{\mu _{T_s}}{\nu})\leq U(\mu_{\max R_k})-U(\mu_s)\leq U(\mu_{t_n})-U(\mu_s).\]
	Letting $k$ tend to infinity yields $U(\shadow{\mu _{T_{t_n}}}{\nu})-U(\shadow{\mu _{T_s}}{\nu})\leq U(\mu_{t_n})-U(\mu_s)$ and since $(\mu _t)_{t \in T}$ is right-continuous, by Lemma \ref{lemma:propertiesPotF} (ii)we obtain
	\begin{equation*}
	0 \leq \lim _{n \rightarrow \infty} U(\shadow{\mu _{T_{t_n}}}{\nu}) - U(\shadow{\mu _{T_t}}{\nu}) \leq \lim _{n \rightarrow \infty} U(\mu _{t_n}) - U(\mu _t) = 0.\qedhere
	\end{equation*}
\end{proof}

If $(\mu _t)_{t \in T}$ is a peacock, Corollary \ref{cor:MonotonicityPCOC} yields that the map $t \mapsto \mu _t$ is continuous outside of a set $S \subset T$ that is (at most) countable and thus also $t \mapsto \shadow{\mu _{T_t}}{\nu}$ is continuous on $T \setminus S$.

Proposition \ref{prop:GeneralSchadow} states that any generalized obstructed shadow $\shadow{\mu _T}{\nu}$ can be approximated by a sequence of finitely obstructed shadows $(\shadow{\mu _{R_n}}{\nu})_{n \in \mathbb{N}}$. The continuity of $t \mapsto \shadow{\mu _{T_t}}{\nu}$ on $T \setminus S$ leads to a huge degree of freedom when choosing this family of nested finite sets $(R_n)_{n \in \mathbb{N}}$:

\begin{lemma} \label{lemma:RCShadowApprox}
	Let $(R_n)_{n \in \mathbb{N}}$ be a nested sequence of finite subsets of $T$.
	If the union $S = \bigcup_{n \in \mathbb{N}} R_n$ is both right- and left-dense in $T$ and contains all time points where $t \mapsto \mu _t$ is not continuous, then $(\shadow{\mu_{(R_n)_t}}{\nu})_{n \in \mathbb{N}}$  converges to $\shadow{\mu _{T_t}}{\nu}$ under $\TO$ for all $t \in T$ and $\nu \leqcp \mu _T$.
\end{lemma}

\begin{proof}
	Fix $t \in T$. The sequence $(\shadow{\mu _{(R_n)_t}}{\nu})_{n \in \mathbb{N}}$ is increasing in convex order as $n$ tends to infinity and is bounded in convex-positive order by $\mu _t$. Hence, Lemma \ref{lemma:CSupSum} (ii) yields that this sequence is converging in $\MO(\mathbb{R})$ under $\TO$. We denote the limit by $\eta _t$ and 
	set $[t]_n := \max (R_n)_t$. Lemma \ref{lemma:ShadowInclusion} (v) implies that $\shadow{\mu _{(R_n)_t}}{\nu} \leq _+ \mu _{[t]_n}$ and since $S$ is both right- and left-dense in $T$, $([t]_n)_{n \in \mathbb{N}}$ converges to $t$ from below. 
	If $t \in S$, then $[t]_n = t$ for $n$ large enough, and, if $t \not \in S$, then $(\mu _{[t]_n})_{n \in \mathbb{N}}$ converges to $\mu _t$ under $\TO$. Hence, in both cases Lemma \ref{lemma:preservingOrder} implies that  $\eta _t \leqp \mu _t$.
	By 	Lemma \ref{lemma:preservingOrder}, the convex-order is preserved under convergence in $\TO$ and thus it holds $\nu \leqc \eta _s  \leqc \eta _t \leqp \mu _{t}$ 
for all $s \leq t$ in $T$. Fix $u \in T$. Any other family $(\eta' _t)_{t\in T}$ which satisfies the previous ordering relations for all $s \leq t$ in $T_u$, satisfies them in particular for all $s \leq t$ in $(R_n)_u$. Thus, Proposition \ref{prop:MaximalElement} applied to $(R_n)_u$ yields $\shadow{\mu _{(R_n)_u}}{\nu} \leqc \eta' _{[u]_n}$ and therefore
	\begin{equation*}
	\eta _u =  \lim _{n \rightarrow \infty} \shadow{\mu _{(R_n)_u}}{\nu} \leqc \mathrm{Csup} \{ \eta ' _{[u]_n} : n \in \mathbb{N} \} \leqc \eta _u '.
	\end{equation*}	
	As a consequence of Proposition \ref{prop:MaximalElement} applied to the index set $T_u$ we get for all $u \in T$
	\begin{equation*}
	\shadow{\mu _{T_u}}{\nu}  = \eta _u  = \lim _{n \rightarrow \infty} \mathcal{S}^{\mu _{(R_n) _u}}(\nu). \qedhere 
	\end{equation*}
\end{proof}

\begin{corollary} \label{cor:CountableObstr}
	There exists a countable right- and left-dense set $S \subset T$ such that $\shadow{\mu _{S_t}}{\nu} = \shadow{\mu _{T_t}}{\nu} $
	for all $t \in T$ and for all $\nu \leqp \mu _0$. In fact, $S$ can be any countable right- and left-dense subset of $T$ which contains all discontinuity points of the map $t \mapsto \mu _t$.
\end{corollary}

\begin{proof}
	Pick by Corollary \ref{cor:MonotonicityPCOC} a countable right- and left-dense subset $S \subset T$ such that $t \mapsto \mu _t$ is continuous on $T \setminus S$.	
	Since $S$ is countable, there exists a nested sequence $(R_n)_{n \in \mathbb{N}}$ such that $\bigcup _{n \in \mathbb{N}} R_n = S$. Let $t$ be in $T$. Lemma \ref{lemma:RCShadowApprox} yields $	\shadow{\mu _{T_t}}{\nu} = \lim _{n \rightarrow \infty}\shadow{\mu _{(R_n)_t}}{\nu} $  
	for all $\nu \leq _{+} \mu _0$ and Lemma \ref{lemma:ShadowInclusion} (iii) implies that $(\shadow{\mu _{(R_n)_t}}{\nu})_{n \in \mathbb{N}}$ converges to $\shadow{\mu _{S_t}}{\nu}$.
\end{proof}

\begin{corollary} \label{cor:NSICountable}
	Let $(\mu _t)_{t \in T}$ be a peacock and $S$ be a countable left- and right-dense subset $S\subset T\subset [0,\infty)$ including $0$ and all time points where $t \mapsto \mu _t$ is not continuous. Then $(\mu _t)_{t \in T}$ is NSI if and only if $(\mu _t)_{t \in S}$ is NSI.
\end{corollary}

\begin{proof}
	This is an easy consequence of Corollary \ref{cor:CountableObstr}, Lemma \ref{lemma:ShadowRightCont}  and Lemma \ref{lemma:CharacNSI} (characterization of NSI property via generalized obstructed shadows). 
\end{proof}

\subsection{Existence and uniqueness of right-continuous shadow martingales} \label{sec:GenExistRC}

In this subsection, we prove the following right-continuous version of Theorem \ref{thm:GenExist} (cf.\ Remark \ref{rem:OtherFormulation}):

\begin{theorem} \label{thm:GenExistRC}
	Let $T \subset [0,\infty)$ with $0 \in T$, $(\mu _t)_{t \in T}$ be a right-continuous peacock and $(\nu ^{\alpha})_{\alpha \in [0,1]}$ a $\leq_{c,s}$-convex parametrization of $\mu _0$. There exists a unique pair $(\pi,(\pi ^{\alpha})_{\alpha \in [0,1]})$ where $\pi \in \MMC_{T}((\mu _t)_{t \in T})$ solves the peacock problem w.r.t.\ $(\mu _t)_{t \in T}$, $(\pi ^{\alpha})_{\alpha \in [0,1]}$ is a c\`adl\`ag martingale parametrization of $\pi$ w.r.t.\ $(\nu ^{\alpha})_{\alpha \in [0,1]}$ and for all $(\alpha,t)\in[0,1]\times T$,
	\begin{equation} \label{eq:71ShadowCurve}
	\pi^{\alpha} (X_t \in \cdot) = \shadow {\mu _{[0,t]}} {\nu ^{\alpha}}.
	\end{equation}	
	Moreover, there exists a Borel set $A \subset [0,1]$ with $\lambda(A) = 1$ such that for all $a \in A$
	the map $\alpha \mapsto \pi ^{\alpha}$ is right-differentiable at $a$ and the marginals of the right-derivative $\hat{\pi}^a$ at $a$ form a NSI peacock. 
	In particular, $\hat{\pi} ^a$ is a Markov martingale measure uniquely defined by its marginal distributions.
\end{theorem}

\begin{proof}
	For \textsf{STEPS 1--4} we assume that $T$ admits a maximal element, i.e.\ $\tmax := \sup T \in T$, and we fix a countable left- and right-dense subset $S$ of $T$ that contains both $0$ and $\tmax$, and all time points where $t \mapsto \mu _t$ is not continuous (cf.\ Corollary \ref{cor:MonotonicityPCOC} and Lemma \ref{lemma:RDenseSubset})
	
	\textsf{STEP 1:} \emph{We show that there exists $\pi \in \MMC_{T}((\mu _t)_{t \in T})$ and a martingale parametrization  $(\pi ^{\alpha})_{\alpha \in [0,1]}$ of $\pi$ w.r.t.\ $(\nu ^{\alpha})_{\alpha \in [0,1]}$ that satisfies \eqref{eq:71ShadowCurve}.} Since $S$  is a countable set and $\tmax \in S$, Theorem \ref{prop:GenExistCount} implies that there exists $\tilde{\pi} \in \mathsf{M}_{S}((\mu _t)_{t \in S})$ and a martingale parametrization  $(\tilde{\pi}^{\alpha})_{\alpha \in [0,1]}$ of $\tilde{\pi}$ w.r.t.\  $(\nu ^{\alpha})_{\alpha \in [0,1]}$ that satisfies 
	\begin{equation} \label{eq:71ShadowCurveCount}
	\tilde{\pi}^{\alpha}(X_t \in \cdot) = \shadow{\mu _{S_t}}{\nu ^{\alpha}}
	\end{equation}
	for all $t \in S$ and $\alpha \in [0,1]$.
	The peacock $(\mu_t)_{t \in S}$ is right-continuous, thus Lemma \ref{lemma:ShadowRightCont} implies that the map $t \mapsto \shadow{\mu _{S_t}}{\nu ^{\alpha}}$ is right-continuous for every $\alpha \in [0,1]$.
	Therefore, by Lemma \ref{lemma:RCandcadlag} (ii) applied to $\tilde{\pi}$ and $\frac1\alpha \tilde{\pi}^\alpha$ for every $\alpha\in (0,1)$, we have both  $\tilde{\pi} \in \MMC_{S}$  and $(\tilde{\pi}^{\alpha})_{\alpha \in [0,1]}$ is a c\`adl\`ag martingale parametrization.
	We can uniquely extend $\tilde{\pi}$ to $\pi \in \MMC_{T}((\mu _t)_{t \in T})$ by Lemma \ref{lemma:RCExtensionMartingale} (ii) and Lemma \ref{lemma:ParamExtension} (ii) shows that we can extend the parametrization  $(\tilde{\pi}^{\alpha})_{\alpha \in [0,1]}$ to a martingale parametrization $(\pi^{\alpha})_{\alpha \in [0,1]}$ of $\pi$. 
	The set $S$ contains all discontinuities of the map $t \mapsto \mu _t$ on $T$ and hence Corollary \ref{cor:CountableObstr} yields for all $t \in S$ and $\alpha \in [0,1]$  the equality
	\begin{equation} 
	\pi ^{\alpha}(X_t \in \cdot ) = \tilde{\pi}^{\alpha}(X_t \in \cdot) \overset{\eqref{eq:71ShadowCurveCount}}{=} \shadow{\mu _{S_t}}{\nu ^{\alpha}} = \shadow{\mu _{[0,t]}}{\nu ^{\alpha}}.
	\end{equation}
	Since both $t \mapsto \pi ^{\alpha}(X_t \in \cdot )$ and $t \mapsto \shadow{\mu _{T_t}}{\nu ^{\alpha}}$ are right-continuous functions from $T$ to $\MO(\mathbb{R})$ w.r.t.\ $\TO$ (see Lemma \ref{lemma:RCandcadlag} (i) and Lemma \ref{lemma:ShadowRightCont}), we deduce that $\pi^{\alpha}$ satisfies \eqref{eq:71ShadowCurve} for all $t \in T$ and $\alpha \in [0,1]$. 
	
	\textsf{STEP 2:} \emph{We show that $\pi$ and $(\pi ^{\alpha})_{\alpha \in [0,1]}$ are uniquely determined.}	Let $\rho \in \MMC_{T}((\mu _t)_{t \in T})$ and $(\rho ^{\alpha})_{\alpha \in [0,1]}$ be a c\`adl\`ag martingale parametrization of $\rho$ w.r.t.\ $(\nu ^{\alpha})_{\alpha \in [0,1]}$ that satisfies
	\begin{equation} 
	\rho^{\alpha} (X_t \in \cdot) = \shadow {\mu _{[0,t]}} {\nu ^{\alpha}}
	\end{equation}
	for all $t \in T$ and $\alpha \in [0,1]$. Lemma \ref{lemma:ParamExtension} (i) yields that the restrictions $({\rho ^{\alpha}}_{|S})_{\alpha \in [0,1]}$ are a martingale parametrization of $\rho _{|S}$. Furthermore, we obtain $\rho^{\alpha} (X_t \in \cdot) = \shadow {\mu _{[0,t]}} {\nu ^{\alpha}} = \shadow {\mu _{S_t}} {\nu ^{\alpha}}$
	for all $t \in S$ and $\alpha \in [0,1]$ by Corollary \ref{cor:CountableObstr}. The uniqueness part of Theorem \ref{prop:GenExistCount} implies that $\rho_{|S}$ and $\pi _{|S}$ coincide and hence $\pi = \rho$ by Lemma \ref{lemma:RCExtensionMartingale} (i). 
	
	\textsf{STEP 3:} \emph{We show that  there exists a Borel set $A \subset [0,1]$ with $\lambda(A) = 1$ such that  $\alpha \mapsto {\pi^{\alpha}}$ is right-differentiable at $a$ for all $a \in A$.}	Theorem \ref{prop:GenExistCount} yields that there exists a Borel set $A \subset [0,1]$ with $\lambda(A) = 1$ such that for all $a \in A$ the curve $\alpha \mapsto {\pi ^{\alpha}}_{|S}$ is right-differentiable at $a$. Hence, by Lemma \ref{lemma:ParamExtension} (iii) it suffices to show that
	the family of marginal distributions of $\frac{1}{h} (\pi ^{a+h} - \pi ^a)$ converges in $\PPC_{T}$ to a right-continuous peacock for all $a \in A$ to show that $\alpha \mapsto {\pi ^{\alpha}}$ is right-differentiable for all $a \in A$.
	
	To this end, fix $a \in A$ and note that the distribution of $X_t$ under $\frac{1}{h} (\pi ^{a+h} - \pi ^a)$ is given by 
	\begin{equation} \label{eq:RDofShadowCurve}
	\zeta _{h,t}^a := \shadow{\frac{1}{h} \left(\mu _s - \shadow{\mu _{T_s}}{\nu ^a} \right)_{s \in T_t}}{\frac{\nu ^{a +h} - \nu ^a}{h}} = \shadow{\frac{1}{h} \left(\mu _s - \shadow{\mu _{S_s}}{\nu ^a} \right)_{s \in S_t}}{\frac{\nu ^{a +h} - \nu ^a}{h}}
	\end{equation}
	due to Corollary \ref{cor:CountableObstr}.
	Since $\alpha \mapsto {\pi ^{\alpha}}_{|S}$ is right-differentiable at $a$, the limit $\hat{\eta}_t ^a := \lim _{h \downarrow 0} \zeta _{h,t} ^a$ under $\TO$ exists for all $t \in S$. Moreover, Lemma \ref{lemma:ShadowInclusion} (i) implies that $\zeta ^a _{h,u}$ is  decreasing in convex-order for $u \downarrow t$ and thus $\hat{\eta}^a_t \leq _c \hat{\eta}^a_u$ for all $t \leq u$ in $S$ by Lemma \ref{lemma:preservingOrder}. In particular, Lemma \ref{lemma:CSupSum} (i) shows that for all $t \in T$ the limit $\lim _{u \downarrow t, u \in S} \zeta ^a _{h,u}$ exists. Since 	$\zeta ^a _{h,u} $ is decreasing in convex-stochastic order for both $u \downarrow t$ and $h \downarrow 0$ (by Lemma  \ref{lemma:ShadowInclusion} (i) and Lemma \ref{lemma:GenShadowOrder} (iii) in conjunction with Lemma  \ref{lemma:GenShadowOrder} (iv)), we may interchange the limits and obtain for $t\in T$ by Lemma \ref{lemma:ShadowRightCont}
	\begin{equation*}
	\lim _{u \downarrow t, u \in S} \hat{\eta} ^a _u  = \lim _{u \downarrow t, u \in S} \lim _{h \downarrow 0} \zeta _{h,u}^a 	= \lim _{h \downarrow 0} \lim _{u \downarrow t} \zeta _{h,u} ^a  = \lim _{h \downarrow 0}  \zeta _{h,t} ^a =: \hat{\eta} _t ^a.
	\end{equation*}
	Thus, the marginal distributions of $(X_t)_{t \in S}$ under $\frac{1}{h} (\pi ^{a+h} - \pi ^a)$ converge in $\PPC_{T}$ to a right-continuous peacock.	
	
	\textsf{STEP 4:} \emph{We show that for all $a \in A$  the marginal distributions of $(X_t)_{t \in T}$ under the right-derivatives of $\alpha \mapsto \pi ^{\alpha}$ at $a$ are a NSI peacock.} By Lemma \ref{lemma:ParamExtension} (iii), the restriction of the right-derivative of $\alpha \mapsto \pi^{\alpha}$ at $a \in A$ to $S$ is the right-derivative of $\alpha \mapsto \pi^{\alpha}_{|S}$ at  $a \in A$. Since the marginal distributions of the latter are NSI by Theorem \ref{prop:GenExistCount}, Corollary \ref{cor:NSICountable} implies that the marginals of the right-derivative of $\alpha \mapsto \pi^{\alpha}$ at $a$ are NSI as well.
	
	\textsf{STEP 5:} \emph{We remove the assumption that $\sup T \in T$.} If $T$ does not admit a maximal element, there exists a sequence $(t_n ^*)_{n \in \mathbb{N}}$ in $T$ approaching $\sup T \in [0, \infty]$. For every $n \in \mathbb{N}$, we have shown that there exists a unique shadow martingale and a unique c\`adl\`ag martingale parametrization w.r.t.\ $(\nu ^{\alpha})_{\alpha \in [0,1]}$ corresponding to the right-continuous peacock $(\mu _t)_{t \in T_{t_n ^*}}$. Since these measures are consistent for different $n$, they define a unique shadow martingale and a unique c\`adl\`ag martingale parametrization w.r.t.\ $(\nu ^{\alpha})_{\alpha \in [0,1]}$ corresponding to the right-continuous peacock $(\mu _t)_{t \in T}$.
\end{proof}

\section{Shadow martingales indexed by a totally ordered set} \label{sec:AbstractIndexSet}

In this section, we prove Theorem \ref{thm:MostGenExist} which includes Theorem \ref{thm:GenExist} as a special case. The key observation is that the martingale property implies that this abstract setting can be embedded into the right-continuous setup of Section \ref{sec:ContTimeRC} in such a way that the structure and main properties of the obstructed shadows can be lifted from the continuous time case to the abstract setup.

\subsection{Embedding into the continous time setup}

Let $(T, \leq)$ be a totally ordered set with minimal element $0 \in T$ and $(\mu _t)_{t \in T}$ a peacock. We define the map $E: T \rightarrow [0, \infty)$ as
\begin{equation*}
E(t) = \int _{\mathbb{R}} \sqrt{1+x^2} \de \mu _t(x) - \int _{\mathbb{R}} \sqrt{1+x^2} \de \mu _0 (x)
\end{equation*}
for all $t \in T$ and set $\tilde{T} = E(T)$. Note that  $E(0)=0$ and hence $0 \in \tilde{T}$.
Since  $(\mu _t)_{t \in T}$ is a peacock, $E$ is increasing, $E(t) \leq E(t')$ implies $\mu _t \leqc \mu _{t'}$, and $E(t) = E(t')$ is equivalent to  $\mu _t = \mu _{t'}$. Moreover,  $E(t_n) \rightarrow E(t)$ yields that $(\mu _{t_n})_{n \in \mathbb{N}}$ converges to $\mu _t$ under $\TO$ (cf. \cite[Section 2.2]{BeHuSt16}).
However,  $E$ is in general neither strictly increasing nor invertible.

\begin{lemma} \label{lemma:InterpolationPCOC}
	We define the family $(\tilde{\mu}_u)_{u \in \tilde{T}}$ by
	\begin{equation*}
	\tilde{\mu} _u = \mu _t \ , \,\text{where } t \in E^{-1}(\{u\}).
	\end{equation*}
	For all $u \in \tilde{T}$, the measure $\tilde{\mu}_u$ is well-defined. Moreover, we have:
	\begin{enumerate}
		\item [(i)] The family $(\tilde{\mu}_u)_{u \in \tilde{T}}$ is a peacock with $\tilde{\mu}_{E(t)} = \mu _t$ for all $t \in T$.
		\item [(ii)]   The map $u \mapsto \tilde{\mu}_u$ is continuous from $\tilde{T}$ to $\PO(\mathbb{R})$ w.r.t.\ $\TO$.
		\item [(iii)] For all $t \in T$ and $\nu \in \MO(\mathbb{R})$ with $\nu \leqp \mu _0 = \tilde{\mu}_0$, it holds $\shadow{\mu _{T_t}}{\nu} = \shadow{\tilde{\mu}_{\tilde{T}_{E(t)}}}{\nu}$. 
	\end{enumerate}
\end{lemma}

\begin{proof}
	Item (i): Since $E$ is monotonously increasing, the claim follows immediately.
	
	Item (ii): For $u \in \tilde{T}$ and a sequence $(u_n)_{n \in \mathbb{N}}$  that converges to $u$ we obtain
	\begin{eqnarray*}
		\lim _{n \rightarrow \infty} \int _{\mathbb{R}} \sqrt{1 + x^2} \de \tilde{\mu}_{u_n} &=& \int _{\mathbb{R}} \sqrt{1+x^2} \de \mu _0 + \lim _{n \rightarrow \infty} u_n \\ &=& \int _{\mathbb{R}} \sqrt{1+x^2} \de \mu _0 + u =  \int _{\mathbb{R}} \sqrt{1 + x^2} \de \tilde{\mu}_{u}.
	\end{eqnarray*}
	Thus, $u \mapsto \tilde{\mu}_u$ is continuous.
	
	Item (iii): Let $(\tilde{\eta}_u)_{u \in \tilde{T}}$ be a family in $\MO(\mathbb{R})$ with $\nu \leqc \tilde{\eta}_u \leqc \tilde{\eta}_{u'} \leqp \mu _{u'}$ for all $u \leq u'$ in $\tilde{T}$. We set $\eta _t := \tilde{\eta}_{E(t)}$ for all $t \in T$. Since $E$ is increasing, it holds $\nu \leqc \eta _t \leqc \eta_{t'}$ for all $t \leq t'$ in $T$ and furthermore we have $\eta _{t'} = \tilde{\eta}_{E(t')} \leqp \tilde{\mu}_{E(t')} = \mu _{t'}$ by (ii).
	
	Conversely, let $(\eta _t)_{t \in T}$ be a family with $\nu \leqc \eta _t \leqc \eta _{t'} \leqp \mu _{t'}$ for all $t \leq t'$ in $T$. We set $\tilde{\eta}_u := \eta _t$ for any $t \in E^{-1}(\{u\})$ for all $u \in \tilde{T}$. Since $E$ is increasing, we get $\nu \leqc \tilde{\eta}_u \leqc \tilde{\eta}_{u'}$ for all $u \leq u'$ in $\tilde{T}$. Moreover, there exists $t' \in T$ with $E(t') = u'$, and similar as in (ii) it holds $\tilde{\eta}_{E(t')} = \eta _{u'}$. Thus, $\tilde{\eta}_{u'} = \tilde{\eta}_{E(t')} = \eta _{t'} \leqp \mu _{t'} = \tilde{\mu}_{E(t')} = \tilde{\mu}_{u'}$.
	Then, Proposition \ref{prop:MaximalElement} implies that $\shadow{\mu _{T_t}}{\nu} = \shadow{\tilde{\mu}_{\tilde{T}_{E(t)}}}{\nu}$ for all $t \in T$.
\end{proof}

\begin{lemma} \label{lemma:ExtensionInterpolationPCOC}
	Let $\mathsf{E}^*: \MO(\mathbb{R}^{\tilde{T}}) \rightarrow \MO(\mathbb{R}^T)$ be the map $\tilde{\pi} \mapsto \mathsf{E}^*(\tilde{\pi})$ where $\mathsf{E}^*(\tilde{\pi})$ is uniquely determined by 
	\begin{equation*}
	\mathsf{E}^*(\tilde{\pi})(Y_{t_1} \in B_1, \ldots, Y_{t_n} \in B_n) := \tilde{\pi}(X_{E(t_1)} \in B_1, \ldots, X_{E(t_n)} \in B_n)
	\end{equation*}
	for all $n \in \mathbb{N}$, $t_1, \ldots, t_n \in T$ and Borel sets $B_1, \ldots, B_n$ where $Y$ and $X$ denote the canonical process on $\mathbb{R}^{\tilde{T}}$ and $\mathbb{R}^T$.
	\begin{enumerate}
		\item [(i)] If $(\tilde{\pi}^{\alpha})_{\alpha \in [0,1]}$ is a martingale parametrization of $\tilde{\pi} \in \MM_{\tilde{T}}$, then $(\mathsf{E}^*(\tilde{\pi}^{\alpha}))_{\alpha \in [0,1]}$ is a martingale parametrization of $\mathsf{E}^*(\tilde{\pi}) \in \MM_T$.
		
		\item [(ii)] The map $\mathsf{E}^*$ is sequentially continuous, i.e.\ if $(\tilde{\pi}_n)_{n \in \mathbb{N}}$ converges to $\tilde{\pi}$ under $\TO$, then  $(\mathsf{E}^*(\tilde{\pi}_n))_{n \in \mathbb{N}}$ converges to $\mathsf{E}^*(\tilde{\pi})$ under $\TO$. 
		\item [(iii)] It holds $\{\pi'  :  \pi' \leqp \pi, \pi \in \MM_T((\mu _t)_{t \in T})\} \subset \mathsf{Im}(\mathsf{E}^*)$.
	\end{enumerate}
\end{lemma}

\begin{proof}
	Item (i) and item (ii) are direct consequences of the previous definitions. For item (iii): Let $\pi \in \MM_T((\mu _t)_{t \in T})$ and $\pi' \leqp \pi$. For any $t, t' \in T$ with $E(t) = E(t')$ it holds $\mu _t = \mu _{t'}$. Thus, since $\pi$ is a martingale measure, we have $\pi(X_t \neq X_{t'}) = 0$ and therefore also $\pi'(X_t \neq X_{t'}) = 0$. Hence, we can uniquely define the measure $\tilde{\pi}' \in \MO(\mathbb{R}^{\tilde{T}})$ by
	\begin{equation*}
	\tilde{\pi}'(X_{E(t_1)} \in B_1, \ldots, X_{E(t_n)} \in B_n) := \pi'(Y_{t_1} \in B_1, \ldots, Y_{t_n} \in B_n)
	\end{equation*}
	for all $n \in \mathbb{N}$, $t_1, \ldots, t_n \in T$ and Borel sets $B_1, \ldots, B_n$ where $X$ and $Y$ denote the canonical processes on $\mathbb{R}^{\tilde{T}}$ and $\mathbb{R}^T$, respectively. Clearly, $\mathsf{E}^*(\tilde{\pi}') = \pi$.
\end{proof}

\subsection{Existence and uniqueness of shadow couplings}

\begin{theorem} \label{thm:MostGenExist}
	Let $T$ be a totally ordered set with $\min T=:0\in T$ and $(\nu ^{\alpha})_{\alpha \in [0,1]}$ a $\leq_{c,s}$-convex parametrization of $\mu _0$. There exists a unique pair $(\pi,(\pi ^{\alpha})_{\alpha \in [0,1]})$ where the martingale measure $\pi \in \MM_{T}((\mu _t)_{t \in T})$ solves the peacock problem w.r.t.\ $(\mu _t)_{t \in T}$, $(\pi ^{\alpha})_{\alpha \in [0,1]}$ is a  martingale parametrization of $\pi$ w.r.t.\ $(\nu ^{\alpha})_{\alpha \in [0,1]}$ and for all $\alpha \in [0,1]$ and $t \in T$,
	\begin{equation} \label{eq:ShadowCurve}
	\pi^{\alpha} (X_t \in \cdot) = \shadow {\mu _{T_t}} {\nu ^{\alpha}}.
	\end{equation}
	
	Moreover, there exists a Borel set $A \subset [0,1]$ with $\lambda(A) = 1$ such that for all $a \in A$ the map $\alpha \mapsto \pi ^{\alpha}$ is right-differentiable at $a$ and the marginals of the right-derivative $\hat{\pi}^a$ at $a$ form a NSI peacock. In particular, $\hat{\pi} ^a$ is a Markov martingale measure uniquely defined by its marginal distributions.
\end{theorem}

Observe that Theorem \ref{thm:MostGenExist} encompasses Theorem \ref{thm:GenExist} since $[0,\infty)$ is a totally ordered set with minimal element. The stochastic formulation in the last paragraph of Theorem \ref{thm:GenExist} follows as in Remark \ref{rem:OtherFormulation}.

\begin{proof}
	We prove the claim by reducing to the right-continuous setting in Theorem \ref{thm:GenExistRC}.
	
	\textsf{STEP 1:} \emph{We show that there exists a $\pi \in \MM_{T}((\mu _t)_{t \in T})$ and a martingale parametrization $(\pi ^{\alpha})_{\alpha \in [0,1]}$ of $\pi$ w.r.t.\ $(\nu ^{\alpha})_{\alpha \in [0,1]}$ such that \eqref{eq:ShadowCurve} is satisfied.} Let $(\tilde{\mu}_u)_{u \in \tilde{T}}$ be the right-continuous peacock associated with $(\mu_t)_{t \in T}$ in Lemma \ref{lemma:InterpolationPCOC}.  By Theorem \ref{thm:GenExistRC}, there exists a unique $\tilde{\pi} \in \MMC_{\tilde{T}}((\tilde{\mu}_u)_{u \in \tilde{T}})$ and a unique c\`adl\`ag martingale parametrization $(\tilde{\pi} ^{\alpha})_{\alpha \in [0,1]}$ of $\tilde{\pi}$ that satisfies $	\tilde{\pi} ^{\alpha}(X_u \in \cdot) = \shadow{\tilde{\mu} _{\tilde{T}_u}}{\nu ^{\alpha}} $
	for all $u \in \tilde{T}$ and $\alpha \in [0,1]$.
	Set $\pi = \mathsf{E}^*(\tilde{\pi})$ and $\pi ^{\alpha} = \mathsf{E}^*(\tilde{\pi}^{\alpha})$ for all $\alpha \in [0,1]$  where the map $\mathsf{E}^*$ is defined as in Lemma \ref{lemma:ExtensionInterpolationPCOC}. 
	Then $\pi \in \MM_{T}((\mu _t)_{t \in T})$ and $(\pi ^{\alpha})_{\alpha \in [0,1]}$ is a martingale parametrization of $\pi$ with
	\begin{equation*}
	\pi ^{\alpha}(X_t \in \cdot) = 	\tilde{\pi} ^{\alpha}(X_{E(t)} \in \cdot) =  \shadow{\tilde{\mu} _{\tilde{T}_{E(t)}}}{\nu ^{\alpha}} =  \shadow{\mu_{T_t}}{\nu ^{\alpha}}
	\end{equation*}
	for all $t \in T$ and $\alpha \in [0,1]$ by Lemma \ref{lemma:InterpolationPCOC}.
	
	\textsf{STEP 2:} \emph{We show that $(\pi ^{\alpha})_{\alpha \in [0,1]}$ (and therefore also $\pi = \pi ^1$) is uniquely determined by \eqref{eq:ShadowCurve}.} Let $\rho \in \mathsf{M}_{T}((\mu _t)_{t \in T})$ and let $(\rho ^{\alpha})_{\alpha \in [0,1]}$ be a martingale parametrization of $\rho$ w.r.t.\ $(\nu ^{\alpha})_{\alpha \in [0,1]}$ with
	\begin{equation} \label{eq:Something}
	\rho ^{\alpha}(X_t \in \cdot) = \shadow{\mu _{T_t}}{\nu ^{\alpha}}
	\end{equation}
	for all $\alpha \in [0,1]$ and $t \geq 0$. Let $\tilde{\rho} ^{\alpha} \in (\mathsf{E}^*)^{-1}(\rho ^{\alpha})$ and $t \in T$ with $E(t) = u$. We have
	\begin{equation*}
	\tilde{\rho} ^{\alpha}(X_u \in \cdot)  = \rho  ^{\alpha}(X_{E(t)} \in \cdot) = \shadow{\mu _{\tilde{T}_{E(t)}}}{\nu ^{\alpha}} =  \shadow{\tilde{\mu}_{\tilde{T}_u}}{\nu ^{\alpha}}
	\end{equation*}
	for all $\alpha \in [0,1]$ and $u \in \tilde{T}$ using \eqref{eq:Something} and Lemma \ref{lemma:InterpolationPCOC}.
	The uniqueness part of Theorem \ref{thm:GenExistRC} yields that $\tilde{\pi} ^{\alpha} = \tilde{\rho}^{\alpha}$ and hence $\pi ^{\alpha} = \rho ^{\alpha}$ for all $\alpha \in [0,1]$.
	
	\textsf{STEP 3:} \emph{We show that  there exists a Borel set $A \subset [0,1]$ with $\lambda(A) = 1$ such that for all $a \in A$
	the map $\alpha \mapsto \pi ^{\alpha}$ is right-differentiable at $a$ and the marginals of the right-derivative $\hat{\pi}^a$ at $a$ are a NSI peacock.} By Theorem \ref{thm:GenExistRC}, there exists a Borel set $A \subset [0,1]$ with $\lambda(A) = 1$ such that for all $a \in A$
	the map $\alpha \mapsto \tilde\pi ^{\alpha}$ is right-differentiable at $a$ with right-derivative $ \hat{\tilde{\pi}}^a$  and the marginals of the right-derivative $\hat{\tilde\pi}^a$ at $a$ are a NSI peacock.
	Note that
	\begin{equation*}
	\frac{\pi^{a + h} - \pi ^a}{h} = \mathsf{E}^* \left( \frac{\tilde{\pi}^{a+h} - \tilde{\pi}^a}{h} \right)
	\end{equation*}
	and thus the sequential continuity  proven in Lemma \ref{lemma:ExtensionInterpolationPCOC} implies that at all $a \in A$ the map  $\alpha \mapsto \pi ^{\alpha}$ is right-differentiable with right-derivative $\mathsf{E}^*\left( \hat{\tilde{\pi}}^a \right)$. Since the marginal distributions of the right-derivative  $\hat{\tilde{\pi}}^a$ are a NSI peacock, the same is true for  $\mathsf{E}^*\left( \hat{\tilde{\pi}}^a \right)$  by Lemma \ref{lemma:InterpolationPCOC} in conjunction with Lemma \ref{lemma:CharacNSI}.
\end{proof}

\subsection{Proof of Corollary \ref{thm:intro1}} \label{sec:ProofThmIntro}

Let $(T, \leq)$ be a totally ordered set with minimal element $0 \in T$ and $(\mu _t)_{t \in T}$ a peacock. 

\begin{lemma} \label{lemma:PseudoQuantileFct}
	Let $(\nu ^{\alpha})_{\alpha \in [0,1]}$ a parametrization of $\mu _0$ and let $\pi$ be the shadow martingale w.r.t.\ $(\mu _t)_{t \in T}$ and  $(\nu ^{\alpha})_{\alpha \in [0,1]}$. Moreover, we denote by $(\pi ^{\alpha})_{\alpha \in [0,1]}$ the corresponding parametrization of $\pi$ and by $\hat{\pi}^a$ the right-derivative of $\alpha \mapsto \pi ^{\alpha}$ at $a \in [0,1]$ (if it exists). If there exists a measurable function $q : \mathbb{R} \rightarrow [0,1]$ such that
\[q_{\#} \mu _0 = \lambda\quad\text{ and }\quad\hat{\pi}^{q(x)}(q(X_0) = q(x)) = 1\text{ for }\mu _0\text{-a.e }x \in \mathbb{R},\]
	then for all $n \in \mathbb{N}$, $t_1 \leq \ldots \leq t_n \in T$, $A \in \mathcal{B}(\mathbb{R}^n)$ and $\sigma$-fields $\mathcal{G} \subset \bigvee_{t\in T} \mathcal F_t$ we have $\pi$-a.s.\
	\begin{equation*}
	\mathbb{E}_{\pi} \left[ \1_A(X_{t_1}, \ldots, X_{t_n}) | X_0, \mathcal{G} \right] = \mathbb{E}_{\hat{\pi}^{q(X_0)}} \left[ \1_A(X_{t_1}, \ldots, X_{t_n})  |X_0, \mathcal{G} \right].
	\end{equation*}
\end{lemma}

\begin{proof}
	Let  $n \in \mathbb{N}$, $t_1 \leq \ldots \leq t_n \in T$, $A \in \mathcal{B}(\mathbb{R}^n)$ and fix a $\sigma$-algebra $\mathcal{G} \subset \sigma \left(\bigcup _{t \in T} \mathcal{F}_t\right)$. Moreover, we set $Y := \1_A(X_{t_1}, \ldots, X_{t_n})$.
	For all $B \in \sigma(\sigma(X_0) \cup \mathcal{G})$ we obtain
	\begin{eqnarray*}
		&& \mathbb{E}_{\pi} \left[ \mathbb{E}_{\pi} \left[ Y \vert X_0, \mathcal{G} \right] \cdot \1_B \right] = \mathbb{E}_{\pi} \left[ Y  \cdot \1_{B} \right]  = \int _{\mathbb{R}} \mathbb{E}_{\hat{\pi}^{q(x)}}\left[  Y \cdot \1_B \right] \de \mu _0 \\ 
		&=&  \int _{\mathbb{R}} \mathbb{E}_{\hat{\pi}^{q(x)}}\left[   \mathbb{E}_{\hat{\pi}^{q(x)}}\left[ Y \vert X_0, \mathcal{G} \right]  \cdot \1_{B} \right] \de \mu _0 =  \int _{\mathbb{R}} \mathbb{E}_{\hat{\pi}^{q(x)}}\left[   \mathbb{E}_{\hat{\pi}^{q(X_0)}}\left[  Y \vert X_0, \mathcal{G} \right]  \cdot \1_{B} \right] \de \mu _0 \\ 
		&=& \mathbb{E}_{\pi} \left[ \mathbb{E}_{\hat{\pi}^{q(X_0)}}\left[  Y \vert X_0, \mathcal{G} \right] \cdot \1_{B} \right]. 
	\end{eqnarray*}
	Hence, we have $\mathbb{E}_{\pi} \left[ Y | X_0, \mathcal{G} \right] = \mathbb{E}_{\hat{\pi}^{q(X_0)}} \left[ Y  |X_0, \mathcal{G} \right]$ $\pi$-a.e.\
\end{proof}

\begin{corollary} \label{cor:ShadowMartingalePQF}
	Let $(\mu _t)_{t \in T}$ be a peacock and let $(I_{\alpha})_{\alpha \in [0,1]}$ be a 
	nested family of intervals that satisfies
	\begin{enumerate}
		\item [(i)] $\mu _0(I_{\alpha}) = \alpha$ for any $\alpha \in [0,1)$,
		\item [(ii)] $\sup I_{\alpha} < + \infty$ and $\partial I_{\alpha} \cap \partial I_{\beta} = \emptyset$ for all $\alpha \neq \beta$ in $[0,1]$ and for which
		\item [(iii)] $\alpha \mapsto  \int _{I_{\alpha}} y \de \mu _0(y)$ is a convex function.
	\end{enumerate}
	there exists unique $\pi \in \MM_T((\mu _t)_{t \geq 0})$ such that for all $\rho \in \MM_T((\mu _t)_{t \geq 0})$ we have
	\begin{equation*}
	\quad\mathsf{Law}_{\pi} (X_t | X_0 \in I_{\alpha}) \leqc \mathsf{Law}_{\rho} (X_t | X_0 \in I_{\alpha})
	\end{equation*}
	for all $\alpha \in [0,1]$ and $t \geq 0$.
	Moreover, $(X_0,X_t)_{t \geq 0}$ is a Markov process under $\pi$.
\end{corollary}

Clearly, this covers the case $T = [0, \infty)$ as stated in Corollary \ref{thm:intro1}.

\begin{proof}
	Lemma \ref{lemma:ExConvParam} shows that $(\nu ^{\alpha})_{\alpha \in [0,1]}$ with $\nu ^{\alpha} = {\mu _0}_{|I_{\alpha}}$ is a $\leqcs$-convex parametrization of $\mu _0$.
	Theorem \ref{thm:MostGenExist} states that there exists a unique shadow martingale $\pi$ w.r.t.\ $(\mu _t)_{t \in T}$ and $(\nu ^{\alpha})_{\alpha \in [0,1]}$. 
	The martingale measure $\pi$ is a solution to the peacock problem w.r.t.\ $(\mu _t)_{t \in T}$ and Remark \ref{rem:uniqueParametrization} yields
	\begin{equation*}
	\alpha \mathrm{Law}_{\pi}(X _t | X_0 \in I_{\alpha}) = \pi ^{\alpha} (X_t \in \cdot) = \shadow{\mu _{T_t}}{\mu _{0|I_{\alpha}}}
	\end{equation*}
	for all $\alpha \in [0,1]$. Thus, $\pi$ is the unique solution to the peacock problem with
	\begin{equation*}
	\mathrm{Law}_{\pi}(X _t | X_0 \in I_{\alpha}) \leq _c \mathrm{Law}_{\rho}(X _t | X_0 \in I_{\alpha})
	\end{equation*}
	for any other $\rho \in \mathsf{M}_{T}((\mu _t)_{t \in T})$.
	
	It remains to show that $(X_0,X_t)_{t \in T}$ is a Markov process under $\pi$. 
	Observe that properties (i)-(iii) of the family $(I_{\alpha})_{\alpha \in [0,1]}$ imply that the pseudo-quantile map $q: \mathbb{R} \rightarrow [0,1]$ defined as
	\begin{equation*}
	x \mapsto q(x) := \sup \{ \alpha \in [0,1]: x \not \in I_{\alpha} \}
	\end{equation*}
	meets the requirements of Lemma \ref{lemma:PseudoQuantileFct}  (W.l.o.g.\ we assume $I_1 = \mathbb{R}$).  Note that the map $q$ is Borel measurable because there exists $x_0 \in \mathbb{R}$ such that $x_0 \in I_{\alpha}$ for all $\alpha > 0$ and 
	$q$ is monotone on $(- \infty, x_0]$ and $[x_0,+ \infty)$. Thus,  for all Borel sets $B \subset \mathbb{R}^2$, we have $\pi$-a.e.
	\begin{align} \label{eq:111}
	\left\{
	\begin{aligned}
	\mathbb{E}_{\pi} \left[\1_{B}(X_0,X_t) \vert X_0, \mathcal{F}_s \right]  &=& \mathbb{E}_{\hat{\pi}^{q(X_0)}} \left[ \1_{B}(X_0,X_t) \vert X_0, \mathcal{F}_s \right]\\
	\mathbb{E}_{\pi} \left[ \1_{B}(X_0,X_t) \vert X_0, X_s \right] &=&  \mathbb{E}_{\hat{\pi}^{q(X_0)}} \left[ \1_{B}(X_0,X_t) \vert X_0, X_s \right] 
	\end{aligned}
	\right. .
	\end{align}
	
	By Theorem \ref{thm:MostGenExist}, there exists $A\subset [0,1]$ with $\lambda(A)=1$ such that for all $a\in A$ the right derivatives $\hat{\pi}^a$ of $\alpha \mapsto \pi ^{\alpha}$ exist and $(X_t)_{t \in T}$ is a Markov process under $\hat{\pi}^a$. Hence, $(X_0,X_t)$ is a Markov process under $\hat{\pi}^a$ and the claim follows with \eqref{eq:111}.
\end{proof}

\section{Examples and counterexamples} \label{sec:Examples}

The purpose of this section is twofold. On the one hand we show in Section \ref{sec:contshadow} that certain continuity properties of the obstructed shadow valid for finite index sets $T=\{0,\ldots,n\}$ do not hold in general. In fact, this is one of the reasons why we had to develop new tools.
On the other hand, we  show  explicit examples of shadow martingales. One particular class of examples allowing for very explicit representations is the class of non-obstructed peacocks introduced in Section \ref{sec:NonObstrPCOC}. In Section \ref{sec:ExplSM}, we give a couple of concrete examples (non-obstructed and obstructed).

\subsection{Continuity of the shadow}\label{sec:contshadow}

In Section \ref{sec:ShadObstbyPCOC} we showed that the map $t \mapsto \shadow{\mu _{T_t}}{\nu}$ is left- resp.\ right-continuous whenever the peacock $\mu _T$ is left- or right-continuous. However, we didn't discuss any continuity properties of the shadow as a function of $\nu$ and $\mu _T$ defined on $\MO(\mathbb{R})$ and $\PP_T$ respectively. For the simple shadow (see Proposition \ref{prop:SimpleShadow}), Juillet showed in  \cite[Theorem 2.30]{Ju14} that
\begin{equation} \label{eq:StabilitySimpShad}
\mathcal{W}_1 \left( \shadow{\mu}{\nu}, \shadow{\mu'}{\nu'} \right) \leq \mathcal{W}_1(\nu, \nu') + 2 \mathcal{W}_1(\mu, \mu')
\end{equation}
for all $\nu, \nu ' \in \MO(\mathbb{R})$, $\nu \leqcp \mu$ and $\nu' \leqcp \mu'$ where the Wasserstein distance extends to finite non-probability measures through the dual Kantorovich equivalent definition of $\mathcal{W}_1$ (cf. \cite[Section 1]{Ju14}).

Similarly, an inductive application of \eqref{eq:StabilitySimpShad} in Lemma \ref{lemma:FinteShadowMinimal} implies that for finite $T$ the map  $\mu _T \mapsto \shadow{\mu _T}{\nu}$ is continuous w.r.t.\ pointwise convergence of $\mu _T$ (If $\mu _T$ is a peacock, then this is exactly the topology on $\PP_T$). 
We want to highlight a few implications of this continuity for the class of NSI peacocks in the case of a finite index set $T$  (where $T_t = \{s \leq t : s \in T \}$ as previously):
\begin{proposition}\label{pro:finiteT}
	Let $T$ be finite with minimal element $0$.
	\begin{itemize}
		\item [(i)] The set of NSI peacocks is a closed subset of $\PO(\mathbb{R})^T$. 
		\item [(ii)] Given an arbitrary peacock $(\eta _t)_{t \in T}$, the peacock $(\mu _t)_{t \in T}$ with $$\mu _t:= \lim _{n \rightarrow \infty}(\shadow{(n \eta _s)_{s \in T_t}}{\eta _0})_{n \in \mathbb{N}}$$ for all $t \in T$ is NSI.\label{ii_and_iii}
		\item [(iii)] Let $(\mu _t)_{t \in T}$ be a peacock, $(\nu ^{\alpha})_{\alpha \in [0,1]}$ some parametrization of $\mu_0$  and $(\pi ^{\alpha})_{\alpha \in [0,1]}$  the martingale parametrization of a shadow martingale $\pi$ w.r.t.\ $(\mu _t)_{t \in T}$ and $(\nu ^{\alpha})_{\alpha \in [0,1]}$. If the right-derivative $\hat{\pi}^a$ exists, the family of marginal distributions is NSI.
	\end{itemize}
\end{proposition}

\begin{proof}
		\begin{itemize}

	\item [(i)] Let $(\mu_T^n)_{n \in \mathbb{N}}$ be a sequence of NSI peacocks that converge in $\mathcal{P}_1(\mathbb{R})^T$ to a family of probability measures $\mu_T$. Lemma \ref{lemma:preservingOrder}(ii) yields that $\mu_T$ is again a peacock. Moreover,  for all $t \in T$  we know that  $(\mu_t^n)_{n \in \mathbb{N}}$ converges to $\mu_t$ and $(2\mu_t^n)_{n \in \mathbb{N}}$ converges to $2\mu_t$ under $\TO$. Hence, we obtain by using the stability inequality \eqref{eq:StabilitySimpShad} at most $|T|$ times that
	\begin{equation*}
		\shadow{(2\mu_s)_{s \in T_t}}{\mu_0} = \lim_{n \rightarrow \mathbb{N}} \shadow{(2\mu^n_s)_{s \in T_t}}{\mu^n_0} = \lim_{n \rightarrow \mathbb{N}}  \mu^n_t = \mu_t
	\end{equation*}
	for all $t \in T$ where the limits are w.r.t.\ $\TO$. The second equality is a consequence of  Lemma \ref{lemma:CharacNSI} and we also deduce from Lemma \ref{lemma:CharacNSI} that $\mu_T$ is NSI.
	
	\item [(ii)] For all $t \in T$ the sequence $(\shadow{(n \eta_s)_{s \in T_t}}{\eta_0})_{n \in \mathbb{N}}$ is monotonously decreasing w.r.t.\ $\leq_c$ and bounded from below by $\eta_0$. Hence, there exists a limit $\mu_t$ under $\TO$. We deduce from Lemma \ref{lemma:preservingOrder}(ii) that $\mu_T$ is a peacock. Directly from the definition of the shadow we see that $\mu_0 = \eta_0$.  Moreover, for all $t \in T$ we obtain by using the stability inequality \eqref{eq:StabilitySimpShad} at most   $|T|$ times that
	\begin{align*}
		\shadow{(2 \mu_s)_{s \in T_t}}{\mu_0} &= \lim _{n \rightarrow \infty} \shadow{(2 \shadow{(n \eta_r)_{r \in T_s}}{\eta_0})_{s \in T_t}}{\mu_0} \\
		&= \lim _{n \rightarrow \infty} \shadow{(2 \shadow{(n \eta_r)_{r \in T_s}}{\eta_0})_{s \in T_t}}{\eta_0} \\
		&= \lim _{n \rightarrow \infty} \shadow{(2 n \eta_s)_{s \in T_t}}{\eta_0} = \mu_t
	\end{align*}
	where the limits are w.r.t.\ $\TO$. Here the third equality is true for all $n \in \mathbb{N}$. Indeed, first observe that for every $s$
	\begin{align*}
		 2 \shadow{(n \eta_r)_{r \in T_s}}{\eta_0} \leqp 2n\eta_s .
	\end{align*}
This implies that we have
	\begin{align}\label{eq:eisessen}
	 \shadow{(2 \shadow{(n \eta_r)_{r \in T_s}}{\eta_0})_{s \in T_t}}{\eta_0} &\geqc \shadow{(2 n \eta_s)_{s \in T_t}}{\eta_0}.
	\end{align}
Moreover, we have  
	\begin{align*}
	\eta_0\leqc \shadow{(2 n \eta_s)_{s \in T_t}}{\eta_0} \leqp \shadow{(2 n \eta_s)_{s \in T_t}}{2\eta_0} = 2 \shadow{(n \eta_s)_{s \in T_t}}{\eta_0}
	\end{align*}
so that $\shadow{(2 n \eta_s)_{s \in T_t}}{\eta_0}$ is a candidate for $\shadow{(2 \shadow{(n \eta_r)_{r \in T_s}}{\eta_0})_{s \in T_t}}{\eta_0}$ in the sense of Proposition \ref{prop:MaximalElement} which therefore yields the opposite inequality to \eqref{eq:eisessen} and hence third equality above.
	 Lemma \ref{lemma:CharacNSI} yields that $\mu_T$ is NSI. 
	
	\item [(iii)] We first assume that $(\nu ^{\alpha})_{\alpha \in [0,1]}$ is the sunset parametrization of $\mu _0$ and the right-derivative of $\alpha \mapsto \pi^\alpha$ exists at $a \in [0,1)$. For all $t \in T$, we have using Lemma \ref{lemma:ShadowFactor} and Proposition \ref{prop:ShadowAssoc}
	\begin{align*}
		\hat{\pi}^a(X_t \sowh) &= \lim _{h \downarrow 0} \frac{1}{h} \left(\pi^{a + h}(X_t \sowh) - \pi^{a }(X_t \sowh)  \right)\\
		&= \lim_{h \downarrow 0} \frac1h \left[\shadow{(\mu_s)_{s\in T_t}}{(a+h)\mu_0}-\shadow{(\mu_s)_{s\in T_t}}{a\mu_0}\right]\\
		 &= \lim _{h \downarrow 0}\shadow{(h^{-1}[\mu_s-\shadow{\mu_{T_s}}{a\mu_0}])_{s \in T}}{ \mu_0}.		
	\end{align*}
	Hence, since $(\mu_t-\shadow{(\mu_s)_{s\in T_t}}{a\mu_0})_{t\in T}$ is a peacock (see Remark \ref{rem:RestIsPCOC}), (ii) implies that the marginal distributions under $\hat{\pi}^a$ are a NSI peacock.
	For any other parametrization we obtain with the same argument
	\[\hat{\pi}^a(X_t \sowh)=\lim _{h \downarrow 0}\mathcal{S}^{(h^{-1}[\mu_s-\shadow{\mu_{T_s}}{\nu^a}])_{s \in T}}\left(\frac{\nu^{a+h}-\nu^a}{h}\right).\]
By inequatilty \eqref{eq:StabilitySimpShad} we can replace $\frac{\nu^{a+h}-\nu^a}{h}$ by its limit $\hat \nu^a$ and conclude by (ii).
\end{itemize}\qedhere

\end{proof}

Armed with Proposition \ref{pro:finiteT} and independently from Theorem \ref{thm:MostGenExist} we can give a short proof of the uniqueness of shadow martingales for a finite set $T$. Moreover, appealing to the connection of one step NSI peacocks and Kellerer dilations explained in Remark \ref{rem:NSIandKellerer} it follows that shadow martingales can be disintegrated into binomial martingales whose law is by definition a concatenation of Kellerer dilations (cf.\ \cite[Proposition 8.5]{NuStTa17}). In the special case of the left-curtain parametrization, we recover \cite[Theorem 8.3]{NuStTa17} by Nutz, Stebegg, and Tan.

\begin{corollary}\label{cor:thm83}
	Let $T=\{0,1,\ldots,n\}$ be finite, $(\mu_t)_{t \in T}$ be a peacock, and $(\nu^\alpha)_{\alpha\in[0,1]}$ be any parametrization of $\mu_0$. 
	There exists a unique shadow martingale $\pi$ w.r.t.\ $(\mu_t)_{t \in T}$
	and a unique martingale parametrization $(\pi ^{\alpha})_{\alpha \in [0,1]}$ of $\pi$ w.r.t.\ $(\nu ^{\alpha})_{\alpha \in [0,1]}$ such that $\alpha \mapsto \pi^{\alpha}$ is right-differentiable at $\lambda$-a.e.\ $a\in [0,1]$ and the right-derivative $\hat{\pi}^a$, whenever it exists, is a binomial martingale measure. If moreover $\mu_0(\{x\})=0$ for all $x \in \mathbb{R}$, the left-curtain shadow martingale is binomial.
\end{corollary}

\begin{proof}
	By Proposition \ref{prop:CountableExistSM} there exists a shadow martingale $\pi$ w.r.t.\ $(\mu _t)_{t \in T}$ and $(\nu ^{\alpha})_{\alpha \in [0,1]}$. Denote by $(\pi ^{\alpha})_{\alpha \in [0,1]}$ the corresponding martingale parametrization (recall that Proposition \ref{prop:CountableExistSM} does not require that $(\nu ^{\alpha})_{\alpha \in [0,1]}$ is $\leqcs$-convex). 
	
	By Lemma \ref{lemma:RightDerivatives} there exists a Borel set $A \subset [0,1]$ with $\lambda(A) = 1$ such that $\alpha \mapsto \pi ^{\alpha}$ is right-differentiable for all $a \in A$. For each $a \in A$, we denote by $(\eta ^a)_{t \in T}$ the marginal distributions of the right-derivative $\hat{\pi}^a$. Proposition \ref{pro:finiteT} (iii) yields that $(\eta ^a _t)_{t \in T}$ is NSI for all $a \in A$ and the associated martingale is unique.
	
	For each $a \in A$, by Definition \ref{def:pcoc} we have that the one-step peacocks $(\eta ^a _{i}, \eta ^a _{i+1})$ are NSI for $0 \leq i \leq n-1$. Thus, Remark \ref{rem:NSIandKellerer} yields that $\hat{\pi}^a$ is a binomial martingale. In particular, $\hat{\pi}^a$ is uniquely determined by its marginals and therefore $\pi$ and $(\pi ^{\alpha})_{\alpha \in [0,1]}$ are uniquely determined by their marginals.	
	If $\mu_0(\{x\})=0$ for all $x \in \mathbb{R}$ and $(\nu ^{\alpha})_{\alpha \in [0,1]}$ is the left-curtain parametrization, we obtain by Lemma \ref{lemma:PseudoQuantileFct} that $\mathrm{Law}_{\pi}(X | X_0) = \hat{\pi}^q(X_0)$, $\pi$-a.s.\ where $q = F_{\mu_0} ^{-1}$ is the quantile function of $\mu_0$. The claim follows from the first part.
\end{proof}

Let us go back to the  case of an infinite index set $T$. As Example \ref{expl:ShadowNotCont} below shows, the continuity of $\mu _{S} \mapsto \shadow{\mu _S}{\nu}$ fails in general. Moreover, neither of the items (i)-(iii) of Proposition \ref{pro:finiteT} are  true any more as Examples \ref{expl:ShadowNotCont} and \ref{expl:RDNotNSI} show.

We would like to stress that the lack of item (ii) (and hence (iii)) of Proposition \ref{pro:finiteT} is the main point separating the case of finite index sets from the one of a countably infinite index set $T$. As a particular consequence, we could not rely on argumentations from \cite{BeJu21} and \cite{NuStTa17} to show uniqueness of shadow martingales but had to develop a new approach.

In Examples \ref{expl:ShadowNotCont} and \ref{expl:RDNotNSI} we consider $T=[0,1]$, or the countable index set $T=[0,1]\cap\mathbb{Q}$.

\begin{example}[Discontinuity of the shadow, NSI is not closed] \label{expl:ShadowNotCont}Define $(\mu^n _t)_{t \in [0,1]}$ by
	\begin{equation*}
	\mu _t ^n = \begin{cases}
	\delta_0   &t \in \left[0, \frac{n-1}{n} \right) \\
	\frac{ \delta _{-1} + \delta _1}{2} &t \in \left[\frac{n-1}{n},1 \right) \\
	\frac{1}{2} \delta _0 + \frac{1}{4} \left( \delta _{-2} + \delta _{2} \right) &t = 1
	\end{cases}
	\end{equation*}
	for $t \in [0,1]$ and $n \in \mathbb{N}$. For all $t \in [0,1]$ the sequence $(\mu _t ^n)_{n \in \mathbb{N}}$  converges in $\MO(\mathbb{R})$ to
	\begin{equation*}
	\mu _t = \begin{cases}
	\delta _0 &t < 1 \\
	\frac{1}{2} \delta _0 + \frac{1}{4} \left( \delta _{-2} + \delta _{2} \right) &t = 1
	\end{cases}.
	\end{equation*}
	But for $\nu = \frac{1}{2}\delta _0$, we have $\lim _{n \rightarrow \infty} \shadow{\mu^n_{[0,1]}}{\nu} = \frac{1}{4} \delta _0 + \frac{1}{8} \left( \delta _{-2} + \delta _{2} \right) \neq \frac{1}{2} \delta _0 = \shadow{\mu _{[0,1]}}{\nu}.$
	
	Thus, the map  $\tilde\mu _{[0,1]} \mapsto \shadow{\tilde\mu _{[0,1]}}{\nu}$  is not (sequentially) continuous.
	Moreover, any element of the sequence $(\mu^n _t)_{t \in [0,1]}$ is NSI but the limit $(\mu_t)_{t\in[0,1]}$ is not NSI by Lemma \ref{lemma:CharacNSI}. Hence the subset of NSI peacocks in $\PP_{[0,1]}$ is not closed. 
\end{example}

\begin{example}[Right-derivatives are not NSI] \label{expl:RDNotNSI}
	Let $(\mu _t)_{t \in [0,1]}$ be defined by
	\begin{equation*}
	\mu _t  = \begin{cases}
	(1-\frac{t}2) \delta_0 + \frac{t}{4} \left( \delta _{-1} + \delta _1\right) &t < 1 \\
	\frac{3}{4} \delta _0 + \frac{1}{8} \left( \delta _{-2} + \delta _{2} \right) &t = 1
	\end{cases}
	\end{equation*}
	for $t \in [0,1]$ and let $\nu ^{\alpha} = \alpha \mu _0 = \alpha \delta _0 $ 
	be the sunset parametrization of $\mu _0$. Moreover, let $\pi$ be the shadow martingale w.r.t.\ $(\mu _t)_{t \in [0,1]}$ and $(\nu ^{\alpha})_{\alpha \in [0,1]}$ and $(\pi ^{\alpha})_{\alpha \in [0,1]}$ the corresponding martingale parametrization. 
	For all $h > 0$ (small enough) and  $a \in [0,1)$, it holds $ 	\frac{\nu ^{a + h} - \nu ^a }{h} = \mu _0 = \delta _0.$
	Thus, the right-derivative of $\alpha \mapsto \nu ^{\alpha}$ exists for all $a \in [0,1)$ and equals $\mu _0$. Actually, it is not difficult to show that the right-derivatives $\hat{\pi}^a$ of $\alpha \mapsto \pi ^{\alpha}$ exist for all $a \in [0,1)$.
	However,  the marginal distributions of $\hat{\pi}^{\frac{1}{2}}$ are
	\begin{equation}\label{eq:expl9.2}
	\hat{\pi}^{\frac{1}{2}} (X_t \sowh) = \lim _{n \rightarrow \infty} \shadow{(n (\mu _s - \shadow{\mu_{T_s}}{\frac12\mu _0}))_{s \in T_t}}{\mu _0} = \begin{cases}
	\delta _0 & t < 1 \\
	\frac{1}{2} \delta _0 + \frac{1}{4} \left( \delta _{-2} + \delta _{2} \right) &t = 1.
	\end{cases}
	\end{equation}
	Observe that $(\mu _s - \shadow{\mu_{T_s}}{\frac12\mu _0}))_{s \geq0}$  is a peacock by Remark \ref{rem:RestIsPCOC}. Moreover, Lemma \ref{lemma:CharacNSI} implies that the family of measures on the right-hand side of \eqref{eq:expl9.2} is not an NSI peacock. Hence, items (ii) and (iii) of Proposition \ref{pro:finiteT} are not satisfied.

NSI implies uniqueness of the martingale associated to the marginals. Note that even without NSI in the present example the martingale for parameter $a=1/2$ is also unique. However, by splitting the Dirac mass $\delta_0$ in the definition of $\mu_t$ and replacing it for instance by $\frac12 (\delta_{-1/10}+\delta_{1/10})$ we  see that the resulting peacock can be represented by several different martingales. This phenomena does not occur for $T$ finite, recall (iii) in Proposition \ref{pro:finiteT}.
\end{example}

\subsection{Non-obstructed shadows} \label{sec:NonObstrPCOC}

We fix a totally ordered set $(T,\leq)$ with minimal element $0\in T$.
In this section, we consider the special case that the additional obstructions  in the shadow between $\mu _0$ and $\mu _t$ imposed by the marginals $(\mu _s)_{s \in T_t}$ are not binding, i.e. $\shadow{\mu _{T_t}}{\nu } = \shadow{\mu _t}{\nu }$ for all $t \in T$.  The associated shadow martingales allow for rather explicit representations as will be shown in Proposition \ref{prop:NonObstructedPCOC} and illustrated in Examples \ref{expl:MCSM}--\ref{expl:SunSM}.

\begin{definition}
	Let $(\mu _t)_{t \in T}$ be a family in $\PO(\mathbb{R})$ and $\nu \leqcp \mu _{T}$. 
	We say that the shadow of $\nu$ in $(\mu _t)_{t \in T}$ is non-obstructed if for all $t \in T$ we have $ 	\shadow{\mu _{T_t}}{\nu} = \shadow{\mu _t}{\nu}$.
\end{definition}

\begin{lemma} \label{lemma:CharNonObs}
	Let $(\mu _t)_{t \in T}$ be a family in $\PO(\mathbb{R})$ and $\nu \leqcp \mu _{T}$. The shadow of $\nu$ in $(\mu _t)_{t \in T}$ is non-obstructed if for all $s \leq t$ in $T$ one of the following equivalent conditions is satisfied:
	\[\mathrm{(i})\;\shadow{\mu _s}{\nu} \leqc \shadow{\mu _t}{\nu}\quad\text{or}\quad\mathrm{(ii})\; \shadow{\mu _t}{\nu}=\shadow{\mu_t}{\shadow{\mu_s}{\nu}} \ .\]
\end{lemma}

\begin{proof}
	For fixed $s \leq t$ in $T$, the equivalence of (i) and (ii) is straightforward. If (i) is satisfied for all $s \leq t$, applying Proposition \ref{prop:MaximalElement} to $(\shadow{\mu _s}{\nu})_{s \in T_t}$ finishes the proof.
\end{proof}

We are interested in pairs of peacocks $(\mu _t)_{t\in T}$ and $\leqcs$-convex parametrization $(\nu ^{\alpha})_{\alpha \in [0,1]}$ of $\mu _0$ for which the shadow of $\nu^\alpha$ in $(\mu _t)_{t \in T}$ is non-obstructed, for all $\alpha \in [0,1]$. For an example available in the literature, one can consider  the middle-curtain parametrization in combination with peacocks increasing in \textit{diatomic convex order}, see \cite{Ju16} and \cite[Section 3.1.3]{BeJu21}.  Setting $x$ the barycenter of $\mu _0$, this condition can be formulated by $\shadow{\mu_s}{\alpha \delta _x} \leqc \shadow{\mu_t}{\alpha \delta _x}$ for all $0 \leq s \leq t$ and $\alpha \in [0,1]$. One can easily check that it precisely corresponds to peacocks for which the middle-curtain parametrization has non-obstructed shadows. Note that in \cite{Ju16} a non-Markov generalization of Kellerer's Theorem \ref{prop:KellerersThm} is given for peacocks in $\mathcal P(\R)$ increasing in diatomic convex order that are indexed by a \textit{partially} ordered set.

The following lemma describes a class of peacocks non-obstructed by parametrizations of ``interval type'':

\begin{lemma}
	\label{lemma:dispersivePCOC}
	Let $(\mu _t)_{t \in T}$ be a peacock such that there exists a nested family of closed intervals $(I_t)_{t \in T}$ with
	\begin{enumerate}
		\item [(i)]  $\mathrm{supp}(\mu_t) \subset I_t$ for all $t \in T$ and
		\item [(ii)] ${\mu _t}_{|I_s} \leq _+ {\mu _s}_{|I_s}$ for all $s \leq t$ in $T$.	 
	\end{enumerate}
	If additionally $\mu _t(\{x\}) = 0$ for all $x \in \mathbb{R}$ and $t \in T$, then, for any closed interval $I \subset \mathbb{R}$, the shadow of $\nu = {\mu _0}_{|I}$ in $(\mu _t)_{t \in T}$ is non-obstructed.
\end{lemma}

\begin{proof}
	Let $s \leq t$ in $T$ and $I \subset \mathbb{R}$ be an interval. W.l.o.g.\ we may assume $I \subset I_s$. Property (ii) yields that $({\mu _s}_{|I} - {\mu _t}_{|I}) \in \MO(\mathbb{R})$ and thus by Proposition \ref{prop:propertiesShadow} (ii)
	\begin{equation} \label{eq:intervalStructure}
	\shadow{\mu _t}{\mu _{s |I}} = \mu _{t |I} +  \shadow{\mu _{t |I^c}}{\mu _{s |I} - {\mu _{t}}_{|I}} = \mu _{t |I} + \mu _{t |I^c \cap J} = \mu _{t |J}
	\end{equation} 
	for some closed interval $J$ with $I \subset J \subset I_t$. Such an interval $J$ exists because $\mathrm{supp}({\mu _s} _{|I} - {\mu _{t}}_{|I})$ is contained in the interval $I$ and $\mathrm{supp}(\mu _{t |I^c})$ belongs to the closure of the complement of $I$ (see Lemma \ref{expl:ShadowOfAtom} (ii)). Applying \eqref{eq:intervalStructure} twice yields both $\shadow{\mu _t}{\nu} = \mu _{t |J}$ and $\shadow{\mu _t}{\shadow{\mu _s}{\nu}} = \mu _{t |J'}$ for two intervals $J$ and $J'$. However, since 
	both measures are in convex order greater than $\nu$, they have the same mass and barycenter by Lemma \ref{lemma:relationOrders} (i) and hence $\mu _{t |J}  =\mu _{t | J'}$. Thus, we have proven that $\shadow{\mu _s}{\nu} \leqc  \shadow{\mu _t}{\nu}$ because $\shadow{\mu _s}{\nu} \leqc \shadow{\mu _t}{\shadow{\mu _s}{\nu}}$ by default.
	The claim follows by Lemma  \ref{lemma:CharNonObs}.
\end{proof}

Note that the condition in Lemma \ref{lemma:dispersivePCOC} is conceptually very similar to the \textit{Dispersion Assumption} introduced by Hobson and Norgilas in \cite{HoNo17}. Recall that $T$ is a totally ordered set with minimal element $0$.

\begin{proposition}
	\label{prop:NonObstructedPCOC}
	Let  $(\mu _t)_{t \in T}$ be a peacock and $(\nu ^{\alpha})_{\alpha \in [0,1]}$ a $\leqcs$-convex parametrization of $\mu _0$. Let $\pi \in \MM_{T}((\mu _t)_{t \in T})$ be the corresponding shadow martingale and $(\hat{\pi}^a)_{a \in [0,1]}$ the family of right-derivatives. 
	For all $x \in \mathbb{R}$ and $a \in [0,1]$, we define 
	the maps $C_+^{x,a}, C_-^{x,a}: [0,1] \rightarrow \mathbb{R}$ as
	\begin{eqnarray*}
		C_+^{x,a} (t) &=& \inf \left\{ [x, + \infty ) \cap \mathrm{supp} ( \mu _t- \shadow{\mu _t}{\nu ^a} ) \right\} \hspace{0,2cm} \text{and} \\
		C_-^{x,a} (t) &=& \sup \left\{ (- \infty, x] \cap \mathrm{supp} ( \mu _t- \shadow{\mu _t}{\nu ^a} ) \right\}.
	\end{eqnarray*}
	If  the shadow of $\nu^{\alpha}$ in $(\mu _t)_{t \in T}$ is non-obstructed for all $\alpha \in [0,1]$, then under $\hat{\pi}^a$ the process $(X_t)_{t \in T}$ is a Markov process with 
	\begin{equation} \label{eq:NOPos}
	X_t \in \{C_+^{X_0,a}(t),C_-^{X_0,a} (t)\} \hspace*{1cm} \hat{\pi}^a\text{-a.e.}
	\end{equation}
	Moreover, if there exists a measurable function $q : \mathbb{R} \rightarrow [0,1]$ as in Lemma \ref{lemma:PseudoQuantileFct}, then $(X_0,X_t)_{t \in T}$ is a Markov process under $\pi$ that jumps between the two curves
	\begin{eqnarray*}
		\tilde{C}_+^{x} (t) &=& \inf \left\{ [x, + \infty) \cap \mathrm{supp}(\mu _t- \shadow{\mu _t}{\nu ^{q(x)}}) \right\} \hspace{0,2cm} \text{and} \\
		\tilde{C}_-^{x} (t) &=& \sup \left\{ (- \infty, x] \cap \mathrm{supp}(\mu _t- \shadow{\mu _t}{\nu ^{q(x)}}) \right\}
	\end{eqnarray*}
	depending on the initial value $X_0 = x$.
\end{proposition}

\begin{proof}
	Theorem \ref{thm:GenExistRC} yields that there exists a Borel set $A \subset [0,1]$ with $\lambda(A) = 1$ such that for all $a \in A$ the right-derivative $\hat{\pi}^a$ of $\alpha \mapsto \pi ^{\alpha}$ at $a$ exists and is  Markov. 
	
	Since the shadow is non-obstructed, the marginal distributions of $\hat{\pi}$ satisfy
	\begin{equation*}
	\hat{\pi}^a(X_t \in \cdot) = \lim _{h \downarrow 0} \shadow{\frac{1}{h} \left(\mu _t - \shadow{\mu _t}{\nu ^a} \right)}{\frac{\nu ^{a +h} - \nu ^a}{h}}.
	\end{equation*}
	Note that this is the simple shadow. Since $\alpha \mapsto \nu ^{\alpha}$ is right-differentiable everywhere with derivative $\hat{\nu} ^a$, we can apply \cite[Lemma 2.8]{BeJu21} (in conjunction with Lemma \ref{lemma:convDom}) for simple shadows, to obtain $\hat{\pi}^a(X_t \in \cdot) = \hat{\nu}^a P_{t,a}$ where $P_{t,a}$ denotes the Kellerer dilation onto the set $\mathrm{supp}(\mu _t - \shadow{\mu _t}{\nu ^a})$ (see Remark \ref{rem:NSIandKellerer}). There is only one martingale coupling between a measure and its Kellerer projection onto a set $F$ and it is given by the Kellerer dilation kernel.
	Hence, \eqref{eq:NOPos} holds a.e. 
	
	In the second case, we know that $(X_0,X_t)_{t \in T}$ is a Markov process under the unique shadow martingale measure $\pi$ (cf. Corollary \ref{cor:ShadowMartingalePQF}) and moreover by Lemma \ref{lemma:PseudoQuantileFct}
	\begin{align*}
	\pi \left( X_t \in \left\{ \tilde{C}^{X_0} _+, \tilde{C}^{X_0} _- \right\}  \right) &= \mathbb{E}_{\pi} \left[ \pi \left( X_t \in \left\{ \tilde{C}^{X_0} _+, \tilde{C}^{X_0} _- \right\} \vert \ X_0  \right) \right] \\
	&= \mathbb{E}_{\pi} \left[ \hat{\pi} ^{q(X_0)} \left( X_t \in \left\{ \tilde{C}^{X_0} _+, \tilde{C}^{X_0} _- \right\} \vert \ X_0  \right) \right] \\
	&= \mathbb{E}_{\pi} \left[ \hat{\pi} ^{q(X_0)} \left( X_t \in \left\{ C^{X_0,q(X_0)} _+, C^{X_0,q(X_0)} _- \right\} \vert \ X_0  \right) \right] = 1.\qedhere
	\end{align*}
\end{proof}

\subsection{Examples of shadow martingales} \label{sec:ExplSM}

In this section we present four examples of shadow martingales. 
For the first three examples we can apply Proposition \ref{prop:NonObstructedPCOC} from the previous subsection.
Indeed,  using Lemma \ref{lemma:CharNonObs} it is easy to check that in these cases the shadow w.r.t.\ the given peacock and parametrization is non-obstructed. 
Alternatively, in Example \ref{expl:MCSM} and Example \ref{expl:LCSM} one could argue via Lemma \ref{lemma:dispersivePCOC} to see that the shadow is non-obstructed. Moreover, in these two examples the second part of Proposition \ref{prop:NonObstructedPCOC} is applicable (cf.\ the proof of Corollary \ref{cor:ShadowMartingalePQF}).

\begin{example} \label{expl:MCSM}
	Let $(\mu _t)_{t \geq 0}$ be the marginal distributions of a standard Brownian motion started at $t = 1$, i.e.\ $\mu _t$ is the Gaussian distribution on $\mathbb{R}$ with mean $0$ and variance $1+t$, and let $(\nu ^{\alpha} _{\text{mc}})_{\alpha \in [0,1]}$ be the middle-curtain parametrization of $\mu _0$.
	
	Let $\pi _{\mc}$ be the shadow martingale w.r.t.\ $(\mu _t)_{t \geq 0}$ and $(\nu ^{\alpha} _{\mc})_{\alpha \in [0,1]}$. The end of Proposition \ref{prop:NonObstructedPCOC} shows that the canonical process $(X_t)_{t \geq 0}$ is a martingale jumping between the two curves
	\begin{equation} \label{eq:CurvesMCSM}
	C_-^{X_0}: \ t \ \mapsto \ -|X_0| \cdot \sqrt{1 + t}  \hspace{0.5cm} \text{and} \hspace{0.5cm}
	C_+^{X_0}: \ t \ \mapsto \ |X_0| \cdot \sqrt{1 + t}.
	\end{equation}
	Since $\mathrm{Law}_{\pi _{\mc}} (X_0) = \mu _0$ is fixed, \eqref{eq:CurvesMCSM} characterizes $\pi  _{\mc}$ uniquely. 
	
	We can describe $\pi _{\mc}$ in purely stochastic terms: If $N$ is a standard Gaussian distributed random variable, the distribution of the unique c\`adl\`ag martingale $(Y_t)_{t \geq 0}$ that satisfies
	\begin{equation*}
	Y_t \in \begin{cases}
	\{N\} & t = 0 \\
	\{- |N| \cdot \sqrt{1 + t}, |N| \cdot \sqrt{1 + t}\} & t > 0
	\end{cases}
	\end{equation*}
	for all $t \geq 0$ is precisely the shadow martingale measure $\pi _{\text{\mc}}$.
\end{example}

Note that in the previous example $(X_t)_{t \geq 0}$ is a Markov process under $\pi _{\lc}$. 
In general, only the process $(X_0,X_t)_{t \geq 0}$ is a Markov process under the shadow martingale measure by Corollary \ref{thm:intro1} and not the canonical process $(X_t)_{t \geq 0}$ itself.   
But since in Example \ref{expl:MCSM} we have $\{C_-^x(t), C_+^x(t) : t \geq 0 \} \cap \{C_-^y(t), C_+^y(t) : t \geq 0 \} = \emptyset$ for all $x \neq y$ (cf.\ Figure \ref{fig:MCSM}),
one can reconstruct $|X_0|$ from $X_t$ and thus $(X_t)_{t \in [0,1]}$ is Markov.

\begin{center}
	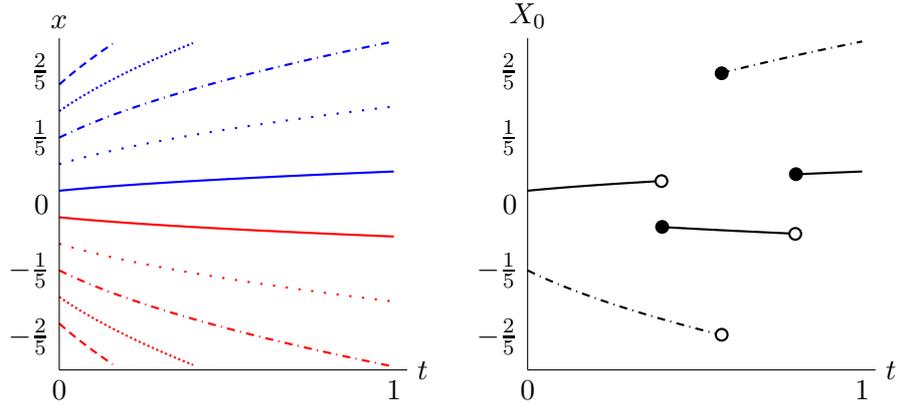
\begin{figure} 
		\begin{tikzpicture}[scale=0.88]
		\draw[densely dashed, red,thick,domain = 0:0.8,smooth,variable=\t] 
		plot({\t},{-1.8*sqrt(1+\t)});
		\draw[densely dotted, red,thick,domain = 0:2,smooth,variable=\t] 
		plot({\t},{-1.4*sqrt(1+\t)});
		\draw[dashdotted,red,thick,domain = 0:5,smooth,variable=\t] 
		plot({\t},{-1*sqrt(1+\t)});
		\draw[loosely dotted,red,thick,domain = 0:5,smooth,variable=\t] 
		plot({\t},{-0.6*sqrt(1+\t)});
		\draw[-,red,thick,domain = 0:5,smooth,variable=\t] 
		plot({\t},{-0.2*sqrt(1+\t)});
		\draw[-, blue,thick,domain = 0:5,smooth,variable=\t] 
		plot({\t},{0.2*sqrt(1+\t)});
		\draw[loosely dotted,blue,thick,domain = 0:5,smooth,variable=\t] 
		plot({\t},{0.6*sqrt(1+\t)});
		\draw[dashdotted,blue,thick,domain = 0:5,smooth,variable=\t] 
		plot({\t},{1*sqrt(1+\t)});
		\draw[densely dotted, blue,thick,domain = 0:2,smooth,variable=\t] 
		plot({\t},{1.4*sqrt(1+\t)});
		\draw[densely dashed, blue,thick,domain = 0:0.8,smooth,variable=\t] 
		plot({\t},{1.8*sqrt(1+\t)});
		\node[left] at (0,1) {$\frac{1}{5}$};
		\node[left] at (0,-1) {$-\frac{1}{5}$};
		\node[left] at (0,0) {$0$};
		\node[left] at (0,2) {$\frac{2}{5}$};
		\node[left] at (0,-2){$-\frac{2}{5}$};
		\draw[-] (0,-2.5) node[below]{$0$} -- (5.2,-2.5) node[right]{$t$};
		\node[below] at (5,-2.5) {$1$};
		\draw[-] (0,-2.5)  -- (0,2.5) node[above] {$x$};

		\draw[thick,domain = 7:8.9,smooth,variable=\t] 
		plot({\t},{0.2*sqrt(1+(\t-7))});
		\draw[thick] (9,0.345) circle (2.5pt);
		\draw[thick,domain = 9.1:10.9,smooth,variable=\t] 
		plot({\t},{-0.2*sqrt(1+(\t-7)});
		\filldraw[thick] (9.01,-0.345) circle (2.5pt);
		\draw[thick] (11,-0.45) circle (2.5pt);
		\draw[thick,domain = 11.1:12,smooth,variable=\t] 
		plot({\t},{0.2*sqrt(1+(\t-7)});
		\filldraw[thick] (11.01,0.45) circle (2.5pt);
		\draw[dashdotted,thick,domain = 7:9.8,smooth,variable=\t] 
		plot({\t},{-sqrt(1+(\t-7))});
		\draw[thick] (9.9,-1.97) circle (2.5pt);
		\draw[dashdotted,thick,domain = 9.9:12,smooth,variable=\t] 
		plot({\t},{sqrt(1+(\t-7)});
		\filldraw[thick] (9.9,1.97) circle (2.5pt);
		\draw[-,dotted] (7,1) node[left]{$\frac{1}{5}$};
		\draw[-,dotted] (7,-1) node[left]{$-\frac{1}{5}$};
		\draw[-,dotted] (7,0) node[left]{$0$};
		\draw[-,dotted] (7,2) node[left]{$\frac{2}{5}$};
		\draw[-,dotted] (7,-2) node[left]{$-\frac{2}{5}$};
		\draw[-] (7,-2.5) node[below]{$0$} -- (12.2,-2.5) node[right] 
		{$t$};
		\node[below] at (12,-2.5) {$1$};
		\draw[-] (7,-2.5) -- (7,2.5) node[above] {$X_0$};
		\end{tikzpicture}
		
		\caption{On the left is a sketch of $C^{x}_+$ (blue) and $C^{x}_-$ (red) in Example \ref{expl:MCSM}. On the right is a sketch of two typical trajectories under the corresponding shadow martingale measure $\pi _{\text{mc}}$.}
		\label{fig:MCSM}
	\end{figure}
\end{center}

\begin{example} \label{expl:LCSM}
	Let $(\mu _t)_{t \geq 0}$ be defined as	$\mu _t = \mathrm{Unif}_{\left[ -1-t, 1+t \right]}$ for all $t \geq 0$ and let $(\nu ^{\alpha} _{\lc})_{\alpha \in [0,1]}$ be the left-curtain parametrization of $\mu _0$.
	
	Let $\pi _{\lc}$ be the shadow martingale w.r.t.\ $(\mu _t)_{t \geq 0}$ and $(\nu ^{\alpha} _{\lc})_{\alpha \in [0,1]}$. The second part of Proposition \ref{prop:NonObstructedPCOC} shows that the canonical process $(X_t)_{t \geq 0}$ is a martingale that jumps between the two curves
\begin{equation} \label{eq:CurvesLCSM} 
	C_-^{X_0}: \ t \mapsto (-1) - \frac{X_0+1}{2} \cdot t  \hspace{0.5cm} \text{and}   \hspace{0.5cm}
	C_+^{X_0}: \ t \mapsto X_0 + \frac{X_0+1}{2} \cdot t.
	\end{equation}
	Since $\mathrm{Law}_{\pi _{\lc}} (X_0) = \mu _0$ is fixed, \eqref{eq:CurvesLCSM} characterizes $\pi  _{\lc}$ uniquely. 
	
	We can describe $\pi _{\lc}$ in purely stochastic terms: If $U$ is a uniformly random variable on $[-1,1]$, the distribution of the unique c\`adl\`ag martingale $(Y_t)_{t \geq 0}$ that satisfies
	\begin{equation*}
	Y_t \in \begin{cases}
	\{U\} & t = 0 \\
	\left\{-1 - \frac{U+1}{2} \cdot t, U + \frac{U+1}{2} \cdot t \right\} & t > 0
	\end{cases}
	\end{equation*}
	for all $t \geq 0$ is precisely the shadow martingale measure $\pi _{\text{\lc}}$.
\end{example}

As mentioned in Section \ref{sec:RelatedLit}, Henry-Labord\`{e}re, Tan and Touzi \cite{HeTaTo16} and Juillet \cite{Ju18} constructed solutions to the peacock problem close but different from $\pi_{\lc}$. Their principle of construction is also to consider the left-curtain parametrization of $\mu_0$, use partitions and left-curtain couplings. However, the simple Markov concatanation is used in place of the generalized obstructed shadows. Concretely, in the setting of Example \ref{expl:LCSM}, their solution exists and has the following behaviour: Let $(\mu _t)_{t \geq 0}$ be as in Example \ref{expl:LCSM}. There exists a unique limit-curtain measure $\pi \in \MMC_{[0, \infty)}((\mu _{t})_{t \geq 0})$. Under $\pi$ the canonical process consists of trajectories piecewise non-decreasing with jumps down at random times to the bottom $-f(t)$ of the interval.
(cf.\ \cite[Theorem B]{Ju18})
This solution to the peacock problem behaves notably differently from the shadow martingale (cf.\ Figure \ref{fig:LCComparison}). 

\begin{center}
	\begin{figure} 
		\begin{tikzpicture}[scale=0.88]
		
		\draw[purple,thick,domain = 0:5,smooth,variable=\t] 
		plot({\t},{-1 + (-1+1)/10 * \t});
		\draw[loosely dashed, blue,thick,domain = 0:5,smooth,variable=\t] 
		plot({\t},{-0.6 + (-0.6+1)/10 * \t});
		\draw[loosely dotted, blue,thick,domain = 0:5,smooth,variable=\t] 
		plot({\t},{-0.2 + (-0.2+1)/10 * \t});
		\draw[dashdotted, blue,thick,domain = 0:5,smooth,variable=\t] 
		plot({\t},{0.2 + (0.2+1)/10 * \t});
		\draw[densely dotted, blue,thick,domain = 0:5,smooth,variable=\t] 
		plot({\t},{0.6 + (0.6+1)/10 * \t});
		\draw[densely dashed, blue,thick,domain = 0:5,smooth,variable=\t] plot({\t},{1 
			+ (1+1)/10 * \t});
		\draw[loosely dashed, red,thick,domain = 0:5,smooth,variable=\t] plot({\t},{-1 
			- (-0.6+1)/10 * \t});
		\draw[loosely dotted, red,thick,domain = 0:5,smooth,variable=\t] plot({\t},{-1 
			- (-0.2+1)/10 * \t});
		\draw[dashdotted, red,thick,domain = 0:5,smooth,variable=\t] plot({\t},{-1 
			- (0.2+1)/10 * \t});
		\draw[densely dotted, densely dotted, red,thick,domain = 0:5,smooth,variable=\t] plot({\t},{-1 
			- (0.6+1)/10 * \t});
		\draw[densely dashed, red,thick,domain = 0:5,smooth,variable=\t] plot({\t},{-1 
			- (1+1)/10 * \t});
		
		\node[left] at (0,1) {$\frac{1}{5}$};
		\node[left] at (0,-1) {$-\frac{1}{5}$};
		\node[left] at (0,0) {$0$};
		\node[left] at (0,2) {$\frac{2}{5}$};
		\node[left] at (0,-2){$-\frac{2}{5}$};
		
		\draw[-] (0,-2.5) node[below]{$0$} -- (5.2,-2.5) node[right]{$t$};
		\node[below] at (5,-2.5) {$1$};
		\draw[-] (0,-2.5)  -- (0,2.5) node[above] {$x$};

		\draw[loosely dashed,thick,domain = 7:8.9,smooth,variable=\t] 
		plot({\t},{-0.6	+ (-0.6+1)/10 * (\t-7)});
		\draw[thick] (9,-0.52) circle (2.5pt);
		\filldraw[thick] (9.0,-1.08) circle (2.5pt);
		
		\draw[loosely dashed,thick,domain = 9:10.9,smooth,variable=\t] 
		plot({\t},{-1 -	(-0.6+1)/10 * (\t-7)});
		\draw[thick] (11,-1.16) circle (2.5pt);
		\filldraw[thick] (11,-0.44) circle (2.5pt);
		
		\draw[loosely dashed,thick,domain = 11:12,smooth,variable=\t] 
		plot({\t},{-0.6 + (-0.6+1)/10 * (\t-7)});

		\draw[dashdotted,thick,domain = 7:7.9,smooth,variable=\t] 
		plot({\t},{0.6 + (0.6+1)/10 * (\t-7)});
		\draw[thick] (8,0.76) circle (2.5pt);
		\filldraw[thick] (8.0,-1.16) circle (2.5pt);
		
		\draw[dashdotted,thick,domain = 8:9.9,smooth,variable=\t] 
		plot({\t},{-1 - (0.6+1)/10 * (\t-7)});
		\draw[thick] (10,-1.48) circle (2.5pt);
		
		\draw[dashdotted,thick,domain = 10:12,smooth,variable=\t] 
		plot({\t},{0.6 + (0.6+1)/10 * (\t-7)});
		\filldraw[thick] (10,1.075) circle (2.5pt);
		
		\draw[-,dotted] (7,1) node[left]{$1$};
		\draw[-,dotted] (7,-1) node[left]{$-1$};
		\draw[-,dotted] (7,0) node[left]{$0$};
		\draw[-,dotted] (7,2) node[left]{$2$};
		\draw[-,dotted] (7,-2) node[left]{$-2$};
		\draw[-] (7,-2.5) node[below]{$0$} -- (12.2,-2.5) node[right] 
		{$t$};
		\node[below] at (12,-2.5) {$1$};
		\draw[-] (7,-2.5) -- (7,2.5) node[above] {$X_0$};
		
		\end{tikzpicture}
		
		\caption{On the left is a sketch of $C^{x}_+$ (blue) and $C^{x} _-$ (red)  in Example \ref{expl:LCSM}. On the right is a sketch of two typical trajectories under the corresponding shadow martingale measure $\pi _{\lc}$.}
		\label{fig:LCComparison}
	\end{figure}
\end{center}

For Example \ref{expl:MCSM} and Example \ref{expl:LCSM} we could apply the second part of Proposition \ref{prop:NonObstructedPCOC} that corresponds to Corollary \ref{thm:intro1} because the parametrization of the initial marginal was given by a restriction of $\mu _0$ to intervals of $\mathbb{R}$. Any shadow martingale w.r.t.\ a sunrise parametrization does not belong to this class. Therefore, in the following example we need to rely on the notion of martingale parametrization to describe the shadow martingale.

\begin{example} \label{expl:SunSM}
	Let $(\mu _t)_{t \geq 0}$ be defined as $\mu _t = \mathrm{Unif}_{\left[ -e^t, e^t \right]}$ for all $t \geq 0$
	and let $(\nu ^{\alpha} _{\text{sun}})_{\alpha \in [0,1]}$ be the sunset parametrization of $\mu _0$.
	
	Let $\pi _{\sun}$ be the shadow martingale w.r.t.\ $(\mu _t)_{t \geq 0}$ and $(\nu ^{\alpha} _{\sun})_{\alpha \in [0,1]}$ and $(\hat{\pi}^a_{\sun})_{a \in [0,1]}$ the right-derivatives of the corresponding martingale parametrization. Proposition \ref{prop:NonObstructedPCOC} shows that under $\hat{\pi}^a _{\sun}$ the canonical process $(X_t)_{t \geq 0}$ is a martingale that jumps between the two curves
	\begin{equation} \label{eq:CurvesSunSM}
	C_-^{X_0,a}: \ t \mapsto \begin{cases}
	X_0 &t \leq - \ln(a) \\
	-ae^t  &t > - \ln(a)
	\end{cases} \hspace{0.5cm} \text{and}  \hspace{0.5cm}
	C_+^{X_0,a}: \ t \mapsto \begin{cases}
	X_0 &t \leq - \ln(a) \\
	ae^t  &t > - \ln(a)
	\end{cases} 
	\end{equation}
	Since $\mathrm{Law}_{\hat{\pi}^a _{\sun}} (X_0) = \hat{\nu}^a _{\sun} = \mu _0$ is fixed, \eqref{eq:CurvesSunSM} characterizes $\hat{\pi}^a_{\sun}$ and thereby the shadow martingale $\pi _{\sun}$ uniquely. 
	
	We can describe $\pi _{\sun}$ in purely stochastic terms: If $T$ is an exponentially distributed random variable with parameter $1$ and $U$ an independent uniformly distributed random variable on $[-1,1]$, the distribution of the unique c\`adl\`ag martingale $(Y_t)_{t \geq 0}$ that satisfies
	\begin{equation*}
	Y_t \in \begin{cases}
	\{U\} & t < T \\
	\{- e^{t-T}, e^{t -T}\} & t \geq T
	\end{cases}
	\end{equation*}
	for all $t \geq 0$ is precisely the shadow martingale measure $\pi _{\text{\sun}}$.
\end{example}

This example shows that martingale parametrizations and their right-derivatives are essential to describe our solution to the peacock problem. The behaviour of the canonical process under every right-derivative $\hat{\pi}^a$ is easy to understand and the shadow martingale is a simple mixture of these measures (see Figure \ref{fig:SunSM}).
\begin{center}
	\begin{figure} 
		\begin{tikzpicture}[scale=0.88]
		
		\draw[purple,thick,-] (0,-1) -- (3.47,-1);
		\filldraw[red,thick] (3.37,-1) circle (2.5pt);
		\draw[purple,thick,-o] (0,-1+0.4) -- (3.47,-1+0.4);
		\draw[purple,thick,-o] (0,-1+0.8) -- (3.47,-1+0.8);
		\draw[purple,thick,-o] (0,-1+1.2) -- (3.47,-1+1.2);
		\draw[purple,thick,-o] (0,-1+1.6) -- (3.47,-1+1.6);
		\draw[purple,thick,-] (0,-1+2) -- (3.47,1);
		\filldraw[blue,thick] (3.37,1) circle (2.5pt);
		\draw[blue,thick,domain = 3.45:5,smooth,variable=\t] 
		plot({\t},{0.5 * exp(\t/5)});
		\draw[red,thick,domain = 3.45:5,smooth,variable=\t] 
		plot({\t},{-0.5 * exp(\t/5)});
		\draw[-,dotted] (0,1) node[left]{$1$} -- (5,1);
		\draw[-,dotted] (0,-1) node[left]{$-1$} -- (5,-1);
		\draw[-,dotted] (0,0) node[left]{$0$} -- (5,0);
		\draw[-,dotted] (0,2) node[left]{$2$} -- (5,2);
		\draw[-,dotted] (0,-2) node[left]{$-2$} -- (5,-2);
		\draw[-] (0,-2.5) node[below]{$0$} -- (5.2,-2.5) node[right]{$t$};
		\node[below] at (5,-2.5) {$1$};
		\node[below] at (3.45,-2.5) {$\ln(2)$};
		\draw[-] (0,-2.5)  -- (0,2.5) node[above] {$x$};

		\draw[-,thick] (7,-1+0.4) -- (8.9,-1+0.4);
		\draw[thick] (9,-0.6) circle (2.5pt);
		\filldraw[thick] (9.01,1) circle (2.5pt);
		\draw[thick,domain = 9:10.4,smooth,variable=\t] 
		plot({\t},{exp(-0.4)* exp((\t-7)/5)});
		\draw[thick] (10.5,1.36) circle (2.5pt);
		\filldraw[thick] (10.5,-1.36) circle (2.5pt);
		\draw[thick,domain = 10.5:11.4,smooth,variable=\t] 
		plot({\t},{-exp(-0.4)* exp((\t-7)/5)});
		\draw[thick] (11.5,-1.65) circle (2.5pt);
		\draw[thick,domain = 11.5:12,smooth,variable=\t] 
		plot({\t},{exp(-0.4)* exp((\t-7)/5)});
		\filldraw[thick] (11.48,1.61) circle (2.5pt);
		\draw[dashed,-,thick] (7,-1+1.4) -- (11,-1+1.4) node[right]{};
		\draw[thick] (11,0.4) circle (2.5pt);
		\filldraw[thick] (11.01,-1) circle (2.5pt);
		\draw[dashed,thick,domain = 11:12,smooth,variable=\t] 
		plot({\t},{-exp(-0.8)* exp((\t-7)/5)});
		\draw[-,dotted] (7,1) node[left]{$1$} -- (12,1);
		\draw[-,dotted] (7,-1) node[left]{$-1$} -- (12,-1);
		\draw[-,dotted] (7,0) node[left]{$0$} -- (12,0);
		\draw[-,dotted] (7,2) node[left]{$2$}-- (12,2);
		\draw[-,dotted] (7,-2) node[left]{$-2$} -- (12,-2);
		\draw[-] (7,-2.5) node[below]{$0$} -- (12.2,-2.5) node[right] 
		{$t$};
		\node[below] at (12,-2.5) {$1$};
		\draw[-] (7,-2.5) -- (7,2.5) node[above] {$X_0$};
		\end{tikzpicture}		
		\caption{Left: Sketch of $C^{x,a}_+$ (blue) and $C^{x,a} _-(t)$ (red) for $a = \frac{1}{2}$ and $x \in \{-1 + \frac{2i}{5} : 0 \leq i \leq 5 \}$; Right: two typical trajectories under the shadow martingale measure $\pi_{\sun}$ in Example \ref{expl:SunSM} for a random $a$.}
		\label{fig:SunSM}
	\end{figure}
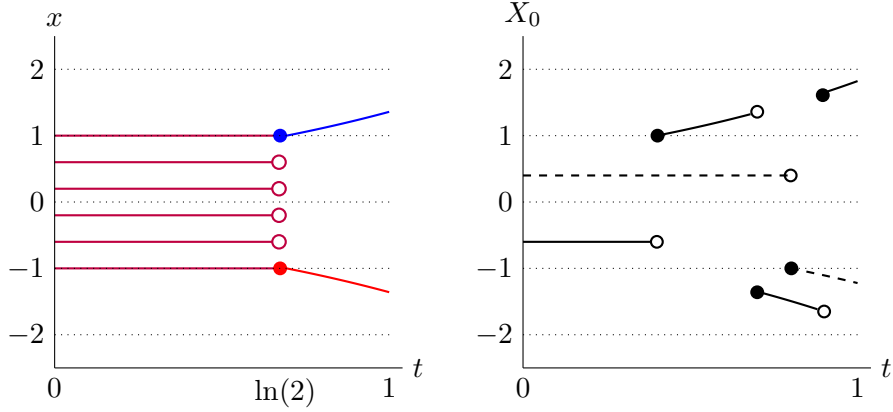
\end{center}

Example \ref{expl:SunSM} was more involved than Example \ref{expl:MCSM} and Example \ref{expl:LCSM} because the parametrization was not given by restrictions to intervals. However, this example still belongs to the special class of non-obstructed peacocks. The last example does not belong to this class. 

\begin{example} \label{expl:SunObstPCOC}
	Let $(S_n)_{n \in \mathbb{N}}$ be the symmetric simple random walk on $\mathbb{Z}$, i.e.\ $S_0=0$ and $S_n := \sum_{i=1} ^n D_i$ for $n\geq 1$ where $(D_i)_{i\geq 1}$ is a sequence of i.i.d.\ random variables uniformly distributed on $\{-1,1\}$. Let the family $(\mu _n)_{n \in \mathbb{N}}$ be defined as 
	\begin{equation*}
	\mu _n := \frac{1}{2} \mathrm{Law}\left( S_n \right) + \frac{1}{2} \mathrm{Law}\left( 3 S_n \right)
	\end{equation*}
	for all $n \in \mathbb{N}$ and let $(\nu ^{\alpha} _{\text{sun}})_{\alpha \in [0,1]}$ be the sunset parametrization of $\mu _0 = \delta _0$, i.e.\ $\nu ^{\alpha} _{\text{sun}} = \alpha \mu _0$. 
	
	By Theorem \ref{thm:MostGenExist} there exists a unique shadow martingale $\pi _{\text{sun}}$ w.r.t.\ $(\mu _n)_{n \in \mathbb{N}}$ and $(\nu ^{\alpha} _{\text{sun}})_{\alpha \in [0,1]}$. 
	Let $(\hat{\pi}^a)_{a \in [0,1]}$ be the family of right-derivatives of  the corresponding martingale parametrization. 
	For $a < \frac{1}{2}$, $\hat{\pi}^a$ is the distribution of $(S_n)_{n \in \mathbb{N}}$ and for $a \geq \frac{1}{2}$, $\hat{\pi}^a$ is the distribution of $(3 S_n)_{n \in \mathbb{N}}$.
	
	We can describe $\pi _{\sun}$ in purely stochastic terms: If $U$ is a uniformly distributed random variable on $[0,1]$ independent from the simple symmetric random walk $(S_n)_{n \in \mathbb{N}}$, the distribution of the  martingale $(Y_n)_{n \in \mathbb{N}}$ defined by
	\begin{equation*}
	Y_n = \begin{cases}
	S_n & \text{if }U < \frac{1}{2} \\
	3S_n & \text{if }U \geq \frac{1}{2}
	\end{cases}
	\end{equation*}
	for all $n \in \mathbb{N}$ is precisely the shadow martingale measure $\pi _{\text{\sun}}$.
\end{example}

It is easy to see that the shadow of the sunset parametrization in the peacock $(\mu _n)_{n \in \mathbb{N}}$ defined as in Example \ref{expl:SunObstPCOC} is not non-obstructed, compare to Figure \ref{fig:ObstrShadow}. Hence, Proposition \ref{prop:NonObstructedPCOC} does not apply and in fact under $\hat{\pi}^a$ the canonical process $(X_t)_{t \in \mathbb{N}}$ does not jump between only two curves.

\section{The left-curtain shadow martingale} \label{sec:LeftMonotone}

Let $(T,\leq)$ be a totally ordered set with minimal element $0\in T$ and $(\mu_t)_{t\in T}$ a peacock  with $\mu_0(\{x\}) = 0$ for all $x \in \mathbb{R}$. Moreover, let  $(\nu ^{\alpha} _{\lc})_{\alpha \in [0,1]}$ be the left-curtain parametrization of $\mu_0$, i.e.\ $\nu _{\lc} ^{\alpha} = {\mu _0}_{|(- \infty, \alpha]}$ for all $\alpha \in [0,1]$ because $\mu _0$ is atomless (see Lemma \ref{lemma:ExConvParam}).
In this section, we give an alternative characterization of the shadow martingale measure w.r.t.\ $(\mu _t)_{t \in T}$  and the left-curtain parametrization $(\nu ^{\alpha} _{\lc})_{\alpha \in [0,1]}$ of $\mu _0$  in terms of  continuous time martingale optimal transport.

More precisely, we prove Theorem \ref{thm:GeneralOptimality} which is the rigorous version of  Theorem \ref{thm:LMAtlOptimal}, 
i.e.\ we show that the shadow martingale is the unique solution to the continuous time martingale optimal transport problem
\begin{equation} \label{eq:ContTimeMOT}
\inf \left\{ \mathbb{E}_{\rho}[c(X_0,X_t)] : \rho \in \MM_{T}(\mu _{T})  \right\}
\end{equation}
simultaneously for all $t \in T$ and for a specific class of cost functions $c$.

\subsection{Optimality}

In the following we denote partial derivatives with the indices of the coordinates, e.g.\ we write $\partial_{122}c$ for $\partial_x\partial^2_y c$.

\begin{definition}
	A function $c: \R^2\to \R$ is said to be a martingale Spence-Mirrlees (MSM) cost function if for every $x < x'$ in $\mathbb{R}$ the increment function
	\begin{equation*}
	\Delta_{x,x'}c:y\mapsto c(x',y)-c(x,y)
	\end{equation*}
	is strictly concave.
\end{definition}

We will principally work with MSM cost functions $c$ that are in $C^{1,2}(\mathbb{R}^2)$ and satisfy the sufficient condition $\partial_{122}c< 0$.
Typical examples are cost functions of the form $c(x,y) = h(y-x)$ where $h'$ is strictly convex, e.g. $c(x,y) = (y-x)^3$, and cost functions of the form $c(x,y) = \varphi (x) \psi (y)$ where $\varphi$ is strictly decreasing and $\psi$ is strictly convex, e.g. $c(x,y)=- x/y$ for $(x,y)\in\R\times (0, \infty)$.  The cost function $c(x,y) = \exp(y-x)$ is in the intersection of the subclasses and $c(x,y)=\sqrt{a(x)+b(x)y^2}$ is a MSM function outside these two subclasses where $a, b$ are non-negative functions for which both $a/\sqrt{b}$ and $b/\sqrt{a}$ are decreasing. The class of MSM cost functions was introduced in \cite{HeTo16} by Henry-Labordère and Touzi and is similar to (but should not be mixed up with) the `non-twisted' condition $\partial_{12}c < 0$ in  classical optimal transport. 

We will show that the left-curtain shadow martingale is the unique simultaneous optimizer of \eqref{eq:ContTimeMOT}. Note that the case of finite $T$  has been worked out with different methods by Beiglb\"ock and Juillet \cite{BeJu21}, Nutz, Stebegg and Tan \cite{NuStTa17} and Beiglb\"ock, Cox, Huesmann \cite{BeCoHu17b}. In fact, the following lemma is very similar to \cite[Lemma 7.14]{NuStTa17}.

\begin{lemma} \label{lemma:MSMRepr}
	Let $c \in C^{1,2}(\mathbb{R}^2)$. For all $M, N \in \mathbb{N}$ and $(x,y) \in (- \infty, M] \times (- \infty, N]$ we have
	\begin{eqnarray*} \label{eq:MSMRepr}
		c(x,y) &=& c(M,y) - (\Delta_{x,M}c)(N) - (y-N)(\Delta_{x,M}c)'(N) \\
		&&+ \int_{-\infty} ^{M} \int_{-\infty} ^{N} \1_{(-\infty,u]}(x)(v-y)^+ \cdot  (-\partial_{122}c(u,v)) \de u \de v.
	\end{eqnarray*}
\end{lemma}

\begin{proof}
	Let $M,N \in \mathbb{N}$ and $f \in C^2(\mathbb{R})$. By partial integration,  for  $y \in (- \infty,N]$ it holds
	\begin{eqnarray*}
		- f(y) &=& - f(N) + \int _{y} ^{N} f'(v) \de v = -f(N) + [(v-y)f'(v)]^{N}_y - \int _y ^{N} (v-y) f''(v) \de v \\
		&=& -f(N) + (N-y) f'(N) - \int _{-\infty} ^{N} (v-y)^+ f''(v) \de v.
	\end{eqnarray*}
The claim follows by applying this to $f = \Delta_{x,M}c$ with $x \leq M$ and rewriting
	\begin{equation*}
	(\Delta_{x,M}c)''(v) = \int _{-\infty} ^{M} \1 _{(- \infty,u]}(x) \partial_{122}c(u,v) \de u. \qedhere
	\end{equation*}
\end{proof}

Lemma \ref{lemma:MSMRepr} shows that - up to some boundary terms - the  interaction of $x$ and $y$ given by   $c$ is basically described by the functions of the form $(x,y) \mapsto \1 _{(- \infty,u]}(x) (v-y)^+$. These, however, are closely connected with the left-curtain shadow martingale.

\begin{lemma} \label{lemma:TildeMSMFct}
	For all $(u,v) \in \mathbb{R}^2$ we define the function $c^{u,v}$ as 
	\begin{equation} \label{eq:DefCTilde}
	c^{u,v}:(x,y)\in \R^2\mapsto \1 _{(- \infty,u]}(x) (v-y)^+\in [0,\infty).
	\end{equation}
	Moreover, let $(\mu _t)_{t \in T}$ be a peacock and $\pi _{\lc}$ the shadow coupling w.r.t.\ $(\mu _t)_{t \in T}$ and the left-curtain parametrization $(\nu ^{\alpha} _{\lc})_{\alpha \in [0,1]}$ of $\mu _0$. Let $\pi\in\MM_T((\mu _t)_{t \in T})$.
	\begin{enumerate}
		\item [(i)]  For every $t \in T$ and $(u,v) \in \mathbb{R}^2$ we have
		\begin{equation*}
		\mathbb{E}_{\pi} \left[c^{u,v}(X_0,X_t) \right] \geq \mathbb{E}_{\pi_{\lc}} \left[ c^{u,v}(X_0,X_t) \right].
		\end{equation*}
		\item  [(ii)] Fix $t\in T$. If for all $(u,v)$ in a dense set of $\mathbb{R}^2$ we have,
		\begin{equation*} \label{eq:SOEqual2}
		\mathbb{E}_{\pi} \left[ c^{u,v}(X_0,X_t) \right] = \mathbb{E}_{\pi_{\lc}} \left[c^{u,v}(X_0,X_t) \right],
		\end{equation*}
		then $(X_0,X_t)$ have the same law under $\pi$ and $\pi _{\lc}$.
		\item [(iii)] Assume $\mu_0$ is atomless. If for every $t\in T$ and all $(u,v)$ in a dense set of $\mathbb{R}^2$ we have
		\begin{equation*} \label{eq:SOEqual3}
		\mathbb{E}_{\pi} \left[c^{u,v}(X_0,X_t) \right] = \mathbb{E}_{\pi_{\lc}} \left[c^{u,v}(X_0,X_t) \right],
		\end{equation*}
		then $\pi$ is the shadow martingale $\pi _{\lc}$.
	\end{enumerate}
\end{lemma}

\begin{proof}
	Item (i): By definition of $\pi_{\lc}$ and Proposition \ref{prop:MaximalElement}, for all $(u,t) \in \mathbb{R}\times T$ we have
	\begin{equation*}
	\pi_{\lc} (X_0 \leq u, X_t \sowh)  = \shadow{\mu _{T_t}}{(\mu _0)_{(- \infty,u]}} \leqc 	\pi( X_0 \leq u, X_t  \sowh).
	\end{equation*}
	Since $y \mapsto (v-y)^+$ is a convex function for all $v \in \mathbb{R}$ we get the desired inequality.
	
	Item (ii): Let $f_1,f_2 : \mathbb{R}^2 \rightarrow \mathbb{R}$ be defined as 
	\begin{eqnarray*}
		f_1(u,v) := \mathbb{E}_{\pi} \left[  c^{u,v}(X_0,X_t) \right] \hspace{1cm} \text{and} \hspace{1cm}
		f_2(u,v) :=\mathbb{E}_{\pi_{\lc}} \left[  c^{u,v}(X_0,X_t) \right].
	\end{eqnarray*}
	Since $f_1(u,\cdot)$ and $f_2(u,\cdot)$ are the potential functions of $\pi_{\lc} (X_0 \leq u, X_t \sowh)$ and $\pi (X_0 \leq u, X_t \sowh)$  and potential functions are convex, $u_1$ and $u_2$ are continuous in $v$.
	Moreover, for all $(u,v) \in \mathbb{R}^2$ and $i \in \{1,2\}$, dominated convergence yields $\lim _{u_n \downarrow u} f_i (u_n,v) = f_i (u,v)$
	because $\mu _t$ has a finite first moment.
	Hence, since $f_1$ and $f_2$ coincide on a dense set, they are equal everywhere. In particular, by Lemma \ref{lemma:characPotF} we obtain for all $u \in \mathbb{R}$
	\begin{equation*}
	\pi_{\lc} (X_0 \leq u, X_t \sowh) = \pi (X_0 \leq u, X_t \sowh).
	\end{equation*}
	
	Item (iii): Since $\mu _0(\{x\}) = 0$ for all $x \in \mathbb{R}$, for all $\alpha \in [0,1]$ there exists $q _{\alpha} \in \mathbb{R}$ with $\nu ^{\alpha} _{\lc} = (\mu_0)_{|(- \infty,q_{\alpha}]}$. Hence, $(\pi ^{\alpha})_{\alpha \in [0,1]}$ with $\pi^{\alpha} := \alpha \mathrm{Law}_{\pi} (X | X_0 \leq q_{\alpha})$
	is a parametrization of $\pi$ w.r.t.\ $(\nu ^{\alpha} _{\lc})_{\alpha \in [0,1]}$ and by (ii) (and Lemma \ref{rem:uniqueParametrization}) we have
	\begin{equation*}
	\pi^{\alpha}(X_t  \sowh) = \alpha \mathrm{Law}_{\pi_{\lc}} (X | X_0 \leq q_{\alpha}) = \shadow{\mu _{T_t}}{\nu ^{\alpha}_{\lc}}
	\end{equation*} 
	for all $\alpha \in [0,1]$ and $t \in T$. By uniqueness of the shadow martingale, $\pi = \pi_{\lc}$.
\end{proof}

\begin{theorem} \label{thm:GeneralOptimality}
	Let $(\mu _t)_{t \in T}$ be a peacock and  $c: \mathbb{R}^2 \rightarrow \mathbb{R}$ a MSM cost function. Suppose that 
	\begin{enumerate}
		\item [(i)] $c \in C^{1,2}(\mathbb{R}^2)$ with $\partial_{122}c<0$ and
		\item [(ii)] for all $t \in T$ there exists $\varphi_1 \in L^1(\mu _0)$ and $\psi_1 \in L^1(\mu _t)$ such that for all $(x,y) \in \mathbb{R}^2$ we have $|c(x,y)| \leq \varphi_1(x) + \psi_1(y)$.
	\end{enumerate}
	Moreover, we assume that at least one of the two following assumptions is satisfied:
	\begin{enumerate}
		\item [(iii-a)] For all $t \in T$ there exists $\varphi_2 \in L^1(\mu _0)$ and $\psi_2 \in L^1(\mu _t)$ such that for all $(x,y) \in \mathbb{R}^2$ we have $|\partial_2 c(x,0)y| \leq \varphi_2(x) + \psi_2(y)$ or
		\item [(iii-b)] there exists $\overline{M} \in \mathbb{N}$ such that $\mathrm{supp}(\mu _0) \subset (-\infty, \bar{M}]$. 
	\end{enumerate}
	Under these assumptions the shadow martingale $\pi_\lc$ w.r.t.\ $(\mu _t)_{t \in T}$ and the left-curtain parametrization $(\nu ^{\alpha} _{\lc})_{\alpha \in [0,1]}$ of $\mu _0$  satisfies 
	\begin{equation*}
	\mathbb{E}_{\pi _{\lc}}[c(X_0,X_t)] = \inf \left\{ \mathbb{E}_{\rho}[c(X_0,X_t)] : \rho \in \MM_T((\mu _t)_{t \in T}) \right\}
	\end{equation*}
	simultaneously for all $t \in T$. 
	
	If $\mu _0$ is atomless, $\pi_{\lc}$ is the unique element of $\MM_T((\mu _t)_{t \in T})$ with this property.
\end{theorem}

We first concentrate on the central argument of the proof and postpone the proofs of two technical results Lemma \ref{lemma:OptimalTechn} and Lemma \ref{lemma:OptimalTechnSecond} to the next subsection.

\begin{proof}
	Fix $t \in T$ and for $\pi_\lc$ a competitor $\pi \in \MM_T((\mu _t)_{t \in T})$. 
	For all $m,M,N \in \mathbb{N}$ and $(x,y) \in \mathbb{R}^2$ we define
	\begin{eqnarray*}
		A_{m,M,N} &:=& [-m, M] \times (- \infty, N] \hspace{1cm} \text{ and} \\
		R_{M, N}:(x,y)\in \R^2 &\mapsto& c(M,y)-[\Delta_{x,M}c(N)+(\Delta_{x,M} c)'(N)\cdot (y-N)]
	\end{eqnarray*}
	where we recall that $\Delta_{x,M}c$ denotes the increment function $y\mapsto c(M,y)-c(x,y)$. In the following we use the notation $\mathbb{E}_{\pi - \pi _{\lc}} [g(X)]$ for $\mathbb{E}_{\pi}[g(X)] - \mathbb{E}_{\pi _{\lc}}[g(X)]$.
	By (i) and Lemma \ref{lemma:MSMRepr}, we know that for all $m,M,N \in \mathbb{R}$ and $(x,y) \in A_{m,M,N}$ we have
	\begin{align}\label{eq:basis}
	c(x,y)&=R_{M,N}(x,y)+\int_{-\infty} ^M \int_{-\infty}^{N} c^{u,v}(x,y)  (-\partial_{122}c(u,v)) \de u \de v
	\end{align}
	where ${c}^{u,v}(x,y) = \1_{(-\infty,u]}(x)(v-y)_+$ as in \eqref{eq:DefCTilde}.  
	On the one hand, by assumption (ii) $c(X_0,X_t)$ is integrable w.r.t.\ $\pi$ and $\pi _{\lc}$, and dominated convergence yields
	\begin{equation*}
	\mathbb{E}_{\pi - \pi _{\lc}}[c(X_0,X_t)] = \lim _{M \rightarrow \infty} \lim _{m \rightarrow \infty} \lim _{N \rightarrow \infty} \mathbb{E}_{\pi - \pi_{\lc}} \left[ \1_{A_{m,M,N}}(X_0,X_t) c(X_0,X_t)\right].
	\end{equation*} 
	On the other hand, for all $m,M,N \in \mathbb{N}$ the random variable $ \1_{m,M,N}(X_0,X_t) R_{M,N}(X_0,X_t)$
	is integrable w.r.t.\ $\pi$ and $\pi _{\lc}$ because the function $x\mapsto\partial_2 c(x,N)$ is continuous on $[- m, M]$ and $\mu _t$ has a finite first moment. We show in Lemma \ref{lemma:OptimalTechn} that under assumptions (i) and (ii) the following successive limits exist and satisfy
	\begin{eqnarray*}
		 \lim _{m \rightarrow \infty} \lim _{N \rightarrow \infty} \mathbb{E_{\pi - \pi_{\lc}}} \left[ \1_{A_{m,M,N}}(X_0,X_t) R_{m,M,N}(X_0,X_t) \right]= \mathbb{E_{\pi - \pi_{\lc}}} \left[ \1_{(- \infty,M]}(X_0) c(M,X_t) \right].
	\end{eqnarray*}
	Since assumption (iii-a) or assumption (iii-b) is satisfied, by Lemma \ref{lemma:OptimalTechnSecond} we have
	\begin{equation*}
	\lim _{M \rightarrow \infty}  \mathbb{E_{\pi - \pi_{\lc}}} \left[ \1_{(- \infty,M]}(X_0) c(M,X_t) \right] = 0.
	\end{equation*}
	Hence, taking the expectation in \eqref{eq:basis}, by Fubini's theorem we have
	\begin{eqnarray*}
		\mathbb{E}_{\pi - \pi _{\lc}} \left[ c(X_0,X_t)\right]
		=\lim _{M \rightarrow \infty} \lim _{m \rightarrow \infty} \lim _{N \rightarrow \infty} \int _m ^M \int_{-\infty} ^N \mathbb{E}_{\pi - \pi _{\lc}} \left[c^{u,v}(X_0,X_t) \right] (-\partial_{122}c(u,v)) \de u \de v.
	\end{eqnarray*}
	Since $\partial_{122}c < 0$ by assumption (i) and $\mathbb{E}_{\pi - \pi _\lc} \left[c^{u,v}(X_0,X_t) \right] \geq 0$ by Lemma \ref{lemma:TildeMSMFct} (i), by monotone convergence we obtain
	\begin{equation*}
	\mathbb{E}_{\pi - \pi _{\lc}} \left[ c(X_0,X_t)\right] =  \int _{-\infty} ^{\infty}  \int _{-\infty} ^{\infty} \mathbb{E}_{\pi - \pi _{\lc}} \left[c^{u,v}(X_0,X_t) \right] (-\partial_{122}c(u,v)) \de u \de v.
	\end{equation*}
	The claim follows with Lemma \ref{lemma:TildeMSMFct} (i) and Lemma \ref{lemma:TildeMSMFct} (iii).
\end{proof}

\begin{remark} \label{rem:OpimalityCostFuctions}
	In the following cases the assumptions of Theorem \ref{thm:GeneralOptimality} are satisfied:
	\begin{itemize}
		\item[(i)] The cost function $c:(x,y)\mapsto \mathrm{tanh}(-x)\sqrt{1+y^2}$ satisfies (i), (ii) and (iii-a) for every peacock and every $t\in T$. 
		\item [(ii)] Let $c(x,y) := (y-x)^3$ and suppose $\mu _t$ has a finite third moment for all $t \in T$. Then we can satisfy assumptions (ii) and (iii-a) with the same functions $x\mapsto 4|x|^3$ and $y\mapsto 4|y^3|$ 
		\item [(iii)] Let $\varphi \in C^1(\mathbb{R})$ with $\varphi' > 0$, $\psi \in C^2(\mathbb{R})$ with $\psi '' < 0$,   $c(x,y) := \varphi(x) \psi(y)$  and suppose there exist  $\frac{1}{p} + \frac{1}{q} = 1$ such that for all $t \in T$  $\varphi \in L^p(\mu _0)$ and $\psi \in L^q(\mu _t)$ and either $\mu_t$ has finite $q$-th moment for all $t\in T$ or there exists $\overline{M} \in \mathbb{N}$ with $\mathrm{supp}(\mu _0) \subset (- \infty,\overline{M}]$. 
	\end{itemize}
\end{remark}

\begin{remark} \label{rem:LMApprox}
	As mentioned before, Theorem \ref{thm:GeneralOptimality} for a finite index set $T$ is proven in \cite{NuStTa17} and \cite{BeCoHu17}.
	For a general index set $T$, picking a suitable sequence of nested finite subsets $(R_n)_n$ of $T$, it is possible to show that the sequence of shadow martingales w.r.t.\ $(\mu_t)_{t\in R_n}$ and $(\nu ^{\alpha} _{\lc})_{\alpha \in [0,1]}$ converges to the unique shadow martingale w.r.t.\ $(\mu _t)_{t \in T}$ and $(\nu ^{\alpha} _{\lc})_{\alpha \in [0,1]}$.  In particular, the unique optimizer of the finite time martingale optimal transport problem provided by \cite{NuStTa17} converge to the unique optimizer of the corresponding continuous-time martingale optimal transport problem. 
\end{remark}

\subsection{Pending proofs}

Recall the notation $A_{m,M,N}$, $R_{M, N}$  and $\mathbb{E}_{\pi - \pi _{\lc}}$ from the proof of Theorem \ref{thm:GeneralOptimality}.
Given a MSM cost function $c$, for all $x < x'$ in $\mathbb{R}$ and $N \in \mathbb{N}$ we denote by $L_{x,x'} ^{N}$ the tangent of the concave increment function $\Delta _{x,x'}c$ at $N$, i.e.\ $L_{x,x'} ^N(y) := (\Delta _{x,x'}c)(N) + (\Delta _{x,x'}c)'(N)(y-N)$ for all $y \in \mathbb{R}$.

\begin{lemma} \label{lemma:OptimalTechn}
	Let $(\mu _t)_{t \in T}$ be a peacock and  $c: \mathbb{R}^2 \rightarrow \mathbb{R}$ a MSM cost function.
	Under the assumptions (i) and (ii) of Theorem \ref{thm:GeneralOptimality}, for all $t \in T$ and $M \in \mathbb{N}$ the following successive limit exists and satisfies 
	\begin{eqnarray*}
        \lim _{m \rightarrow \infty} \lim _{N \rightarrow \infty} \mathbb{E_{\pi - \pi_{\lc}}} \left[ \1_{A_{m,M,N}}(X_0,X_t) R_{M,N}(X_0,X_t) \right]  =\mathbb{E_{\pi - \pi_{\lc}}} \left[ \1_{(- \infty,M]}(X_0) c(M,X_t) \right].
	\end{eqnarray*}
\end{lemma}
\begin{proof}
	Fix $\pi \in \MM_T((\mu _t)_{t \in T})$ and $t \in T$. 
	
	\textsf{STEP 1:} For all $m,M, N \in \mathbb{N}$ and $x \in [-m,M]$, the tangent function $L_{m,M}^N$ of $c$ is an affine function. Hence, the martingale property yields that 
	\begin{eqnarray*}
		&& \mathbb E_\pi \left[ \1 _{[- m, M]}(X_0) \left( c(M, X_t) - R_{M,N}(X_0,X_t) \right)\right] \\
		&=&\mathbb E_\pi  \left[ \1 _{[- m,M]}(X_0) L_{X_0,M}^N(X_t) \right] =\mathbb E_\pi  \left[ \1 _{[- m,M]}(X_0) L_{X_0,M}^N(X_0) \right]
	\end{eqnarray*}
	is independent of $\pi\in\MM _T((\mu _t)_{t \in T})$.  Thus,
	\begin{equation*}
	\mathbb{E}_{\pi - \pi _{\lc}} \left[ \1_{[- m, M]}(X_0) R_{M,N}(X_0,X_t) \right] = \mathbb{E}_{\pi - \pi _{\lc}} \left[  \1 _{[- m, M]}(X_0) c(M,X_t)  \right].
	\end{equation*}
	
	\textsf{STEP 2:} Fix $M \in \mathbb{N}$ and $x \in [- m, M]$. For all $N \in \mathbb{N}$ we define $ B_{m,M,N} = [-m,M] \times (N, \infty)$
	so that $\1_{[-m,M]}(x)=\1_{A_{m,M,N}}(x,y)+\1_{B_{m,M,N}}(x,y)$ for all $y \in \mathbb{R}$.
	The function
	\begin{equation*}
	y \mapsto c(M,y) - R_{N,M} (x,y)= L_{x,M} ^N(y)
	\end{equation*}
	is  the tangent of the concave function $\Delta _{x,M}c$ at the position $N$. By concavity, for every $y\in (N,+\infty)$ (where we recall $N\geq0$) we have $\Delta_{x,M}c (y) \leq L_{x,M}^{N} (y) \leq L_{x,M} ^0(y).$ 
	Hence,
	\begin{eqnarray*} 
		&&\1 _{B_{m,M,N}}(X_0,X_t)| R_{M,N}(X_0,X_t)| \\ 
		&\leq& |c(M,X_t)| + |\Delta_{X_0,M}c(X_t)| + \1 _{[-m,M]}(X_0) |L_{X_0,M}^{0}(X_t)|.
	\end{eqnarray*}
	Note that the right-hand side is independent of $N$  and integrable with respect to $\pi$ and $\pi _{\lc}$.
	
	\textsf{STEP 3:} With the majorant from Step 2, the dominated convergence theorem yields
	\begin{equation*}
	\lim _{N \rightarrow \infty} \mathbb{E}_{\pi - \pi _{\lc}} \left[ \1 _{B_{m,M,N}}(X_0,X_t) R_{m,M,N}(X_0,X_t)\right] = 0.
	\end{equation*}
	 Therefore, we have
	\begin{eqnarray*}
		&&\lim _{N \rightarrow \infty} \mathbb{E}_{\pi - \pi _{\lc}} \left[ \1 _{A_{m,M,N}}(X_0,X_t) R_{M,N}(X_0,X_t)\right] \\
		&=& \lim _{N \rightarrow \infty}\left( \mathbb{E}_{\pi - \pi _{\lc}} \left[ \1 _{[- m, M]}(X_0) R_{M,N}(X_0,X_t)\right] - \mathbb{E}_{\pi - \pi _{\lc}} \left[ \1 _{B_{m,M,N}}(X_0,X_t) R_{M,N}(X_0,X_t)\right]\right) \\
		&=& \mathbb{E}_{\pi - \pi _{\lc}} \left[  \1 _{[- m, M]}(X_0) c(M,X_t)  \right].
	\end{eqnarray*}
	
	\textsf{STEP 4:} For all $M \in \mathbb{N}$ assumption (ii) ensures that $\1 _{(- \infty,M]}(X_0)c(M,X_t)$ is integrable w.r.t.\ $\pi$ and $\pi _{\lc}$. Hence, the claim follows with dominated convergence.
\end{proof}

\begin{lemma} \label{lemma:OptimalTechnSecond}
	Let $(\mu _t)_{t \in T}$ be a peacock and  $c: \mathbb{R}^2 \rightarrow \mathbb{R}$ a MSM cost function that satisfies assumptions (i) and (ii) of Theorem \ref{thm:GeneralOptimality}. If additionally assumption (iii-a) or assumption (iii-b) of Theorem \ref{thm:GeneralOptimality} are satisfied,  for all $t \in T$ we have 
	\begin{equation*}
	\lim _{M \rightarrow \infty} \mathbb{E_{\pi - \pi_{\lc}}} \left[ \1_{(- \infty,M]}(X_0) c(M,X_t) \right] =0.
	\end{equation*}
\end{lemma}

\begin{proof}
	Fix $\pi \in \MM_T((\mu _t)_{t \in T})$ and $t \in T$.
	
	If assumption (iii-b) is satisfied, $\1 _{(- \infty,\overline{M}]} = 1$ $\pi$-a.e. Hence, for all $M \geq \overline{M}$ we have 
	\begin{equation*}
	\mathbb{E_{\pi - \pi_{\lc}}} \left[ \1_{(- \infty,M]}(X_0) c(M,X_t) \right] = \mathbb{E_{\pi - \pi_{\lc}}} \left[ c(M,X_t) \right] =  0
	\end{equation*} 
	and the claim follows.
	
	Now suppose assumption (iii-a) is satisfied. In this case, w.l.o.g.\ we can assume that 
	\begin{equation} \label{eq:Normalization}
	c(0,y) = c(x,0) =\partial_2 c(x,0) = 0, \hspace*{1cm} \text{for all } (x,y) \in \mathbb{R}^2.
	\end{equation}
	Otherwise, we replace $c$ with $\tilde{c}$ defined as
	\begin{eqnarray*}
		\tilde{c} (x,y) &:=& \Delta_{0,x}c(y) - (\Delta_{0,x}c)(0) - (\Delta_{0,x}c)'(0) y \\
		&=& c(x,y) - c(0,y) - (\Delta_{0,x}c)(0) - (\Delta_{0,x}c)'(0) y.
	\end{eqnarray*}
	Indeed, $\tilde{c}$ satisfies assumptions (i), (ii), and (iii-a) or (iii-b) (depending on $c$)  and $\tilde c(0,y) = \tilde c(x,0) =  \partial_2 \tilde c(x,0) = 0$ for all $(x,y) \in \mathbb{R}^2$. Moreover, by the martingale property  we have
	\begin{eqnarray*}
		&& \lim _{M \rightarrow \infty} \mathbb{E}_{\pi - \pi _{\lc}} \left[ \1_{(- \infty,M]}(X_0) \left( \tilde{c}(M,X_t) - c(M,X_t) \right) \right] \\
		&=& \lim _{M \rightarrow \infty} \mathbb{E}_{\pi - \pi _{\lc}} \left[ \1_{(- \infty,M]}(X_0) \left( c(0,X_t) + L_{0,M}^0(X_0) \right) \right] = \lim _{M \rightarrow \infty} \mathbb{E}_{\pi - \pi _{\lc}} \left[  c(0,X_t) \right] =  0.
	\end{eqnarray*}
	Hence the limit of $\mathbb{E_{\pi - \pi_{\lc}}} \left[ \1_{(- \infty,M]}(X_0) c(M,X_t) \right]$ exists and vanishes as $M$ tends to infinity if and only if the limit of  $\mathbb{E_{\pi - \pi_{\lc}}} \left[ \1_{(- \infty,M]}(X_0) \tilde c(M,X_t) \right]$ exists and vanishes as $M\to \infty$.
	
	For all $0 \leq  x \leq x'$, the MSM property and \eqref{eq:Normalization} yield that $\Delta_{x,x'}c$ is a concave function with $(\Delta_{x,x'}c)(0) = (\Delta_{x, x'}c)'(0) = 0$ and thus $\Delta_{x,x'}c \leq 0$. Hence,  for all $0 \leq M \leq M'$ and $(x,y) \in \mathbb{R}^2$ we obtain
	\begin{align*}
		&\1_{(M, \infty)}(x) \Delta _{M,x}c(y) = \1_{(M, M']}(x) \Delta _{M,x}c(y) + \1_{(M', \infty)}(x) \Delta _{M,x}c(y) \\
		\leq& \1_{(M', \infty)}(x) \Delta _{M,x}c(y) =  \1_{(M', \infty)}(x) \left( \Delta_{M,M'} c(y) + \Delta_{M',x}c(y) \right) 
		\leq  \1_{(M', \infty)}(x) \Delta _{M',x}c(y).
	\end{align*}
	Since for all $(x,y) \in \mathbb{R}^2$ the sequence $\1_{(M, \infty)}(x) \Delta _{M,x}c(y)$ converges to $0$ as $M$ tends to infinity,
	monotone convergence applied separately for $\pi$ and $\pi_{\lc}$ yields 
	\begin{equation*}
	\lim_{M \rightarrow \infty} \mathbb{E}_{\pi - \pi _{\lc}} \left[  \1_{(M, \infty)}(X_0) \Delta _{M,X_0}c(X_t)  \right] = 0.
	\end{equation*}
	Finally, we can conclude 
	\begin{eqnarray*}
		&&	\lim _{M \rightarrow \infty} \mathbb{E}_{\pi - \pi _{\lc}} \left[  \1 _{(- \infty, M]}(X_0) c(M,X_t)  \right] \\
		&=& 	\lim_{M \rightarrow \infty} \left(  \mathbb{E}_{\pi - \pi _{\lc}} \left[ c(M,X_t)  \right]  -  \mathbb{E}_{\pi - \pi _{\lc}} \left[  \1 _{(M,\infty)}(X_0) c(M,X_t)  \right]  \right) \\
		&=&  - \lim _{M \rightarrow \infty} \mathbb{E}_{\pi - \pi _{\lc}} \left[  \1 _{(M,\infty)}(X_0) c(M,X_t)  \right] \\
		&=& \lim _{M \rightarrow \infty} \mathbb{E}_{\pi - \pi _{\lc}} \left[  \1 _{(M,\infty)}(X_0) \Delta _{M,X_0}c(X_t)  \right] -  \lim _{M \rightarrow \infty} \mathbb{E}_{\pi - \pi _{\lc}} \left[  \1 _{(M,\infty)}(X_0) c(X_0,X_t)  \right] =0
	\end{eqnarray*}
	using dominated convergence in the last equality employing assumption (ii).
\end{proof}

\section*{Acknowledgements}

We thank the referees for a careful reading of the manuscript and several hints to improve the presentation of our results. We thank Mathias Beiglb\"ock for stimulating discussions at the outset of this project and Gudmund Pammer for pointing out a mistake in an earlier version of this paper.

\normalem

\end{document}